\RequirePackage{fix-cm}
\RequirePackage{amsmath}
\documentclass[smallextended]{svjour3} 
\smartqed  
\usepackage{graphicx}

\usepackage{mathptmx}      

\usepackage[authoryear]{natbib}

\usepackage{comment}
\usepackage{color}
\usepackage[font={footnotesize}]{caption}
\usepackage{ulem}
\usepackage{bbm}
\usepackage{latexsym}
\usepackage{float}
\usepackage{amssymb}

\usepackage{paralist}
\usepackage{url}
\usepackage{apptools}
\usepackage{float}
\usepackage[caption=false]{subfig}
\usepackage[colorlinks=true,citecolor=blue,urlcolor=blue,linkcolor=blue]{hyperref}

\usepackage[english]{babel}

\numberwithin{equation}{section}

\newtheorem{assumption}{Assumption}
\newtheorem{fig}{Figure}

\AtAppendix{\numberwithin{lemma}{section}}

\newcommand{\bthe}{\begin{theorem}}
\newcommand{\ethe}{\end{theorem}}

\newcommand{\ben}{\begin{enumerate}}
\newcommand{\een}{\end{enumerate}}

\newcommand{\bit}{\begin{itemize}}
\newcommand{\eit}{\end{itemize}}

\newcommand{\ex}{\textnormal{e}^}

\newcommand{\beq}{\begin{equation}}
\newcommand{\eeq}{\end{equation}}

\newcommand{\ble}{\begin{lemma}}
\newcommand{\ele}{\end{lemma}}

\newcommand{\bde}{\begin{definition}\rm}
\newcommand{\ede}{\halmos\end{definition}}

\newcommand{\bco}{\begin{corollary}}
\newcommand{\eco}{\end{corollary}}

\newcommand{\bpr}{\begin{proposition}}
\newcommand{\epr}{\end{proposition}}

\newcommand{\brem}{\begin{remark}\rm}
\newcommand{\erem}{\end{remark}}

\newcommand{\bproof}{\begin{proof}}
\newcommand{\eproof}{\hfill\(\square\)\end{proof}}

\newcommand{\bexam}{\begin{example}\rm}
\newcommand{\eexam}{\end{example}}

\newcommand{\bfi}{\begin{fig}}
\newcommand{\efi}{\end{fig}}

\newcommand{\btab}{\begin{tab}}
\newcommand{\etab}{\end{tab}}

\newcommand{\beao}{\begin{eqnarray*}}
\newcommand{\eeao}{\end{eqnarray*}\noindent}

\newcommand{\beam}{\begin{eqnarray}}
\newcommand{\eeam}{\end{eqnarray}\noindent}

\newcommand{\barr}{\begin{array}}
\newcommand{\earr}{\end{array}}

\newcommand{\bdis}{\begin{displaymath}}
\newcommand{\edis}{\end{displaymath}\noindent}

\newcommand{\R}{\mathbb R}
\newcommand{\N}{\mathbb N}

\newcommand{\Z}{\mathbb Z}

\newcommand{\E}{\mathbb E}

\newcommand{\cals}{\mathcal{S}}

\newcommand{\calf}{\mathcal{F}}

\newcommand{\stp}{\stackrel{P}{\rightarrow}}
\newcommand{\std}{\stackrel{\mathcal{D}}{\rightarrow}}
\newcommand{\stas}{\stackrel{\rm a.s.}{\rightarrow}}

\newcommand{\stv}{\stackrel{v}{\rightarrow}}

\newcommand{\nto}{n\to\infty}

\newcommand{\ga}{{\gamma}}

\newcommand{\si}{{\sigma}}

\newcommand{\var}{{\mathbb{V}ar}}
\newcommand{\cov}{{\mathbb{C}ov}}

\newcommand{\ov}{\overline}
\newcommand{\wh}{\widehat}
\newcommand{\wt}{\widetilde}

\newcommand{\bs}{\pmb}

\DeclareMathOperator*{\argmin}{arg\,min}
\DeclareMathOperator*{\diag}{diag}

\newcommand{\bbn}{\mathbb{N}}
\newcommand{\bbz}{\mathbb{Z}}
\newcommand{\bbr}{\mathbb{R}}
\newcommand{\bbp}{\mathbb{P}}

\newcommand{\one}{\mathbbmss{1}_}

\newcommand{\Rho}{\mathrm{P}}

\newcommand{\kto}{k\to\infty}

\def\cald{{\mathcal{D}}}
\def\calf{{\mathcal{F}}}
\def\cali{{\mathcal{I}}}
\def\cals{{\mathcal{S}}}

\def\calh{{\mathcal{H}}}

\newcommand{\halmos}{\quad\hfill\mbox{$\Box$}}  

\newcommand{\trans}{{}^{^{\intercal}}}

\newcommand{\GLSE}{{\rm GLSE}}

\allowdisplaybreaks[4]

\begin{document}

\title{Generalised least squares estimation of regularly varying space-time processes based on flexible observation schemes}

\titlerunning{Generalised LSE of regularly varying space-time processes}

\author{Sven Buhl         \and
        Claudia Kl{\"u}ppelberg 
}

\institute{S. Buhl \at
              Center for Mathematical Sciences, Technical University of Munich,  85748 Garching, Boltzmannstr. 3, Germany\\
              \email{sven.buhl@tum.de}           
           \and
           C. Kl{\"u}ppelberg \at
              Center for Mathematical Sciences, Technical University of Munich,  85748 Garching, Boltzmannstr. 3, Germany\\
              \email{cklu@tum.de}
}

\date{Received: date / Accepted: date}

\maketitle

\begin{abstract}
Regularly varying stochastic processes model extreme dependence between process values at different locations and/or time points. 
For such stationary processes we propose a two-step parameter estimation of the extremogram, when some part of the domain of interest is fixed and another increasing.
We provide conditions for consistency and asymptotic normality of the empirical extremogram centred by a pre-asymptotic version for such observation schemes.
For max-stable processes with Fr{\'e}chet margins we provide conditions, such that the empirical extremogram (or a bias-corrected version) centred by its true version is asymptotically normal.
In a second step, for a parametric extremogram model, we fit the parameters by generalised least squares estimation and prove consistency and asymptotic normality of the estimates.
We propose subsampling procedures to obtain asymptotically correct confidence intervals.
Finally, we apply our results to a variety of Brown-Resnick processes.
A simulation study shows that the procedure works well also for moderate sample sizes.

\end{abstract}

\section{Introduction}\label{s1}

Max-stable processes and regularly varying processes have in recent years attracted attention as time series models, spatial processes and space-time processes.
Regularly varying processes have been investigated in \citet{HL2005,HL} and basic results for max-stable processes can be found in \citet{deHaan5}.
Such processes provide a useful framework for modelling and estimation of extremal events in their different settings. 

Among the various regularly varying models considered in the literature, max-stable Brown-Resnick processes play a prominent role allowing for flexible fractional variogram models as often observed in environmental data.
They have been introduced for time series in \citet{Brown}, for spatial processes in \citet{Schlather2}, and in a space-time setting in \citet{Steinkohl}. 

For max-stable processes with parametrised dependence structure, various estimation procedures have been proposed for extremal data. 
Composite likelihood methods have been described in \citet{Ribatet} and \citet{Huser}. 
Threshold-based likelihood methods have been proposed in \citet{Wadsworth2} and \citet{Engelke1}.
For the max-stable Brown-Resnick process asymptotic results of composite likelihood estimators have been derived in  \cite{Steinkohl2}, \citet{Huser2}, and \citet{buhl1}. 
In some special cases full likelihood estimation is feasible, which opens the door for frequentist or Bayesian approaches; see for example \citet{Dombryetal,Thibaud1}.

Parameter estimation based on likelihood methods can be laborious and time consuming, and also the choice of good initial values for the optimization routine is essential. 
As a consequence, a semiparametric estimation procedure can be an alternative or a prerequisite for a subsequent likelihood method.
Such an estimation method has been suggested and analysed for  space-time processes with additively separable dependence function in \citet{steinkohlphd} and \citet{Steinkohl3} based on the extremogram, which is a natural extremal analogue of the correlation function for stationary processes. 
The extremogram was introduced for time series in \citet{Davis2} and \citet{Fasen},  and extended to a spatial and space-time settings in \cite{steinkohlphd} and \citet{cho}. 
Semiparametric estimation requires a parametric extremogram model.
The parameter estimation is then based on the empirical extremogram, and a subsequent least squares estimation of the parameters. 

The processes considered in \cite{steinkohlphd}, \cite{cho}, \cite{buhl3}, and \cite{Steinkohl3} are isotropic {in space}; cf. model (I) in Section~\ref{s32} below. 
The central goal of this paper is to generalise the semiparametric method developed {in  \cite{Steinkohl3} in various aspects.}
{We list the most important extensions:}\\
-- {In \cite{Steinkohl3} ordinary least squares estimation was performed separately for the spatial and the temporal dependence parameters.}
This was possible, since we assumed an additively separable dependence model, linear in its parameters after a suitable transformation. In the present paper we allow for a much larger class of dependence models provided they satisfy some weak regularity conditions.
In particular, we allow for non-linear structures in the dependence models, and
we estimate a space-time dependence model, which is not necessarily separable.\\
-- To fit these general models to data, we develop a generalised (weighted) least squares estimation method, which estimates all dependence parameters in one go.\\
-- We again focus on extremogram estimation, but extend the observation scheme as described below.  
In the context of spatial or space-time extremogram estimation {based on gridded data}, the observation scheme used so far in \cite{steinkohlphd}, \cite{cho}, \cite{buhl3}, and \cite{Steinkohl3} has been a regular grid in space, possibly observed at equidistant time points and assumed to expand to infinity in all spatial dimensions as well as in time.
We extend this observation scheme to a more realistic setting:
in practice one often observes data on a $d$-dimensional area ($d \in \bbn$), which is small with respect to some of its dimensions (for instance, the spatial dimensions) and large with respect to others (for instance, the temporal dimension). 
Hence, with regard to such cases, it is appropriate to assume the observed data to expand to infinity in some dimensions, but remain fixed in some others.
Such observation schemes require to split up every point and every lag in its components corresponding to the fixed and increasing domains. \\
-- {For such general observation schemes} we have to  extend the asymptotic theory developed in  \cite{buhl3}  considerably. The empirical extremogram estimator used in the first step of the semiparametric estimation procedure needs to be extended and asymptotic results need to be verified.
For an arbitrary parametric extremogram model we then derive asymptotic results of its generalised least squares estimators, which differ considerably from those obtained when the grid increases in all dimensions. 

Our paper is organised as follows.
In Section~\ref{s2} we introduce the theoretical framework of strictly stationary regularly varying processes. We define the extremogram, the observation scheme with its fixed and increasing dimensions as well as assumptions and asymptotic second order properties following from regular variation. 
Section~\ref{s4} presents the empirical and the pre-asymptotic extremogram.
Here we prove a CLT for the empirical extremogram centred by the pre-asymptotic version.
We also specify the asymptotic covariance matrix.
We prove a CLT for the empirical extremogram centred by the true extremogram under more restrictive assumptions. 
To formally state the asymptotic properties of the empirical extremogram, we need to quantify the dependence in a stochastic process, taking into account the different types of observation areas. 
For processes with Fr{\'e}chet margins we prove asymptotic normality of the empirical extremogram centred by the true one. 
In case the required conditions are not satisfied, we provide  assumptions under which a CLT for a bias corrected version of the empirical extremogram can be obtained.
Section~\ref{s5} is dedicated to the parameter estimation by a generalised least squares method.
Under appropriate regularity conditions we prove consistency and asymptotic normality, where the rate of convergence depends on the observation scheme. We also present the covariance matrix in a semi-explicit form. 
In Section~\ref{s3} we show our method at work for Brown-Resnick space-time processes.
We state conditions for Brown-Resnick processes that imply the mixing conditions from Section~\ref{s4} and are hence sufficient to obtain the corresponding CLTs for the empirical extremogram. 
These conditions depend highly on the model for the associated variogram.
Finally, in Section~\ref{s32} we apply these results to three different dependence models of the Brown-Resnick process, and prove the mixing conditions, which guarantee the asymptotic normality of the empirical extremogram, as well as the regularity conditions of the generalised least squares estimates. 
In Section~\ref{s6}  we examine the finite sample properties of the GLSEs in a simulation study, fitting  the parametric models described in Section~\ref{s32} to simulated  Brown-Resnick processes. We apply subsampling methods to obtain asymptotically valid confidence bounds of the parameters. We examine how  the sample size affects the estimates and compare with the theoretical results obtained in previous sections.
Many proofs are rather technical and postponed to an Appendix.

\section{Model description and the observation scheme} \label{s2}

We consider the same theoretical framework as in \citet{buhl3} and \citet{Steinkohl3} of {\textit{a strictly stationary regularly varying process}} $\{X(\bs s): \bs s \in \mathbb{R}^d\}$ for $d\in\N$, defined on a probability space $(\Omega,\mathcal{G},\bbp)$. 
This implies that there exists some normalizing sequence $0<a_n \rightarrow \infty$ such that $\mathbb{P}(|X(\bs 0)|>a_n) \sim n^{-d}$ as $n \rightarrow \infty$ and that for every  finite  set $\mathcal{I} \subset \mathbb{R}^d$ with cardinality $|\mathcal{I}|<\infty$,
\begin{align} \label{regvar}
n^d\mathbb{P}\Big(\frac{X_{\mathcal{I}}}{a_n} \in \cdot\Big) \stv \mu_{\mathcal{I}}(\cdot), \quad n \rightarrow \infty,
\end{align}
for some non-null  Radon measure $\mu_{\mathcal{I}}$ on the Borel sets in $\overline{\mathbb{R}}^{|\mathcal{I}|}\backslash\{\bs 0\}$, where $\ov{\mathbb{R}}=\mathbb{R} \cup \{-\infty,\infty\}$ and $X_{\mathcal{I}}$ denotes the vector $(X(\bs s) : \bs s\in \mathcal{I}).$ 
The limit measure is homogeneous:
$$\mu_{\mathcal{I}}(xC)=x^{-\beta}\mu_{\mathcal{I}}(C), \quad x>0,$$ 
for every Borel set $C\subset\overline{\mathbb{R}}^{|\mathcal{I}|}\backslash\{\bs 0\}$.
The notation $\stv$ stands for vague convergence, and $\beta>0$ is called the {\textit{index of regular variation}}. 
Furthermore, $f(n)\sim g(n)$ as $\nto$ means that $\lim_{\nto} f(n)/ g(n) =1$. 
If $\mathcal{I}$ is a singleton; i.e., $\mathcal{I}=\{\bs s\}$ for some $\bs s \in \mathbb{R}^d$, we set 
\beam\label{mu}
\mu_{\{\bs s\}}(\cdot)=\mu_{\{\bs 0\}}(\cdot)=:\mu(\cdot),
\eeam
which is justified by stationarity. 
For more details see \cite{buhl3}.
For background on regular variation for stochastic processes and vectors see \citet{HL2005,HL} and \citet{Resnick,Resnick3}.

The extremogram for values in $\R^d$ is defined as follows.

\begin{definition}[Extremogram] \label{DefExtremo}
Let $\{X(\bs{s}): \bs{s} \in \mathbb{R}^d\}$ be a strictly stationary regularly varying process and $a_n\to\infty$ a sequence satisfying \eqref{regvar}.
For $\mu$ as in \eqref{mu} and two $\mu$-continuous Borel sets $A$ and $B$  in $\overline{\mathbb{R}}\backslash\{0\}$ (i.e., $\mu(\partial A)=\mu(\partial B)=0$) such that $\mu(A)>0$,  the \textnormal{extremogram} is defined as
\begin{align} \label{extremo}
\rho_{AB}(\bs{h})=\lim_{n \rightarrow \infty} \frac{\mathbb{P}(X(\bs{0})/a_n \in A,X(\bs{h})/a_n \in B)}{\mathbb{P}(X(\bs{0})/a_n \in A)}, \quad \bs{h} \in \mathbb{R}^d.
\end{align} 
For $A=B=(1,\infty)$, the extremogram $\rho_{AB}(\bs h)$ is the {{tail dependence coefficient}} between $X(\bs{0})$ and $X(\bs{h})$ (cf. \citet{Beirlant}, Section~9.5.1).
\end{definition}
For the data we allow for realistic observation schemes described in the following.

\begin{assumption}\label{as2.2}
The data are given in an observation area $\mathcal{D}_n \subset \mathbb{Z}^d$ that can (possibly after reordering) be decomposed into 
\begin{align}\label{observed}
\mathcal{D}_n=\mathcal{F} \times \mathcal{I}_n,
\end{align}
where for $q,w\in\N$ satisfying $w+q=d$:
\begin{enumerate}
\item[(1)]
$\mathcal{F} \subset \mathbb{Z}^q$ is a fixed domain independent of $n$, and
\item[(2)]
 $\mathcal{I}_n=\{1,\ldots,n\}^w$ 
  is an increasing sequence of regular grids.  
 \end{enumerate}
\end{assumption}

This setting is similar to that used in \citet{Genton}, where asymptotic properties of space-time covariance estimators are derived.
The natural extension of the regular grid $\mathcal{I}_n$ to grids with different side lengths only increases notational complexity, which we avoid here. Our focus is on observations schemes, which are partially fixed and partially tend to infinity.

\bexam\label{spacetime}
In the special case where the observation area is given by 
$$\mathcal{D}_n=\mathcal{F} \times \{1,\ldots,n\}$$ 
for $\mathcal{F} \subset \mathbb{R}^{d-1}$, we interpret the observations as generated by a space-time process $\{X(\bs s,t): \bs s \in \mathbb{R}^{d-1}, t \in [0,\infty)\}$ on a fixed spatial and an increasing temporal domain.
\eexam

We shall need some definitions and assumptions, which we summarize as follows.
\begin{assumption}\label{ass0}\hspace*{2cm}\newline
$(1)$ \, For some fixed $\ga>0$ and $\bs{0},\bs \ell\in\R^d$ we define the balls
\begin{align*}
 \mathcal{B}(\bs{0},\ga) &= \big\{\bs{s}\in \mathbb{Z}^d : \|\bs s\| \leq \ga\big\}\text{ and }  \mathcal{B}(\bs{\ell},\ga) = \big\{\bs{s}\in \mathbb{Z}^d : \|\bs \ell-\bs s\| \leq \ga\big\}=\bs \ell + \mathcal{B}(\bs 0,\ga).
\end{align*}
$(2)$ \, 
The estimation of the extremogram is based on a set
$\mathcal{H}=\{\bs h^{(1)},\ldots,\bs h^{(p)}\} \subset \mathcal{B}(\bs 0,\ga)$ of observed lag vectors.\\[2mm]
$(3)$ \,
We decompose points $\bs s \in \mathbb{R}^d$ with respect to the fixed and increasing domains into $\bs s=(\bs f,\bs i) \in \mathbb{R}^q \times \mathbb{R}^w$.\\[2mm]
$(4)$ \,
Similarly, we decompose lag vectors $\bs h=\bs s-\bs s'$ or $\bs \ell=\bs s-\bs s'$ for some $\bs s,\bs s' \in \mathbb{R}^d$ into $\bs h=(\bs h_{\mathcal{F}},\bs h_{\mathcal{I}})$ or $\bs \ell=(\bs \ell_{\mathcal{F}},\bs \ell_{\mathcal{I}})$ in $\mathbb{R}^q \times \mathbb{R}^w$. 
The letter $\bs h$ is used throughout as argument of the extremogram or its estimators.\\[2mm]
$(5)$ \, 
We define the vectorised process $\{\bs{Y}(\bs{s}): \bs{s}\in\bbr^d\}$ by 
$$\bs{Y}(\bs s):=X_{\mathcal{B}(\bs s,\gamma)};$$
 i.e.,  $\bs Y(\bs s)$ is the vector of values of $X$ with indices in  the ball $\mathcal{B}(\bs s,\ga)$.
\\[2mm]
$(6)$ \,
We shall also need the following relations, {already stated in (3.3) and (3.4)} of \cite{buhl3}.
For $a_n\to\infty$ as in \eqref{regvar}, the following limits exist by regular variation of $\{X(\bs s): \bs s \in \mathbb{R}^d\}$. For $\bs \ell \in \mathbb{R}^d$ and $\ga>0$,
\beam
\mu_{\mathcal{B}(\bs 0, \ga)}(C)  &:=& \lim_{\nto} n^d\mathbb{P}(\bs Y(\bs 0)/a_{n} \in C)\label{B6},\\
\tau_{{\mathcal{B}(\bs 0,\gamma) \times \mathcal{B}(\bs \ell,\gamma)}} (C \times D) 
&:=& \lim_{\nto} n^d\mathbb{P}\Big(\frac{\bs Y(\bs 0)}{a_{n}} \in C, \frac{\bs Y(\bs \ell)}{a_{n}} \in D\Big),\label{B7}
\eeam
for a $\mu_{\mathcal{B}(\bs 0, \ga)}$-continuous Borel set $C$ in $\ov\R^{|\mathcal{B}(\bs 0,\ga)|}\backslash\{\bs 0\}$ and a $\tau_{{\mathcal{B}(\bs 0,\gamma) \times \mathcal{B}(\bs \ell,\gamma)}}$-continuous Borel set $C\times D$ in the product space. \\[2mm]
$(7)$ \,  
For arbitrary but fixed $\mu$-continuous Borel sets $A$ and $B$ in $\overline{\mathbb{R}}\backslash\{0\}$ such that $\mu(A)>0$, we define sets $D_1,\ldots,D_{p},D_{p+1}$ by the identity
\beam\label{Di}
\{\bs Y(\bs s)\in D_i\} = \{X(\bs{s})\in A,X(\bs{s}+\bs h^{(i)})\in B\}
\eeam
for $i=1,\ldots,p$, and $\{\bs Y(\bs s)\in D_{p+1}\} = \{X(\bs s)\in A\}$.
Note in particular that, by the relation between $\{\bs Y(\bs s): \bs s \in\R^d\}$ and $\{X(\bs s): \bs s \in\R^d\}$ and  regular variation, 
$$\mu_{\mathcal{B}(\bs 0,\ga)}(D_{p+1})=\lim\limits_{\nto} n^d \mathbb{P}({\bs Y} (\bs 0)/a_n \in D_{p+1})=\lim\limits_{\nto} n^d \mathbb{P}(X(\bs 0)/a_n \in A)=\mu(A).$$
\halmos
\end{assumption}

\section{Limit theory for the empirical extremogram}\label{s4}

{We derive asymptotic properties of the empirical extremogram
by formulating appropriate mixing conditions, generalising the results obtained in \cite{buhl3} to the more realistic setting of this paper. }
The proofs are based on spatial mixing conditions, which have to be adapted to the decomposition into a fixed and an increasing observation domain.
In principle, our proofs rely on general results of \citet{Ibragimov} and \citet{Bolthausen}.

The main theorem of this section states asymptotic normality of the empirical extremogram sampled at lag vectors $\bs h \in \calh$ and centred by its {pre-asymptotic} counterpart. 
The empirical and the pre-asymptotic extremograms are defined in Eq.~\eqref{EmpEst} and \eqref{preasymptotic}.

For the definition of the empirical extremogram we need the following notation: for $k \in \mathbb{N}$, an arbitrary set $\mathcal{Z} \subset \mathbb{Z}^k$ and a fixed vector $\bs h \in \mathbb{Z}^k$,
define the sets
\beam\label{lagclos}
\mathcal{Z}(\bs h):=\{\bs z \in \mathcal{Z}: \bs z+ \bs h\in \mathcal{Z}\},
\eeam
which is the set of vectors $\bs z \in \mathcal{Z}$ such that with $\bs z$ also the lagged vector $\bs z+\bs h$ belongs to $\mathcal{Z}$.

\begin{definition}\label{ee}
Let $\{X(\bs{s}): \bs{s} \in \mathbb{R}^d\}$ be a strictly stationary regularly varying  process, which is observed on $\mathcal{D}_n=\mathcal{F} \times \mathcal{I}_n$ as in \eqref{observed}. 
Let $A$ and $B$ be $\mu$-continuous Borel sets  in $\overline{\mathbb{R}}\backslash\{0\}$ such that $\mu(A)>0$. 
For a sequence $m=m_n \rightarrow \infty$ and $m_n=o(n)$ as $\nto$ define the following quantities:
\begin{enumerate}
\item[(1)]
The \textnormal{empirical extremogram} 
\begin{align} \label{EmpEst}
\widehat{\rho}_{AB,m_n}(\bs h):=\frac{\dfrac1{|\mathcal{D}_n(\bs h)|}\sum\limits_{\bs s \in \mathcal{D}_n(\bs h)} \mathbbmss{1}_{\{X(\bs s)/a_m \in A, X(\bs s+\bs h)/a_m \in B\}}}
{\dfrac1{|\mathcal{D}_n|}\sum\limits_{\bs s \in \mathcal{D}_n} \mathbbmss{1}_{\{ X(\bs s)/a_m \in A\}}}, \quad  \bs h\in\calh.
\end{align}
For a fixed data set the value $a_m=a_{m_n}$ has to be specified as a large empirical quantile.
\item[(2)]
The \textnormal{pre-asymptotic extremogram} 
\begin{equation}\label{preasymptotic}
\rho_{AB,m_n}(\bs h) = \frac{\mathbb{P}\left(X(\bs{0})/a_m \in A, X(\bs{h})/a_m\in B\right)}{\mathbb{P}(X(\bs{0})/a_m\in A)}, \quad \bs{h} \in \mathbb{R}^d.
\end{equation}
\end{enumerate}
\end{definition}

Key of the proofs of consistency and asymptotic normality of the empirical extremogram below is the fact that $\wh{\rho}_{AB,m_n}(\bs h)$ is the empirical version of the pre-\-asymp\-totic extremogram $\rho_{AB,m_n}(\bs h)$. {This can for different $\bs h \in \mathcal{B}(\bs 0,\ga)$ in turn be viewed as a ratio of pre-asymptotic versions of $\mu_{\mathcal{B}(\bs 0, \ga)}(C(\bs h))$ (cf. Eq.~\eqref{B6}). The sets $C(\bs h)$ are implicitly defined by $\{\bs Y(\bs s) \in C(\bs h)\}=\{X(\bs s) \in A, X(\bs s+\bs h) \in B\}$ for $\bs s \in \mathbb{R}^d$. 
Then in particular, for $\bs h \in \mathcal{B}(\bs 0,\ga)$, 
$$\mathbb{P}\Big(\frac{X(\bs 0)}{a_m} \in A, \frac{X(\bs h)}{a_m} \in B\Big)=\mathbb{P}\Big(\frac{\bs Y(\bs 0)}{a_m} \in C(\bs h)\Big).$$ 
Note that, by \eqref{Di}, if $\bs h=\bs h^{(i)} \in \mathcal{H}$, then $C(\bs h)=D_i$, and  if $\bs h=\bs 0$ and $A=B$ then $C(\bs h)=D_{p+1}$.}

In view of \eqref{B6}, $\mu_{\mathcal{B}(\bs 0, \ga)}(C(\bs h))$ can be estimated by an empirical mean, where the estimator has to cope with Assumption~\ref{as2.2} of an observation area with fixed and increasing domain.

\bde
Assume the situation of Definition~\ref{ee}.
Based on observations on $\mathcal{D}_n=\mathcal{F} \times \mathcal{I}_n$ as in \eqref{observed} decompose the observations $\bs s=(\bs f,\bs i)\in \mathcal{F} \times \mathcal{I}_n$ and the lags  $\bs h=(\bs h_\calf,\bs h_{\cali})\in\calh$ as in Assumption~\ref{ass0}(3) and~(4).
For $\bs h_{\mathcal{F}}\in\calh$ define $\mathcal{F}(\bs h_{\mathcal{F}})$ as in \eqref{lagclos}.
Then an empirical version of $\mu_{\mathcal{B}(\bs 0, \ga)}(C(\bs h))$ is for $\bs h\in\calh$ given by
\beam\label{estmu}
\widehat{\mu}_{\mathcal{B}(\bs 0,\ga),m_n}(C(\bs h)) := \frac{m_n^d}{n^w} \sum_{\bs i\in \mathcal{I}_n} \frac{1}{|\mathcal{F}(\bs h_{\mathcal{F}})|} \sum_{\bs f \in \mathcal{F}(\bs h_{\mathcal{F}})} \mathbbmss{1}_{\{\frac{\bs Y(\bs f, \bs i)}{a_m} \in C(\bs h)\}}.
\eeam
\ede

Observe that for fixed $\bs h_{\mathcal{F}} \in \Z^q$ and observations on $\mathcal{D}_n=\mathcal{F} \times \mathcal{I}_n$ there will be points $\bs s=(\bs f,\bs i) \in \mathcal{F}(\bs h_{\mathcal{F}}) \times \mathcal{I}_n$ with $\bs i$ near the boundary of $\mathcal{I}_n$, such that not all components of the vector $\bs Y(\bs s)=\bs Y(\bs f,\bs i)$ are observed.
However, since we investigate asymptotic properties of $\mathcal{I}_n$ whose boundary points are negligible, we can ignore such technical details. 
As will be seen in the proofs below, for every $\bs h \in \mathcal{H}$, the empirical extremogram $\widehat{\rho}_{AB,m_n}(\bs h)$ is asymptotically equivalent to the ratio of estimates $\widehat{\mu}_{\mathcal{B}(\bs 0,\ga),m_n}(C(\bs h))/\widehat{\mu}_{\mathcal{B}(\bs 0,\ga),m_n}(D_{p+1})$.

Limit results for the empirical extremogram \eqref{EmpEst} involve the calculation of mean and variance of $\widehat{\mu}_{\mathcal{B}(\bs 0,\ga),m_n}(C(\bs h^{(i)}))=\widehat{\mu}_{\mathcal{B}(\bs 0,\ga),m_n}(D_i)$ for $\bs h^{(i)} \in \calh$.
Strict stationarity and Assumption~\ref{ass0}(6) yields immediately by a law of large numbers that \\$\E[\widehat{\mu}_{\mathcal{B}(\bs 0,\ga),m_n}(D_i)] \to {\mu}_{\mathcal{B}(\bs 0,\ga)}(D_i)$ as $\nto$.
Calculation of the variance involves the covariance structure and
we decompose as in Assumption~\ref{ass0}(4) $\bs h^{(i)}$ into $\bs h^{(i)}=(\bs h^{(i)}_{\mathcal{F}},\bs h^{(i)}_{\mathcal{I}})$ $\in \mathbb{R}^q \times \mathbb{R}^w$.
We have to calculate for $\bs f,\bs f' \in \mathcal{F}(\bs h_{\mathcal{F}}^{(i)})$ and $\bs i,\bs i' \in \mathcal{I}_n$, 
$$\cov\Big[\one{\{\frac{\bs Y(\bs f,\bs i)}{a_m} \in D_i\}},\one{\{\frac{\bs Y(\bs f',\bs i')}{a_m} \in D_i\}}\Big]=\cov\Big[\one{\{\frac{\bs Y(\bs 0)}{a_m} \in D_i\}},\one{\{\frac{\bs Y(\bs \ell_{\mathcal{F}},\bs \ell_{\mathcal{I}})}{a_m} \in D_i\}}\Big]$$
with $\bs \ell_{\mathcal{F}}=\bs f-\bs f'$ and $\bs \ell_{\mathcal{I}}=\bs i-\bs i'$, where the equality holds by stationarity. 
The lag vectors $\bs\ell_\calf$ and $\bs\ell_\cali$ are contained in $L_{\mathcal{F}}^{(i,i)}$ and $L_n$, respectively, where for $i,j \in \{1,\ldots,p\}$,
\begin{align}
L_{\mathcal{F}}^{(i,j)}:=\{\bs f-\bs f': \bs f \in \mathcal{F}(\bs h_{\mathcal{F}}^{(i)}), \bs f' \in \mathcal{F}(\bs h_{\mathcal{F}}^{(j)})\}\quad\mbox{and}\quad
L_n:=\{\bs i-\bs i': \bs i, \bs i'\in \mathcal{I}_n\}. \label{lag_fix_2}
\end{align}
The number of appearances of the lag $\bs \ell_{\mathcal{F}}$  we denote for $i,j \in \{1,\ldots,p\}$ by
\begin{align}\label{lag_appearance_fix}
\textnormal{N}^{(i,j)}_{\mathcal{F}}(\bs \ell_{\mathcal{F}}):=\sum_{\bs f \in \mathcal{F}(\bs h_{\mathcal{F}}^{(i)})} \sum_{\bs f' \in \mathcal{F}(\bs h_{\mathcal{F}}^{(j)})} \one{\{\bs f-\bs f'=\bs \ell_{\mathcal{F}}\}} 
\end{align}
Observe that a lag $(\bs \ell_{\mathcal{F}},\bs \ell_{\mathcal{I}})$ with $\bs \ell_{\mathcal{I}}=(\ell_{\mathcal{I}}^{(1)},\ldots,\ell_{\mathcal{I}}^{(w)})$ appears in $L_{\mathcal{F}}^{(i,i)} \times L_n$ exactly $\textnormal{N}_{\mathcal{F}}^{(i,i)}(\bs \ell_{\mathcal{F}})\prod_{j=1}^w(n-|\ell_{\mathcal{I}}^{(j)}|)$ times.
We show in Lemma~\ref{st-asy_la} that
\begin{align}\label{varpm0}
\lefteqn{\var\big[\widehat{\mu}_{\mathcal{B}(\bs 0,\ga),m_n}(D_i)\big] 
=\frac{m_n^{2d}}{n^{2w}|\mathcal{F}(\bs h_{\mathcal{F}}^{(i)})|^2} \var\big[\sum_{\bs f \in \mathcal{F}(\bs h_{\mathcal{F}}^{(i)})}\sum_{\bs i\in \mathcal{I}_n} \mathbbmss{1}_{\{\frac{\bs Y(\bs f,\bs i)}{a_m} \in D_i\}}\big]}\notag\\
&=\frac{m_n^{2d}}{n^{2w}|\mathcal{F}(\bs h_{\mathcal{F}}^{(i)})|^2} \Big(|\mathcal{F}(\bs h_{\mathcal{F}}^{(i)})| n^w\var\big[\mathbbmss{1}_{\{\frac{\bs Y(\bs 0)}{a_m} \in D_i\}}\big] \\
&\quad\quad\quad\quad +\sum_{\bs f, \bs f' \in \mathcal{F}(\bs h_{\mathcal{F}}^{(i)})}\sum_{\bs i, \bs i' \in \mathcal{I}_n \atop (\bs f,\bs i) \neq (\bs f',\bs i')} 
\cov\big[\mathbbmss{1}_{\{\frac{\bs Y(\bs f, \bs i)}{a_m} \in D_i\}},\mathbbmss{1}_{\{\frac{\bs Y(\bs f', \bs i')}{a_m} \in D_i\}}\big]\Big) \notag\\
&\sim  \frac{m_n^d}{n^w} \frac{1}{|\mathcal{F}(\bs h_{\mathcal{F}}^{(i)})|} 
\Big(\mu_{\mathcal{B}(\bs 0,\ga)}(D_i) \nonumber\\
&\quad\quad\quad\quad +\sum_{\bs \ell_{\mathcal{I}} \in \mathbb{Z}^w} \frac{1}{|\mathcal{F}(\bs h_{\mathcal{F}}^{(i)})|}\sum_{\bs \ell_{\mathcal{F}} \in L_{\mathcal{F}}^{(i,i)} \atop (\bs \ell_{\mathcal{F}}, \bs \ell_{\mathcal{I}}) \neq \bs 0} \textnormal{N}_{\mathcal{F}}^{(i,i)}(\bs \ell_{\mathcal{F}}) \, \tau_{\mathcal{B}(\bs 0,\ga)\times \mathcal{B}((\bs \ell_{\mathcal{F}}, \bs \ell_{\mathcal{I}}),\ga)}(D_i \times D_i)\Big) \nonumber\\
&=:  \frac{m_n^d}{n^w}\sigma_{\mathcal{B}(\bs 0,\gamma)}^2(D_i),\quad\nto. \label{varpm3}
\end{align}

\brem
For comparison we recall the expression in the corresponding Lemma~5.1 of \citet{buhl3}, where $\mathcal{F}$ is not fixed, but part of the increasing regular grid. 
Then $|\mathcal{F}(\bs h_{\mathcal{F}}^{(i)})| \sim \textnormal{N}_{\mathcal{F}}^{(i,i)}(\bs \ell_{\mathcal{F}}) \sim n^q$ as $\nto$, such that \eqref{varpm0} can be approximated as follows: 
\begin{align*}
&\var\big[\widehat{\mu}_{\mathcal{B}(\bs 0,\gamma),m_n}(D_i)\big] \\
\sim & \frac{m_n^d}{n^w n^q} \Big(\mu_{\mathcal{B}(\bs 0,\ga)}(D_i)+\sum_{\bs \ell_{\mathcal{I}} \in \mathbb{Z}^w} \sum_{\bs \ell_{\mathcal{F}} \in \mathbb{Z}^q \atop (\bs \ell_{\mathcal{F}}, \bs \ell_{\mathcal{I}}) \neq \bs 0} \tau_{\mathcal{B}(\bs 0,\ga)\times \mathcal{B}((\bs \ell_{\mathcal{F}}, \ell_{\mathcal{I}}),\ga)}(D_i \times D_i)\Big) \\
= & \Big(\frac{m_n}{n}\Big)^d \Big(\mu_{\mathcal{B}(\bs 0,\ga)}(D_i)+\sum_{\bs \ell \in \mathbb{Z}^d \setminus \{\bs 0\}} \tau_{\mathcal{B}(\bs 0,\ga)\times \mathcal{B}(\bs \ell,\ga)}(D_i \times D_i)\Big),\quad \nto. 
\end{align*}
Thus, a difference from the setting of a partly fixed observation area $\mathcal{F}\subset\cald_n$ is that the fixed observation terms do not disappear asymptotically, but remain as  constants in the limit expression.
\erem

\subsection{The extremogram for regularly varying processes}\label{s41}

For proving asymptotic normality of the empirical extremogram we have to require appropriate mixing conditions and make use of a large/small block argument as in \cite{buhl3}.
For simplicity we assume that $n^w/m_n^d$ is an integer and subdivide $\mathcal{D}_n$ into $n^w/m_n^d$ non-overlapping $d$-dimensional large blocks $\mathcal{F} \times \mathcal{B}_i$ for $i=1,\ldots,n^w/m_n^d$, where the $\mathcal{B}_i$ are $w$-dimensional cubes with side lengths $m_n^{d/w}$.
 From those large blocks we then cut off smaller blocks, which consist of the first $r_n$ elements in each of the $w$ increasing dimensions. The large blocks are then separated (by these small blocks) with at least the distance $r_n$ in all $w$ increasing dimensions and shown to be asymptotically independent. 
{Such large/small block arguments are common in verifying properties of estimators in extreme value theory, in particular in a time series context, cf. for example \citet{Davis2}, Section~6. For a visualization in the $2$-dimensional case $d=2$ with $w=1$ increasing dimension, see Figure~\ref{blocks_bigsmall}.}

\begin{figure}
\centering
\includegraphics[scale=0.45]{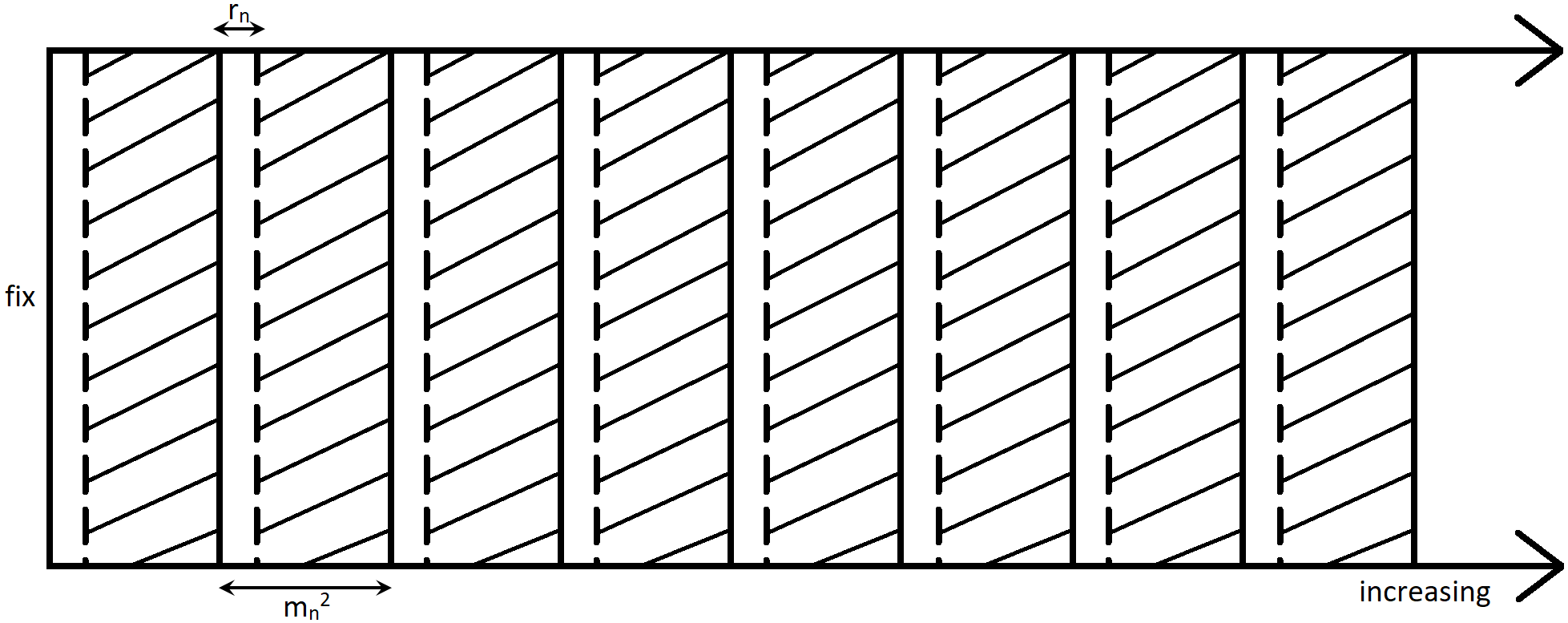}
\caption{Visualization of the large/small block argument in the case $d=2$ and $w=1$. The large blocks are the hatched areas; the small blocks are given by the small areas between them.} \label{blocks_bigsmall}
\end{figure}

In order to formulate the CLT below, in particular, the asymptotic covariance matrix, we need to compute $\cov[\widehat{\mu}_{\mathcal{B}(\bs 0,\ga),m_n}(D_i),\widehat{\mu}_{\mathcal{B}(\bs 0,\ga),m_n}(D_j)]$ for possibly different $i,j \in \{1,\ldots,p\}$. 
The asymptotic results stated in Theorem~\ref{stasyn} extend those in Theorem~4.2 of \cite{buhl3}, where the observation area $\mathcal{D}_n$ is assumed  to increase with $n$ in all dimensions. 
The decomposition \eqref{observed} into a fixed domain $\mathcal{F}$ and an increasing domain $\cali_n$ results in mixing conditions which focus on properties for $\cali_n$ increasing to $\Z^w$,  while $\calf$ remains fix and appears in the limit, similarly as in Eq.~\eqref{varpm0}.

\bthe\label{stasyn}
Let $\{X(\bs s): \bs s \in \mathbb{R}^d\}$ be a strictly stationary regularly varying  process, which is observed on $\mathcal{D}_n=\mathcal{F} \times \mathcal{I}_n$ as in \eqref{observed}. Let $\calh=\{\bs h^{(1)},\ldots,\bs h^{(p)}\} \subset \mathcal{B}(\bs 0,\ga)$ for some $\ga>0$ be a set of observed lag vectors. Suppose that the following conditions are satisfied.
\begin{enumerate}[(M1)]
\item 
$\{X(\bs{s}): \bs{s}\in\bbr^d\}$ is $\alpha$-mixing with respect to $\R^w$ with mixing coefficients $\alpha_{k_1,k_2}(\cdot)$ defined in \eqref{alphaBolt}.
\end{enumerate}
There exist sequences $m_n, r_n \to \infty$ with $m_n^d/n^w \to 0$ and $r_n^w/m_n^d \to 0$ as $\nto$ such that:
\begin{enumerate}[(M1)]
\setcounter{enumi}{1}
\item 
$m_n^{2d}r_n^{2w}/n^w \to 0$.
\label{cond0}
\item
For all $\epsilon>0$, and for all fixed $\bs \ell_{\mathcal{F}} \in \mathbb{R}^q$
with $a_m=a_{m_n}\to\infty$ as in \eqref{regvar},\\[2mm]
\hspace*{-0.7cm}$\lim\limits_{\kto } \limsup\limits_{\nto }
\sum\limits_{\bs \ell_{\mathcal{I}} \in \mathbb{Z}^w \atop k< \|\bs \ell_{\mathcal{I}}\| \leq r_n} 
 m_n^d \, \mathbb{P}(\max\limits_{\bs s \in \mathcal{B}(\bs 0,\gamma)} 
|X(\bs s)|>\epsilon a_m, \max\limits_{\bs s' \in \mathcal{B}((\bs \ell_{\mathcal{F}},\bs \ell_{\mathcal{I}}),\gamma)} |X(\bs s')|>\epsilon a_m)=0$. \label{cond1}
\item 
\begin{enumerate}[(i)]
\item 
$\lim\limits_{\nto } m_n^d \sum\limits_{\bs \ell \in \mathbb{Z}^w : \Vert{\bs \ell}\Vert > r_n} \alpha_{1, 1}(\|\bs \ell\|)=0$, \label{cond2} 
\item 
$\sum\limits_{\bs \ell \in \mathbb{Z}^w} \alpha_{k_1,k_2}(\Vert{\bs \ell}\Vert) < \infty$ \ \mbox{ for } $2\le k_1+k_2\le 4$, \label{cond3}
\item 
$\lim\limits_{\nto }  m_n^{d/2} n{^{w/2}} \ \alpha_{1,{n^w}}(r_n)=0$. \label{cond4} 
\end{enumerate}
\end{enumerate}
Then the empirical extremogram $\wh{\rho}_{AB,m_n}$ defined in \eqref{EmpEst}, sampled at lags in $\calh$ and centred by the pre-asymptotic extremogram $\rho_{AB,m_n}$ given in \eqref{preasymptotic}, is asymptotically normal; i.e.,
\begin{equation}
\sqrt{\frac{n^w}{m_n^d}}\Big[\widehat{\rho}_{AB,m_n}(\bs h^{(i)}) -\rho_{AB,m_n}(\bs h^{(i)})\Big]_{i=1,\ldots,p} \std \mathcal{N}(\bs 0,\Pi),\quad \nto,\label{asyspace}
\end{equation}
where
 $\Pi = \mu(A)^{-4} F\Sigma F\trans \in \mathbb{R}^{p \times p}$. Writing $\bs h^{(i)}=(\bs h_{\mathcal{F}}^{(i)},\bs h_{\mathcal{I}}^{(i)})$ for $1 \leq i \leq p+1$, with the convention that $(\bs h_{\mathcal{F}}^{(p+1)},\bs h_{\mathcal{I}}^{(p+1)})=\bs 0$,  {and recalling  \eqref{lag_fix_2} and \eqref{lag_appearance_fix},}
 the matrix $\Sigma \in \mathbb{R}^{(p+1) \times (p+1)}$ has components {
\begin{align}\label{SigmaCoeff2}
\Sigma_{ij} =& \frac{1}{|\mathcal{F}(\bs h_{\mathcal{F}}^{(i)})||\mathcal{F}(\bs h_{\mathcal{F}}^{(j)})|} \Big(|\mathcal{F}(\bs h_{\mathcal{F}}^{(i)}) \cap \mathcal{F}(\bs h_{\mathcal{F}}^{(j)})| \mu_{\mathcal{B}(\bs 0,\ga)}(D_i \cap D_j)\\
&+\sum_{\bs \ell_{\mathcal{I}} \in \mathbb{Z}^w} \sum_{\bs \ell_{\mathcal{F}} \in L_{\mathcal{F}}^{(i,j)} \atop (\bs \ell_{\mathcal{F}}, \bs \ell_{\mathcal{I}}) \neq \bs 0} \textnormal{N}_{\mathcal{F}}^{(i,j)}(\bs \ell_{\mathcal{F}}) \,\tau_{\mathcal{B}(\bs 0,\ga)\times \mathcal{B}((\bs \ell_{\mathcal{F}}, \bs \ell_{\mathcal{I}}),\ga)}(D_i \times D_j)\Big), \quad 1 \leq i, j \leq p+1. \nonumber
\end{align}
If $i=j$, we have $\Sigma_{ii}=\sigma^2_{\mathcal{B}(\bs 0,\gamma)}(D_i)$ with $\sigma^2_{\mathcal{B}(\bs 0,\gamma)}(D_i)$ specified in \eqref{varpm3}.}
The matrix $F=[F_1,F_2]$ consists of a diagonal matrix $F_1$ and a vector $F_2$ in the last column:
\begin{align*}
F_1 = \textnormal{diag}(\mu(A)) \in \mathbb{R}^{p \times p}, \quad F_2 = (-\mu_{\mathcal{B}(\bs 0, \ga)}(D_1),\ldots,-\mu_{\mathcal{B}(\bs 0, \ga)}(D_p))\trans. 
\end{align*}
\ethe
Note that condition (M3) is the analogue of condition~(3.3) of \citet{Davis2} in the time series case and thus similar in spirit but weaker than the classical anti-clustering condition $D'(\epsilon a_n)$ {as explained there.}

\bco\label{LLN}
Assume the setting of Theorem~\ref{stasyn} and
suppose that the following conditions are satisfied.
\begin{enumerate}[(1)]
\item 
$\{X(\bs{s}): \bs{s}\in\bbr^d\}$ is $\alpha$-mixing with respect to $\R^w$ with mixing coefficients $\alpha_{k_1,k_2}(\cdot)$ defined in \eqref{alphaBolt}.
\item
There exist sequences $m:=m_n, r:=r_\nto$ with $m_n^d/n^w \to 0$ and $r_n^w/m_n^d \to 0$ as $\nto$ such that (M3) and (M4i) hold.
\end{enumerate}
Then, as $\nto$,
\begin{align*}
\widehat{\rho}_{AB,m_n}(\bs h^{(i)}) \stp \rho_{AB}(\bs h^{(i)}), \quad i=1,\ldots,p, 
\end{align*} 
\eco

\bproof
As in part II of the proof of Theorem~\ref{stasyn} (cf. Appendix~\ref{app_1}), we find that for $i=1,\ldots,p$, as $\nto$,
$$\widehat{\rho}_{AB,m_n}(\bs h^{(i)}) \sim \frac{\widehat{\mu}_{\mathcal{B}(\bs 0,\ga),m_n}(D_i)}{\widehat{\mu}_{\mathcal{B}(\bs 0,\ga),m_n}(D_{p+1})} \stp \frac{{\mu}_{\mathcal{B}(\bs 0,\ga)}(D_i)}{{\mu}_{\mathcal{B}(\bs 0,\ga)}(D_{p+1})}=\rho_{AB}(\bs h^{(i)}),$$
where the sets $D_i$ and $D_{p+1}$ are defined in \eqref{Di}. Convergence in probability follows by Lemma~\ref{st-asy_la} and Slutzky's theorem. The last identity holds by definitions~\eqref{extremo} and~\eqref{B6}, recalling that ${\mu}_{\mathcal{B}(\bs 0,\ga)}(D_{p+1})=\mu(A)>0$.
\eproof

\brem\label{rem4.3}
(i) \, If the choice $m_n=n^{\beta_1}$ and $r_n=n^{\beta_2}$ with $0 < \beta_2 < \beta_1d/w <1$ satisfies conditions (M3) and (M4),
then for $\beta_1 \in (0,w/(2d))$ and $\beta_2 \in (0,\min\{\beta_1d/w;1/2-\beta_1d/w\})$ the condition (M2) also holds and we obtain the CLT \eqref{asyspace}. \\[2mm]
(ii) \, The pre-asymptotic extremogram \eqref{preasymptotic} in the CLT \eqref{asyspace} can be replaced by the true one \eqref{extremo}, if  the  pre-asymptotic extremogram converges to the true extremogram with the same convergence rate; i.e., if 
\begin{equation}\label{condition5}
\sqrt{\frac{n^w}{m_n^d}}\Big[\rho_{AB,m_n}(\bs h^{(i)}) - \rho_{AB}(\bs h^{(i)})\Big]_{i=1,\ldots,p} \to \bs 0,\quad n\to \infty.
\end{equation}
(iii) \, Unfortunately, for general regularly varying processes, it is not known if the bias condition~\eqref{condition5} holds, but the CLT~\eqref{asyspace} based on the pre-asymptotic extremogram holds. 
Hence,  the important asymptotic interpretation of the empirical extremogram as a conditional probability of extremal events remains; {cf. \citet{cho}, \citet{Davis2}, and \citet{Drees} and references therein.}
An important class of processes, where we know conditions such that \eqref{condition5} is satisfied or not, are the {max-stable processes} with finite-dimensional Fr{\'e}chet marginal distributions, as defined in Section~\ref{s42}.
\erem

\subsection{The extremogram of processes with Fr{\'e}chet marginal distributions}\label{s42}

We start with the definition of {max-stable processes}.

\begin{definition}[{Max-stable process}]
A process $\{X(\bs s): \bs s \in \bbr^d\}$ is called max-stable if there exist sequences $c_n(\bs s)>0$ and $d_n(\bs s)$ for $\bs s \in \bbr^d$ and $n \in \bbn$ such that
\begin{align}
\Big\{c_n^{-1}(\bs s)\Big(\bigvee_{j=1}^n X_j(\bs s)-d_n(\bs s)\Big): \bs s \in \bbr^d\Big\} \stackrel{d}{=} \{X(\bs s): \bs s \in \bbr^d\}, \label{max_stable_intro}
\end{align}
where $\{X_j(\bs s): \bs s \in \bbr^d\}$ are independent replicates of $\{X(\bs s): \bs s \in \bbr^d\}$ and the maximum is taken componentwise.
\end{definition}
If max-stable processes have Fr{\'e}chet marginal distributions, they are regularly varying.
 Theorem~\ref{CLT_True}  below states a necessary and sufficient condition for such processes such that both \eqref{asyspace} and \eqref{condition5} hold, yielding the CLT~\eqref{ohnebias} for the empirical extremogram~\eqref{EmpEst} centred by the the true one~\eqref{extremo}. 
 In case this condition is not satisfied, Theorem~\ref{Thm_asynbias} states conditions such that \eqref{ohnebias} holds for a bias corrected version of the empirical extremogram.


\begin{theorem}[CLT for processes with Fr{\'e}chet margins]\label{CLT_True}
Let $\{X(\bs s): \bs s \in \bbr^d\}$ be a strictly stationary max-stable process with 
standard unit Fr{\'e}chet margins, which is  observed on $\cald_n=\calf\times\cali_n$ as in \eqref{observed}.
Let $\calh=\{\bs h^{(1)},\ldots,\bs h^{(p)}\} \subset \mathcal{B}(\bs 0,\ga)$ for some $\ga>0$ be a set of observed lag vectors. 
Suppose that conditions (M1)--(M4) of Theorem~\ref{stasyn} hold for appropriately chosen sequences $m_n,r_n\to\infty$. 
Let $\rho_{AB}$ be the extremogram \eqref{extremo} and $\rho_{AB,m_n}$ the  pre-asymptotic version \eqref{preasymptotic} for sets $A=(\underline{A},\overline{A})$ and $B=(\underline{B},\overline{B})$ with $0<\underline{A}<\overline{A} \leq \infty$ and $0 <\underline{B}<\overline{B} \leq \infty.$  
Then the limit relation \eqref{condition5} holds if and only if $n^w/m_n^{3d} \to 0$ as $\nto$.
In this case we obtain
\begin{equation}
\sqrt{\frac{n^w}{m_n^d}}\Big[\widehat{\rho}_{AB,m_n}(\bs h^{(i)}) -\rho_{AB}(\bs h^{(i)})\Big]_{i=1,\ldots,p} \std \mathcal{N}(\bs 0,\Pi),\quad \nto,\label{asyspace_BR_true}
\end{equation}
with $\Pi$ specified in Theorem~\ref{stasyn}.
\end{theorem}

\begin{proof}
All finite-dimensional distributions are max-stable distributions with standard unit Fr\'echet margins, hence they are multivariate regularly varying. Furthermore we can choose $a_m = m_n^d$ in Definition~\ref{DefExtremo}.
Let $V_2(\bs h;\cdot,\cdot)$ be the bivariate exponent measure defined by $\mathbb{P}(X(\bs 0) \leq x_1,X(\bs h) \leq x_2)$ $=\exp\{-V_2(\bs h;x_1,x_2)\}$ for $x_1,x_2>0$, cf. \citet{Beirlant}, Section~8.2.2.
From Lemma~A.1(b) of \cite{buhl3} we know that for $\bs h \in \mathcal{H}$ and with $\ov{V}^2_2(\bs h):=\underline{A} \overline{A}/(\overline{A}-\underline{A})(V_2^2(\bs h;\overline{A},\overline{B}) - V_2^2(\bs h;\overline{A},\underline{B}) - V_2^2(\bs h;\underline{A},\overline{B}) +V_2^2(\bs h;\underline{A},\underline{B}))$,
\begin{align}
	\rho_{AB,m_n}(\bs h)=\rho_{AB}(\bs h) + (1+o(1))\Big[\frac1{2\,m_n^{d}}\ov{V}^2_2(\bs h)\Big],\quad\nto. \label{relpreasym}
\end{align}
If $\overline{A}=\infty$ and/or $\overline{B}=\infty$, appropriate adaptations need to be taken, which are described in Lemma~A.1 of \cite{buhl3}.
Hence, for $\bs h \in \mathcal{H}$,
 \begin{align*}
\sqrt{\frac{n^w}{m_n^d}}\big(\rho_{AB,m_n}(\bs h) - \rho_{AB}(\bs h)\big)  
=(1+o(1)) \sqrt{\frac{n^w}{m_n^{3d}}} \frac{\ov{V}^2_2(\bs h)}{2},\quad\nto,
 \end{align*}
which converges to 0 if and only if $n^w/m_n^{3d} \to 0$.
\end{proof}

If $n^w/m_n^{3d} \not\to 0$ in Theorem~\ref{CLT_True}, a CLT centred by the true extremogram can still be obtained for a bias corrected empirical estimator. Eq.~\eqref{relpreasym} is the basis for such a bias correction if the sets $A$ and $B$ are given by $A=(\underline{A},\infty)$ and $B=(\underline{B},\infty)$ with $\underline{A},\underline{B}>0$. In that case we have
\begin{align}\label{preasym_Fr2}
\rho_{AB,m_n}(\bs h)=&\rho_{AB}(\bs h) 
+ (1+o(1))\Big[\frac1{2\,m_n^{d}\underline{A}} \big(\rho_{AB}(\bs h)-2\underline{A}/ \underline{B}\big)\big(\rho_{AB}(\bs h)-1\big)\Big],\quad\nto;
\end{align} 
see \cite{buhl3}, Eq.~(A.4).
An asymptotically bias corrected estimator is given by
\begin{align*} 
\widehat{\rho}_{AB,m_n}(\bs h)-\frac{1}{2m_n^d\underline{A}}\big(\widehat{\rho}_{AB,m_n}\big(\bs h)-2\underline{A}/ \underline{B}\big)\big(\widehat{\rho}_{AB,m_n}(\bs h)-1\big)
\end{align*}
and we set, covering both cases, 
\begin{align}
&\wt{{\rho}}_{AB,m_n}(\bs h) :=\label{biascorrectedextremo} \\
&\begin{cases}
\wh{\rho}_{AB,m_n}(\bs h)-\dfrac{1}{2m_n^d\underline{A}}\big(\widehat{\rho}_{AB,m_n}(\bs h)-2\underline{A}/ \underline{B}\big)\big(\widehat{\rho}_{AB,m_n}(\bs h)-1\big) & \mbox{ if }  \frac{n^w}{m_n^{3d}} \not\to 0 \text{ but }\frac{n^w}{m_n^{5d}} \to 0, \nonumber\\[2mm]
\wh{\rho}_{AB,m_n}(\bs h) & \mbox{ if }\frac{n^w}{m_n^{3d}} \to 0. \nonumber
\end{cases} 
\end{align}
Theorem~\ref{Thm_asynbias} below guarantees asymptotic normality of the bias corrected extremogram for an---according to Theorem~\ref{stasyn}---valid sequence $m_n$ satisfying $n^w/m_n^{5d} \to 0$. 
The proof, which is given in Appendix~\ref{app_3}, generalises that of {Theorem~4.4} of \citet{Steinkohl3}, which covers the special case $\underline{A}=\underline{B}=1$ for Brown-Resnick processes.

\begin{theorem}[CLT for the bias corrected extremogram for processes with Fr{\'e}chet margins]\label{Thm_asynbias}
Let $\{X(\bs s): \bs s \in \bbr^d\}$ be a strictly stationary max-stable process with standard unit Fr{\'e}chet margins. 
Assume the situation of Theorem~\ref{CLT_True} for sets $A=(\underline{A},\infty)$ and $B=(\underline{B},\infty)$ with $\underline{A},\underline{B}>0$.
Then the bias corrected extremogram~\eqref{biascorrectedextremo} is asymptotically normal if and only if $n^w/m_n^{5d} \to 0$. In that case,
\begin{align}
\sqrt{\frac{n^w}{m_n^d}} \Big[\wt{{\rho}}_{AB,m_n}(\bs h^{(i)})-\rho_{AB}(\bs h^{(i)})\Big]_{i=1,\ldots,p} \std \mathcal{N}(\bs 0,\Pi), \label{asynbias_gen}
\end{align}
where $\Pi$ is specified in Theorem~\ref{stasyn}. 
\end{theorem}

\brem \label{summary_cases}
From Theorems~\ref{CLT_True} and~\ref{Thm_asynbias} in relation to Remark~\ref{rem4.3}~(i) we deduce two cases:\\
 (I) For  $w/(5d)<\beta_1 \leq w/(3d)$ we cannot replace the pre-asymptotic extremogram by the theoretical version in \eqref{asyspace_BR_true}, but can resort to a bias correction as described in \eqref{biascorrectedextremo} to obtain 
\beam\label{asynbias}
 n^{(w-\beta_1d)/2}\Big[\wt{\rho}_{AB,m_n}(\bs h^{(i)}) -\rho_{AB}(\bs h^{(i)})\Big]_{i=1,\ldots,p} \std \mathcal{N}(\bs 0,\Pi),\quad \nto,
\eeam
for sets $A=(\underline{A},\infty)$ and $B=(\underline{B},\infty)$ with covariance matrix $\Pi$ specified in Theorem~\ref{stasyn}.\\
 (II) For $w/(3d)<\beta_1<w/(2d)$ we obtain indeed
 \beam\label{ohnebias}
 n^{(w-\beta_1d)/2}\Big[\widehat{\rho}_{AB,m_n}(\bs h^{(i)}) -\rho_{AB}(\bs h^{(i)})\Big]_{i=1,\ldots,p} \std \mathcal{N}(\bs 0,\Pi),\quad \nto,
 \eeam
with covariance matrix $\Pi$ specified in Theorem~\ref{stasyn}.
\erem
Observe that Remark~\ref{summary_cases} generalises {Remark~4.1} of \cite{Steinkohl3}.

\section{Generalised least squares extremogram estimates}\label{s5}
In this section we fit parametric models to the empirical extremogram using least squares techniques for the parameter estimation. 
Our approach and extremogram models extend the
weighted least squares estimation developed in \citet{steinkohlphd} and \citet{Steinkohl3} considerably.
In these papers {isotropic space-time models such as the Brown-Resnick model (I) of Section~\ref{s32} below have been estimated by separation of space and time, which is not possible for  all models of interest.}
In what follows we present generalised least squares approaches to fit general parametric extremogram models
taking the observation scheme $\cald_n=\calf \times \cali_n$ of a fixed and an increasing domain into account. 
The approach bears some similarity to the semiparametric variogram estimation in \citet{Lahiri2}. 

 Our setting is as follows. 
 Let  $\{\rho_{AB, \bs \theta}(\bs h): \bs h \in \mathbb{R}^{d}, \bs \theta \in \Theta\}$ be some parametric {valid} extremogram model with parameter space $\Theta$ and continuous in $\bs h \in \mathbb{R}^{d}$. 
 Assume that $\rho_{AB}(\cdot)=\rho_{AB,\bs \theta^{\star}}(\cdot)$ with true parameter vector $\bs \theta^{\star}$, which lies by assumption in the interior of $\Theta$. 
Denote by $\widehat{\rho}_{AB,m_n}(\bs h)$ any of the estimators of Theorem~\ref{stasyn}, Theorem~\ref{CLT_True}, or Theorem~\ref{Thm_asynbias} for the appropriately chosen $\mu$-continuous Borel sets $A$ and $B$ such that $\mu(A)>0$ and lags $\bs h \in \mathcal{H}=\{\bs h^{(1)},\ldots,\bs h^{(p)}\}$.

First note that under the much weaker conditions of Corollary~\ref{LLN} the empirical extremogram is a consistent estimator of the extremogram such that as $n \rightarrow \infty$,
\begin{align}
\widehat{\rho}_{AB,m_n}(\bs h^{(i)}) \stp \rho_{AB,\bs \theta^{\star}}(\bs h^{(i)}), \quad i=1,\ldots,p, \label{consthetastar}
\end{align}
Under more restrictive conditions needed for the three CLTs above,
\begin{align}
\sqrt{\frac{n^w}{m_n^d}} \Big[\widehat{\rho}_{AB,m_n}(\bs h^{(i)})-\rho_{AB,\bs \theta^{\star}}(\bs h^{(i)})\Big]_{i=1,\ldots,p} \std \mathcal{N}(\bs 0,\Pi), \label{asynthetastar}
\end{align} 
where $\Pi$ is the  covariance matrix specified in Theorem~\ref{stasyn}.

As we shall prove below, consistency of the empirical extremogram entails consistent generalised least squares parameter estimates, whereas asymptotic normality of the empirical extremogram entails asymptotically normal generalised least squares parameter estimates.

\begin{definition}[Generalised least squares extremogram estimator (GLSE)] \label{DefGLSE}
Let $\{X(\bs{s}): \bs{s} \in \mathbb{R}^d\}$ be a strictly stationary regularly varying  process, which is observed on $\mathcal{D}_n=\mathcal{F} \times \mathcal{I}_n$ as in \eqref{observed}. 
Let $A$ and $B$ be $\mu$-continuous Borel sets  in $\overline{\mathbb{R}}\backslash\{0\}$ such that $\mu(A)>0$. 
For a sequence $m=m_n \rightarrow \infty$ and $m_n=o(n)$ as $\nto$ define for $\bs\theta\in\Theta$ the column vector 
\begin{align}
\widehat{\bs g}_n(\bs \theta):=\big[\widehat{\rho}_{AB,m_n}(\bs h^{(i)})-\rho_{AB, \bs \theta}(\bs h^{(i)})\big]\trans_{i=1,\ldots,p}. \label{g_hat}
\end{align}
For some non-singular positive definite weight matrix $V(\bs \theta) \in \mathbb{R}^{p \times p}$, the \GLSE\ is defined as 
\begin{align}
\widehat{\bs \theta}_{n,V}:=\argmin\limits_{\bs \theta \in \Theta}\{\widehat{\bs g}_n(\bs \theta)\trans V(\bs \theta) \widehat{\bs g}_n(\bs \theta)\}. \label{GLSE}
\end{align}
\end{definition}

Assumption~\ref{regcondgls} presents a set of conditions, which imply consistency and asymptotic normality of the GLSE.

\begin{assumption} \label{regcondgls}
Assume the situation of Definition~\ref{DefGLSE}. We shall require the following conditions.
\begin{enumerate}[(G1)]
\item Consistency: 
$\widehat{\rho}_{AB,m_n}(\bs h^{(i)}) \stp \rho_{AB,\bs \theta^{\star}}(\bs h^{(i)})$ as $\nto$ for $i=1,\ldots,p.$
\item Asymptotic normality: 
$\sqrt{\dfrac{n^w}{m_n^d}}\widehat{\bs g}_n(\bs \theta^\star) \std \mathcal{N}(\bs 0,\Pi)$ as $\nto$. 
\item \label{glscond1} 
Identifiability condition: For all $\epsilon>0$ there exists some $\delta>0$ such that \\
$\inf\Big\{\sum\limits_{i=1}^p (\rho_{AB,\bs \theta_1}(\bs h^{(i)})-\rho_{AB, \bs \theta_2}(\bs h^{(i)}))^2: \bs \theta^{(1)}, \bs \theta^{(2)} \in \Theta, \Vert{\bs \theta^{(1)}-\bs \theta^{(2)}}\Vert \geq \epsilon\Big\}> \delta.$ 
If the parameter space $\Theta$ is compact, this condition can be replaced by the weaker condition
$$(G3') \quad \sum\limits_{i=1}^p (\rho_{AB,\bs \theta_1}(\bs h^{(i)})-\rho_{AB, \bs \theta_2}(\bs h^{(i)}))^2>0, \quad \bs \theta^{(1)} \neq \bs \theta^{(2)} \in \Theta.$$
\item \label{glscond2}Smoothness condition 1: For all $i=1,\ldots,p$: 

$\rho_{AB,\bs \theta}(\bs h^{(i)})$ has continuous partial derivatives of order $z_1 \geq 0$ w.r.t. $\bs \theta$,
where $z_1=0$ corresponds to $\rho_{AB,\bs \theta}(\bs h^{(i)})$ being continuous in $\bs \theta$.

\item 
\label{glscond3}Smoothness condition 2: 
\begin{enumerate}[(i)]
\item $\sup\limits_{\bs \theta \in \Theta}\{\|V(\bs \theta)\|_M+\|V(\bs \theta)^{-1}\|_M\} < \infty,$ where $\|\cdot\|_M$ is some arbitrary matrix norm.
\item The matrix valued function $V(\bs \theta)$ has continuous derivatives of order $z_2 \geq 0$ w.r.t. $\bs \theta$, where
 $z_2=0$ corresponds to $V(\bs \theta)$ being continuous in $\bs \theta$.
\end{enumerate}
\item
Rank condition:
For $\bs \theta=(\theta_1, \ldots, \theta_k)\in\Theta\subset\R^k$ we denote by 
$\Rho_{AB}(\bs \theta)$ the Jacobian matrix of $(-\rho_{AB,\bs \theta}(\bs h^{(1)}),\ldots,-\rho_{AB,\bs \theta}(\bs h^{(p)}))\trans$; i.e.,
\begin{align}
\Rho_{AB}(\bs \theta)=\begin{pmatrix}
-\frac{\partial}{\partial \theta_{1}}\rho_{AB,\bs \theta}(\bs h^{(1)}) & -\frac{\partial}{\partial \theta_{2}}\rho_{AB,\bs \theta}(\bs h^{(1)}) & \ldots & -\frac{\partial}{\partial \theta_{k}}\rho_{AB,\bs \theta}(\bs h^{(1)}) \\
-\frac{\partial}{\partial \theta_{1}}\rho_{AB,\bs \theta}(\bs h^{(2)}) & -\frac{\partial}{\partial \theta_{2}}\rho_{AB,\bs \theta}(\bs h^{(2)}) & \ldots & -\frac{\partial}{\partial \theta_{k}}\rho_{AB,\bs \theta}(\bs h^{(2)}) \\
\vdots & \vdots & & \vdots \\
-\frac{\partial}{\partial \theta_{1}}\rho_{AB,\bs \theta}(\bs h^{(p)}) & -\frac{\partial}{\partial \theta_{2}}\rho_{AB,\bs \theta}(\bs h^{(p)}) & \ldots & -\frac{\partial}{\partial \theta_{k}}\rho_{AB,\bs \theta}(\bs h^{(p)})
\end{pmatrix} \in \mathbb{R}^{p \times k}. \label{Rho_matrix}
\end{align}
The Jacobian matrix has full rank: $\textnormal{rank}(\Rho_{AB}(\bs \theta^\star))=k$. 
\halmos
\end{enumerate}
\end{assumption}

The proof of the next theorem  can be found in Appendix~\ref{app_2}.

\begin{theorem}[Consistency and asymptotic normality of the GLSE]\label{GLSEcons}
Assume the situation of Definition~\ref{DefGLSE}.
If Assumptions~\ref{regcondgls}(G1) and (G3) hold as well as (G4) and (G5) for $z_1=z_2=0$, respectively, then the GLSE is consistent; i.e.,
\begin{align}
\widehat{\bs \theta}_{n,V} \stp \bs \theta^{\star}, \quad n \rightarrow \infty. \label{GLSEcons2}
\end{align}
If Assumption~\ref{regcondgls}(G2) and (G3) hold as well as (G4) and (G5) for $z_1=z_2=1$, respectively, and the rank condition (G6) holds, then the GLSE is asymptotically normal; i.e., 
\begin{align}\label{asymptGLS}
\sqrt{\frac{n^w}{m_n^d}}(\widehat{\bs \theta}_{n,V}-\bs \theta^{\star}) \std \mathcal{N}(\bs 0, \Pi_V),\quad\nto,
\end{align}
with asymptotic covariance matrix
\begin{align*}
\Pi_V=B(\bs \theta^{\star})\Rho_{AB}(\bs \theta^{\star})\trans [V(\bs \theta^{\star})+V(\bs \theta^{\star})\trans]\, \Pi \, [V(\bs \theta^{\star})+V(\bs \theta^{\star})\trans] \Rho_{AB}(\bs \theta^{\star})B(\bs \theta^{\star}),
\end{align*} 
where $B(\bs \theta^{\star}):=\big(\Rho_{AB}(\bs \theta^{\star})\trans[V(\bs \theta^{\star})+V(\bs \theta^{\star})\trans]\Rho_{AB}(\bs \theta^{\star})\big)^{-1}$ and $\Pi$ is the asymptotic covariance matrix in Eq.~\eqref{asynthetastar}.
\end{theorem}

\brem\label{rem_weight_matrix}
The quality of the GLSE depends on the matrix $V(\bs \theta)$.
Simple choices for the matrix $V(\bs \theta)$ in \eqref{GLSE} are the identity matrix, leading to the ordinary least squares estimator, or some general weight matrix, leading to weighted least squares estimators. 

An asymptotically optimal matrix $V(\bs\theta)$ can be obtained as follows.
Let $\Pi=\Pi(\bs \theta^\star)$ be the asymptotic covariance matrix of the empirical extremogram in Eq.~\eqref{asynthetastar}.
{Assume that $\Pi(\bs \theta^\star)$ has a closed form that depends on the true parameter vector $\bs \theta^\star$ which can be extended to a matrix function $\Pi(\bs \theta)$ on the whole parameter space $\bs\Theta$. }
{Assume also that the inverse $V(\bs\theta)=\Pi^{-1}(\bs \theta)$ exists for all $\bs\theta\in\bs\Theta$ and satisfies the Assumption~\ref{regcondgls}(G5) for $z_2=1$.}
Then, as pointed out in \citet{Lahiri2}, Theorem~4.1, for spatial variogram estimators and in \citet{EKS}, Corollary~2.3, for extreme parameter estimation based on iid random vector observations, the resulting asymptotic covariance matrix $\Pi_V=\Pi_V(\bs \theta^\star)$ of the GLSE in \eqref{asymptGLS} is asymptotically optimal among all valid matrices $V'=V'(\bs \theta)$.
This means that $\Pi_V$ is minimal in the sense that for all valid matrices $V'$, the difference $\Pi_{V'}-\Pi_V$ is positive semidefinite. 
\erem

\section{Estimation of Brown-Resnick space-time processes}\label{s3}

\subsection{Brown-Resnick processes} \label{s31}

We consider a {\textit{strictly stationary Brown-Resnick process}} with spectral representation
\beam\label{BR}
\eta(\bs{s}) = \bigvee\limits_{j=1}^\infty \left\{\xi_j \,  e^{W_j(\bs{s})-\delta(\bs{s})} \right\},\quad \bs{s}\in\R^{d},
\eeam
where 
$\{\xi_j : j\in\N\}$ are points of a Poisson process on $[0,\infty)$ with intensity $\xi^{-2}d\xi$, the {\textit{dependence function}} $\delta$ is nonnegative and conditionally negative definite, and 
$\{W_j(\bs{s}): \bs{s} \in \mathbb{R}^{d}\}$ are independent replicates of a Gaussian process $\{W(\bs s): \bs{s} \in \mathbb{R}^{d}\}$ with stationary increments,
$W(\bs{0})=0$, $\E [W(\bs{s})]=0$ and covariance function
$$\cov[W(\bs{s}^{(1)}),W(\bs{s}^{(2)})] 
=\delta(\bs{s}^{(1)})+\delta(\bs{s}^{(2)}) -\delta(\bs{s}^{(1)}-\bs{s}^{(2)}).$$
{Spectral representations of max-stable processes go back to \citet{deHaan} and \citet*{Gine}, the specific representation \eqref{BR} to \citet{Brown} in a time series context, to \citet{Schlather2} in a spatial and to \citet{Steinkohl} in a space-time setting.}
The univariate margins of the process $\eta$ follow standard unit Fr{\'e}chet distributions. Non-stationary Brown-Resnick models have recently been discussed and fitted to data in \citet{Asadi2015}, \citet{Engelke1},  {and \citet{Huser3}}.

There are various quantities to describe the dependence in \eqref{BR}, where explicit expressions can be derived:
\begin{enumerate}[$\bullet$]
\item 
In geostatistics, the {dependence function} $\delta$ is termed the \textit{semivariogram} of the process $\{W(\bs s): \bs s \in \mathbb{R}^d\}$ based on the fact that for $\bs{s}^{(1)}, \bs{s}^{(2)} \in \R^{d}$, 
$$\var[W(\bs{s}^{(1)})-W(\bs{s}^{(2)})]=2\delta(\bs{s}^{(1)}-\bs{s}^{(2)}).$$
\item 
For $\bs{h} \in \mathbb{R}^{d}$, the {tail dependence coefficient} is given by (see e.g. \citet*{Steinkohl}, Section 3)
\beam\label{tdcbrownresnick}
\rho_{(1,\infty)(1,\infty)}(\bs h)=\lim\limits_{n \rightarrow \infty} \mathbb{P} \Big(\eta(\bs h)>n \,\Big|\, \eta(\bs 0)>n\Big)
= 2\Big(1-\Phi\Big(\sqrt{\frac{\delta(\bs{h})}{2}}\Big)\Big),
\eeam
where $\Phi$ denotes the standard normal distribution function.
\item 
For $D=\{\bs{s}^{(1)},\ldots,\bs{s}^{(|D|)}\} $ and $\bs{y}=(y_1,\ldots,y_{|D|})>\bs 0$ the finite-dimensional margins are given by
\begin{align}
\mathbb{P}(\eta(\bs{s}^{(1)})\le y_1, \eta(\bs{s}^{(2)})\le y_2,\cdots, \eta(\bs{s}^{(|D|)})\le y_{|D|}) =  \exp\{-V_{D}(\bs y)\}. \label{exponentmeasureD}
\end{align}
Here $V_D$ denotes the \textit{exponent measure} (cf. \citet{Beirlant}, Section~8.2.2), which is homogeneous of order -1 and depends solely on the dependence function $\delta$.
For $D=\{\bs s,\bs s+\bs h\}$ where $\bs s \in \bbr^d$ and $\bs h \in \bbr^d$ is some fixed lag vector, we get (cf. \citet{Steinkohl}, Section~3)
\beam\label{expom}
V_2(y_1,y_2)=V_2(\bs h;y_1,y_2)=V_D(y_1,y_2) = \frac1{y_1}\wt\Phi\Big(\frac{y_2}{y_1}\Big) + \frac1{y_2}\wt\Phi\Big(\frac{y_1}{y_2}\Big),\quad y_1,y_2>0,
\eeam
with 
\beam\label{tildephi}
\wt\Phi\Big(\frac{x}{y}\Big)=\wt\Phi\Big(\bs h;\frac{x}{y}\Big):=\Phi\Big(\frac{\log(x/y)}{\sqrt{2 \delta(\bs h)}}+\sqrt{\frac{\delta(\bs h)}{2}}\Big),\quad x,y>0.
\eeam
\item
For $\bs h \in \mathbb{R}^{d}$ and sets $A=(\underline{A},\overline{A})$ and $B=(\underline{B},\overline{B})$ with $0<\underline{A}<\overline{A}\leq\infty$ and $0<\underline{B}<\overline{B}\leq\infty$, the extremogram~\eqref{extremo} is given by (see \cite{buhl3}, Eq.~(A.1))
\begin{align}
\rho_{AB}(\bs h)=&\frac{\underline{A}\overline{A}}{\overline{A}-\underline{A}}\Big(-V_2(\overline{A},\overline{B})+V_2(\overline{A},\underline{B})+V_2(\underline{A},\overline{B})-V_2(\underline{A},\underline{B})\Big) \label{extremBR}
\end{align}
for $V_2$ as in \eqref{expom}.
For $A=(\underline{A},\infty)$ and $B=(\underline{B},\infty)$ we get formula (31) of \citet{cho}:
\beam\label{cho}
\rho_{AB}(\bs h) = \underline{A}\Big\{\underline{A}^{-1}\Big(1-\wt \Phi\Big(\underline{B}/ \underline{A}\Big)\Big)+ \underline{B}^{-1}\Big(1-\wt \Phi\Big(\underline{A}/ \underline{B}\Big)\Big)\Big\}.
\eeam
\item 
The {\textit{extremal coefficient}} $\xi_{D}$ (see \cite{Beirlant}, Section~8.2.7) for any finite set $D\subset\R^{d}$ is defined as
$$\mathbb{P}(\eta(\bs{s}^{(1)})\le y, \eta(\bs{s}^{(2)})\le y,\cdots, \eta(\bs{s}^{(|D|)})\le y) =  \exp\{-\xi_{D}/y\},\quad y>0;$$
i.e., $\xi_{D}=V_{D}(1,\ldots,1).$ 
If $|D|=2$ and $\bs h = \bs{s}^{(1)}-\bs{s}^{(2)}$, then 
\begin{align}
\xi_{D} = 2-\rho_{(1,\infty)(1,\infty)}(\bs h) = 2\Phi\Big(\sqrt{\frac{\delta(\bs{h})}{2}}\Big), \label{reltdec}
\end{align}
where the first identity holds in general (cf. \citet{Beirlant}, Section 9.5.1), and the last one by \eqref{tdcbrownresnick}.
\end{enumerate}

Our aim is to fit a parametric extremogram model 
of a Brown-Resnick process \eqref{BR} based on observations given in $\cald_n=\calf\times\cali_n$ as in \eqref{observed}.
This approach is semiparametric in the sense that we first compute (possibly bias corrected) empirical estimates \eqref{biascorrectedextremo} of the extremogram ${\rho}_{AB}(\bs{h})$ for different $\bs h \in \mathcal{H}$, and fit a parametric model $\rho_{AB,\bs \theta}(\bs h)$ by GLSE 
to the empirical extremogram. 
For sets $A=B=(\underline{A},\infty)$ with $\underline{A}>0$, this yields an estimator of the dependence function, since by \eqref{tildephi} and  \eqref{cho} there is a one-to-one relation between extremogram and dependence function.

\subsection{Asymptotic properties of the empirical extremogram of a Brown-Resnick process}\label{s33}

Let $\{\eta(\bs s): \bs s \in \mathbb{R}^d\}$ be a strictly stationary Brown-Resnick process as in \eqref{BR} with some valid (i.e., nonnegative and conditionally negative definite) dependence function $\delta$.
Before investigating the asymptotic properties of the GLSE, we state sufficient conditions for $\delta$ so that the regularity conditions of Theorem~\ref{stasyn} are satisfied. 

\begin{theorem} \label{Prop_regcond2}
Let $\{\eta(\bs s): \bs s \in \mathbb{R}^d\}$ be a strictly stationary Brown-Resnick process as in \eqref{BR}, observed on $\mathcal{D}_n=\mathcal{F} \times \mathcal{I}_n$  as in \eqref{observed}. 
Let $\calh=\{\bs h^{(1)},\ldots,\bs h^{(p)}\} \subset \mathcal{B}(\bs 0,\ga)$ for some $\ga>0$ be a set of observed lag vectors.
Assume sequences
\beam\label{BRsequences}
m_n,r_n \rightarrow \infty, \quad m_n^d/n^w \to 0,\quad r_n^w/m_n^d \to 0, \quad m_n^{2d}r_n^{2w}/n^w \to 0,\quad n\to\infty.
\eeam
Writing $\bs v=(\bs v_{\mathcal{F}},\bs v_{\cali}) \in \bbr^q \times \bbr^w$ according to the fixed and increasing domains, assume that the dependence function $\delta$ satisfies for  arbitrary fixed finite set $L \subset \mathbb{Z}^q:$ 
\begin{enumerate}[(A)]
\item 
$m_n^d \sum\limits_{z > r_n} z^{w-1} \exp\Big\{-\frac{1}{4} \inf\limits_{\bs v \in L \times \mathbb{Z}^w: \|\bs v_{\mathcal{I}}\|\geq z} \delta(\bs v)\Big\} \rightarrow 0$ as $n \rightarrow \infty$.
\item 
$m_n^{d/2} n^{(3w)/2} \exp\Big\{-\frac{1}{4} \inf\limits_{\bs v \in L \times \mathbb{Z}^w: \|\bs v_{\mathcal{I}}\| >r_n}\delta(\bs v)\Big\} \rightarrow 0$ as $n \rightarrow \infty.$
\end{enumerate}
Then conditions (M1)-(M4) of Theorem~\ref{stasyn} are satisfied, and the empirical extremogram $\wh{\rho}_{AB,m_n}$ defined in \eqref{EmpEst} sampled at lags in $\calh$ 
and centred by the pre-asymptotic extremogram $\rho_{AB,m_n}$ given in \eqref{preasymptotic}, is asymptotically normal; i.e.,
\begin{equation}\label{asyspace_BR1}
\sqrt{\frac{n^w}{m_n^d}}\Big[\widehat{\rho}_{AB,m_n}(\bs h^{(i)}) -\rho_{AB,m_n}(\bs h^{(i)})\Big]_{i=1,\ldots,p} \std \mathcal{N}(\bs 0,\Pi),\quad \nto,
\end{equation}
where the covariance matrix $\Pi$ is specified in Theorem~\ref{stasyn}.
\end{theorem}

\begin{proof}
First note that, since all finite-dimensional distributions are max-stable distributions with standard unit Fr\'echet margins, they are multivariate regularly varying. We first show (M3). Let $\epsilon > 0$
and fix $\bs \ell_{\mathcal{F}} \in \mathbb{R}^q$. 
For $\ga>0$ define the set 
$$L_{\ga}(\bs \ell_{\mathcal{F}},\bs \ell_{\mathcal{I}}):=\{\bs s_1-\bs s_2: \bs s_1 \in \mathcal{B}(\bs 0,\ga), \bs s_2 \in \mathcal{B}((\bs \ell_{\mathcal{F}},\bs \ell_{\mathcal{I}}),\ga)\}.$$
Note that, writing $\bs s_1=(\bs f_1,\bs i_1)$ and $\bs s_2=(\bs f_2,\bs i_2) \in \bbr^q \times \bbr^w$ according to the fixed and increasing domains as before, it can be decomposed into $L_{\ga}(\bs \ell_{\mathcal{F}},\bs \ell_{\mathcal{I}})=L_{\ga}^{(1)} \times L_{\ga}^{(2)}(\bs \ell_{\cali})$ where $L_{\ga}^{(1)}:=\{\bs f_1-\bs f_2: \bs s_1 \in \mathcal{B}((\bs 0,\bs 0),\ga), \bs s_2 \in \mathcal{B}((\bs \ell_{\mathcal{F}},\bs 0),\ga)\}$, which is independent of $\bs \ell_{\cali}$, and $L_{\ga}^{(2)}(\bs \ell_{\cali}):=\{\bs i_1-\bs i_2: \bs s_1 \in \mathcal{B}((\bs 0,\bs 0),\ga), \bs s_2 \in \mathcal{B}((\bs \ell_{\mathcal{F}},\bs \ell_{\cali}),\ga)\}$. 
Then, recalling that $a_m=m_n^{d}$, and using a second order Taylor expansion as in the proof of {Theorem~4.3} of \citet{Steinkohl3}, we have as $n \rightarrow \infty$,
\begin{align*}
&\mathbb{P}(\max\limits_{\bs s \in \mathcal{B}(\bs 0,\gamma)} \eta(\bs s) > \epsilon a_m,\max\limits_{\bs s' \in \mathcal{B}((\bs \ell_{\mathcal{F}},\bs \ell_{\mathcal{I}}),\gamma)} \eta(\bs s') > \epsilon a_m)\\
\leq & \sum\limits_{\bs s \in \mathcal{B}(\bs 0,\ga)}\sum\limits_{\bs s' \in \mathcal{B}((\bs \ell_{\mathcal{F}},\bs \ell_{\mathcal{I}}),\ga)}\mathbb{P}(\eta(\bs s) > \epsilon m_n^{d}, \eta(\bs s') > \epsilon m_n^{d}) \\
=& \sum\limits_{\bs s \in \mathcal{B}(\bs 0,\ga)}\sum\limits_{\bs s' \in \mathcal{B}((\bs \ell_{\mathcal{F}},\bs \ell_{\mathcal{I}}),\ga)} \Big( 1-2 \exp\Big\{-\frac{1}{\epsilon m_n^{d}}\Big\}+\exp\Big\{-\frac{2}{\epsilon m_n^{d}} \Phi\Big(\sqrt{\frac{\delta(\bs s-\bs s')}{2}}\Big)\Big\}\Big)\\
\le &\frac{2|\mathcal{B}(\bs 0,\gamma)|^2}{\epsilon m_n^{d}} \Big(1-\Phi\Big(\Big(\frac12{\inf_{\bs v \in L_{\ga}(\bs \ell_{\mathcal{F}},\bs \ell_{\mathcal{I}})} \delta(\bs v)}\Big)^{1/2}\Big)\Big) + \mathcal{O}\Big(\frac{1}{m_n^{2d}}\Big).
\end{align*}
Therefore, 
\begin{align*}
&\limsup\limits_{n \rightarrow \infty}  \sum\limits_{\bs \ell_{\mathcal{I}} \in \mathbb{Z}^w \atop k< \|\bs \ell_{\mathcal{I}}\| \leq r_n} m_n^{d} \mathbb{P}(\max\limits_{\bs s \in \mathcal{B}(\bs 0,\ga)} \eta(\bs s)>\epsilon a_m, \max\limits_{\bs s' \in \mathcal{B}((\bs \ell_{\mathcal{F}},\bs \ell_{\mathcal{I}}),\ga)} \eta(\bs s')>\epsilon a_m) \\
\leq & 2 |\mathcal{B}(\bs 0,\ga)|^2 \limsup\limits_{n \rightarrow \infty} \sum\limits_{\bs \ell_{\mathcal{I}} \in \mathbb{Z}^w \atop k< \|\bs \ell_{\mathcal{I}}\| \leq r_n} 
\Big\{\frac{1}{\epsilon} \Big(1-\Phi\Big(\Big(\frac12 \inf_{\bs v \in L_{\ga}(\bs \ell_{\mathcal{F}},\bs \ell_{\mathcal{I}})} \delta(\bs v)\Big)^{1/2}\Big)\Big)+\mathcal{O}\Big(\frac{1}{m_n^{d}}\Big)\Big\}.
\end{align*}
Since the number of grid points $\bs \ell_{\cali}$ in $\mathbb{Z}^w$ with norm $\|\bs \ell_{\mathcal{I}}\|=z$ is of order $\mathcal{O}(z^{w-1})$, there exists a positive constant $C$ such that the right hand side can be bounded from above by
\begin{align*}
& 2 C |\mathcal{B}(\bs 0,\ga)|^2 \limsup\limits_{n \rightarrow \infty} \sum\limits_{\atop k< z \leq r_n} \Big\{ \frac{z^{w-1}}{\epsilon} \Big(1-\Phi\Big(\Big(\frac{1}{2}\inf_{\bs v \in L_{\ga}(\bs \ell_{\mathcal{F}},\bs \ell_{\mathcal{I}}): \bs \ell_{\mathcal{I}} \in \mathbb{Z}^w, \|\bs \ell_{\mathcal{I}}\|=z} \delta(\bs v)\Big)^{1/2}\Big)\Big)\\&\quad+\mathcal{O}\Big(\frac{z^{w-1}}{m_n^{d}}\Big)\Big\} \\
&\leq\frac{2C |\mathcal{B}(\bs 0,\ga)|^2}{\epsilon} \limsup\limits_{n \rightarrow \infty} \sum\limits_{\atop k< z < \infty} \Big\{z^{w-1} \Big(\exp\Big\{-\frac{1}{4}\inf_{\bs v \in L_{\ga}(\bs \ell_{\mathcal{F}},\bs \ell_{\mathcal{I}}): \bs \ell_{\mathcal{I}} \in \mathbb{Z}^w, \|\bs \ell_{\mathcal{I}}\|= z} \delta(\bs v)\Big\}\Big)\Big\}\\&\quad+\mathcal{O}\Big(\frac{r_n^w}{m_n^{d}}\Big)\\
&\leq\frac{2C |\mathcal{B}(\bs 0,\ga)|^2}{\epsilon} \limsup\limits_{n \rightarrow \infty} \sum\limits_{\atop k< z < \infty} \Big\{z^{w-1} \Big(\exp\Big\{-\frac{1}{4}\inf_{\bs v \in L_{\ga}^{(1)} \times \bbz^w: \|\bs v_{\mathcal{I}}\| \geq z - \ga} \delta(\bs v)\Big\}\Big)\Big\}\\&\quad+\mathcal{O}\Big(\frac{r_n^w}{m_n^{d}}\Big),
\end{align*}
where we have used in the second last step that $1-\Phi(x) \leq \exp\{-x^2/2\}$ for $x > 0$ and in the last step the decomposition $L_{\ga}(\bs \ell_{\mathcal{F}},\bs \ell_{\mathcal{I}})=L_{\ga}^{(1)} \times L_{\ga}^{(2)}(\bs \ell_{\cali})$. 
By condition (A), since we can neglect the constant $\ga$, we have 
$$\lim\limits_{k \rightarrow \infty} \sum\limits_{\atop k< z < \infty} z^{w-1}\exp\Big\{-\frac{1}{4}\inf_{\bs v \in L_{\ga}^{(1)} \times \bbz^w: \|\bs v_{\mathcal{I}}\| \geq z - \ga} \delta(\bs v)\Big\}=0.$$
Together with $r_n^{w}=o(m_n^d)$ as $n \rightarrow \infty$, this implies that 
$$\lim\limits_{k \rightarrow \infty} \limsup\limits_{n \rightarrow \infty} \sum_{k< z \leq r_n} \Big\{z^{w-1} \Big(\exp\Big\{-\frac{1}{4}\inf_{\bs v \in L_{\ga}^{(1)} \times \bbz^w: \|\bs v_{\mathcal{I}}\| \geq z - \ga} \delta(\bs v)\Big\}\Big)\Big\}+\mathcal{O}\Big(\frac{r_n^{w}}{m_n^d}\Big)= 0.$$

Next we prove (M1) and (M4i)-(M4iii). {To this end we bound the $\alpha$-mixing coefficients $\alpha_{k_1,k_2}(\cdot)$ for $k_1,k_2 \in \bbn$ of  $\{\eta(\bs s): \bs s \in \mathbb{R}^d\}$ with respect to $\R^w$, which are  defined in \eqref{alpha_balls}.} Observe that $d(\Lambda_1,\Lambda_2)$ for sets $\Lambda_i \subset \mathbb{Z}^w$ as in Definition~\ref{mixing} can only get  large within the increasing domain. 
Define the set 
$$L_{\mathcal{F}}:=\{\bs s_1-\bs s_2: \bs s_1,\bs s_2 \in \mathcal{F}\}.$$
We use Eq.~\eqref{reltdec}, as well as \citet{Dombry}, Eq.~(3) and Corollary~2.2 to obtain 
\begin{align}
\alpha_{k_1,k_2}(z) &\leq 2 \sup\limits_{d(\Lambda_1,\Lambda_2) \geq z} \sum\limits_{\bs s_1 \in \mathcal{F} \times \Lambda_1} \sum\limits_{\bs s_2 \in \mathcal{F} \times \Lambda_2} \rho_{(1,\infty)(1,\infty)}(\bs s_1-\bs s_2) \notag \\
&\leq 2k_1k_2 |\mathcal{F}|^2 \sup\limits_{\bs v  \in L_{\mathcal{F}} \times \mathbb{Z}^w: \|\bs v_{\mathcal{I}}\| \geq z} \rho_{(1,\infty)(1,\infty)}(\bs v) \notag \\
& = 4 k_1k_2 |\mathcal{F}|^2 \Big(1-\Phi \Big(\Big(\frac12\,{\inf\limits_{\bs v  \in L_{\mathcal{F}} \times \mathbb{Z}^w: \|\bs v_{\mathcal{I}}\| \geq z}\delta(\bs v)}\Big)^{\frac{1}{2}}\Big)\Big) \notag\\
& \leq 4k_1k_2 |\mathcal{F}|^2 \exp\Big\{-\frac{1}{4}\inf\limits_{\bs v \in L_{\mathcal{F}} \times \mathbb{Z}^w: \|\bs v_{\mathcal{I}}\| \geq z}\delta(\bs v)\Big\}. \label{alphabound}
\end{align}
By condition (A) we have $\alpha_{k_1,k_2}(z) \rightarrow 0$, since necessarily $\inf\limits_{\bs v  \in L_{\mathcal{F}} \times \mathbb{Z}^w: \|\bs v_{\mathcal{I}}\| \geq z}\delta(\bs v) \rightarrow \infty$ as $z \rightarrow \infty$ and, therefore, the process $\{\eta(\bs s): \bs s \in \mathbb{R}^d\}$ is $\alpha$-mixing; i.e., (M1) holds.
We continue by estimating
\begin{align*}
& m_n^d \sum\limits_{\bs \ell \in \mathbb{Z}^{w}: \|\bs \ell\| > r_n} \alpha_{1,1}(\|\bs \ell\|) 
 \, \leq \,  C  m_n^d  \sum\limits_{z > r_n} z^{w-1} \alpha_{1,1}(z) \\
 \leq \, & 4C |\mathcal{F}|^2 m_n^d \sum\limits_{z > r_n} z^{w-1} \exp\Big\{-\frac{1}{4}\inf\limits_{\bs v  \in L_{\mathcal{F}} \times \mathbb{Z}^w: \|\bs v_{\mathcal{I}}\| \geq z}\delta(\bs v)\Big\} \to 0,\quad n\to\infty, 
\end{align*} 
by condition (A).
This shows (M4i). 
Similarly, it can be shown that (M4ii) holds, if (A) is satisfied. 
Finally, we show (M4iii). Using Eq.~\eqref{alphabound}, we find
\begin{align*}
&m_n^{d/2} n^{w/2} \alpha_{1, n^w}(r_n)\leq 4  m_n^{d/2} n^{(3w)/2} |\mathcal{F}|^2 \exp\Big\{-\frac{1}{4}\inf\limits_{\bs v  \in L_{\mathcal{F}} \times \mathbb{Z}^w: \|\bs v_{\mathcal{I}}\| \geq r_n}\delta(\bs v)\Big\} 
\rightarrow 0
\end{align*}
as $n \rightarrow \infty$ because of condition (B).
\end{proof}

The following is an immediate corollary of Theorem~\ref{Prop_regcond2}.

\begin{corollary}\label{sufficient_BR}
Assume the setting of Theorem~\ref{Prop_regcond2}.
 Suppose that the dependence function $\delta$ satisfies for positive constants $C$ and $\alpha$, and for an arbitrary norm $\|\cdot\|$ on $\bbr^w$,  
 \begin{align}\delta(\bs v)
\geq C 
\|\bs v_{\cali}\|^{\alpha} \label{suff_cond_BR}
\end{align} for every $\bs v=(\bs v_{\calf},\bs v_{\cali}) \in L \times \bbz^w$, where $L \subset \bbz^q$ is arbitrary, but fixed. In particular, $\delta(\bs v)\to \infty$ if 
$\|\bs v_{\cali}\| \to \infty$. With $m_n=n^{\beta_1}$ and $r_n=n^{\beta_2}$ with $\beta_1 \in (0,w/(2d))$ and $\beta_2 \in \min\{\beta_1d/w;1/2-\beta_1d/w\}$, the conditions of Theorem~\ref{Prop_regcond2} are satisfied for $\{\eta(\bs s): \bs s \in \bbr^d\}$ and we conclude
\begin{equation}
n^{(w-d\beta_1)/2}\Big[\widehat{\rho}_{AB,m_n}(\bs h^{(i)}) -\rho_{AB,m_n}(\bs h^{(i)})\Big]_{i=1,\ldots,p} \std \mathcal{N}(\bs 0,\Pi),\quad \nto.\label{asyspace_BR}
\end{equation}
\end{corollary}

\begin{proof}
 Due to equivalence of norms on $\bbr^w$ we will make no difference between the norm in \eqref{suff_cond_BR} and the one used in Theorem~\ref{Prop_regcond2}. 
 Clearly the sequences $m_n$ and $r_n$ satisfy the requirements 
$m_n,r_n \rightarrow \infty$, $m_n^d/n^w \to 0$, $r_n^w/m_n^d \to 0$ and $m_n^{2d}r_n^{2w}/n^w \to 0$ as $n\to\infty.$ 
We have for $z>0$,
\begin{align*}
\exp\Big\{-\frac{1}{4} \inf\limits_{\bs v \in L \times \mathbb{Z}^w: \|\bs v_{\mathcal{I}}\| >z}\delta(\bs v)\Big\} 
& \leq \exp\Big\{-\frac{1}{4} \inf\limits_{\bs v \in L \times \mathbb{Z}^w: \|\bs v_{\mathcal{I}}\| >z} C \|\bs v_{\cali}\|^{\alpha} \Big\} \\
& \leq \exp\Big\{-\frac{Cz^{\alpha}}{4}  \Big\}.
\end{align*}
Condition (B) of Theorem~\ref{Prop_regcond2} is satisfied since 
\begin{align*}
n^{(\beta_1d)/2} n^{(3w)/2} \exp\Big\{-\frac{Cr_n^\alpha}{4}  \Big\} &=n^{(\beta_1d)/2} n^{(3w)/2} \exp\Big\{-\frac{Cn^{\beta_2\alpha}}{4} \Big\}\\
&=\exp\Big\{-\frac{Cn^{\beta_2\alpha}}{4} +\frac{\beta_1d+3w}{2}\log(n)\Big\}\to 0,\,\, \nto.
\end{align*}
Condition (A) holds since by Lemma~A.3 of \citet{Steinkohl3}, there is a positive constant $K$ such that for sufficiently large $n$  the sequence $z^{w-1} \exp\{-Cz^{\alpha}/4\}$ is  decreasing for $z \geq r_n$,
\begin{align*}
m_n^d \sum\limits_{z > r_n} z^{w-1} \exp\Big\{-\frac{Cz^{\alpha}}{4} \Big\} &\leq K m_n^d r_n^w\exp\Big\{-\frac{Cr_n^{\alpha}}{4} \Big\} \\
&=K \exp\Big\{-\frac{Cn^{\beta_2\alpha}}{4}  + (\beta_1d+\beta_2w) \log(n) \Big\}  \rightarrow 0,\,\, \nto.
\end{align*}
\end{proof}

With the particular choice of sequences $m_n=n^{\beta_1}$ and $r_n=n^{\beta_2}$ given in Corollary~\ref{sufficient_BR}, we are in the setting of Remark~\ref{summary_cases}. Hence, in addition to the CLT~\eqref{asyspace_BR}, we obtain the CLT~\eqref{ohnebias} of the empirical extremogram centred by the true one and the CLT~\eqref{asynbias} corresponding to the bias corrected estimator.

{\brem 
\begin{enumerate}[(i)]
\item Corollary~\ref{sufficient_BR} requires the dependence function $\delta$ of the Brown-Resnick process to be unbounded. This requirement is not satisfied, for example, by the Schlather model or extremal-$t$-models, which do not capture possible extremal independence between two process values; see for example \citet{DavisonPadoanRibatet}, Section~6.1 and \citet{Opitz}, Section~4.
\item Other prominent max-stable processes that satisfy the conditions of Theorem~\ref{stasyn} are the max-moving average processes (see Example~4.6 of \citet{buhl3}) or special cases of the random set model in \citet{Huser}. 
\end{enumerate}
\erem}

\subsection{Space-time Brown-Resnick processes: different models for the extremogram} \label{s32}

We explore the semiparametric estimation for strictly stationary  Brown-Resnick processes  in their space-time form $\{\eta(\bs s,t): \bs s \in \mathbb{R}^{d-1}, t \in [0,\infty)\}$. 
For three classes of parametric models for the dependence function $\delta_{\bs \theta}$ we prove that the GLSE is consistent and asymptotically normal.   

Note that by Eq.~\eqref{cho} every model $\{\delta_{\bs \theta}: \bs \theta \in \Theta\}$ for the dependence function 
 yields a model $\{\rho_{AB,\bs \theta}: \bs \theta \in \Theta\}$ for its space-time extremogram. 
Moreover, the extremogram \eqref{cho} is always of the same form, and only $\wt\Phi$ in \eqref{tildephi} changes with the model. 
We consider  three different model classes, which together cover a large field of environmental applications such as the modelling of extreme precipitation (cf. \cite{Steinkohl}, \cite{buhl1}, \cite{fondeville}, \cite{Steinkohl3}), extreme wind speed (cf. \cite{Engelke1}) or extremes on river networks (cf. \cite{Asadi2015}), provided they are valid (i.e., nonnegative and conditionally negative definite) dependence functions in the considered metric. \\[2mm]
\textbf{(I) \, Fractional space-time model.} \\
 \citet{Steinkohl} introduce the spatially isotropic model
\beam\label{varioStein}
\delta_{\bs \theta}(\bs h, u)=C_1 \Vert{\bs h}\Vert^{\alpha_1}+C_2 |u|^{\alpha_{2}}, \quad (\bs h,u) \in \mathbb{R}^{d},
\eeam 
with parameter vector 
$$\bs \theta\in \left\{(C_1,C_2,\alpha_1,\alpha_2): C_1,C_2 \in (0, \infty), \alpha_1,\alpha_2 \in (0,2]\right\}.$$ 
The isotropy assumption, where \eqref{varioStein} depends on the norm of the spatial lag $\bs h$, can be relaxed in a natural way by introducing \textit{geometric anisotropy}.
We only discuss the case $d-1=2$, but the approach is easily transferable to higher dimensions. 
Let $\varphi \in [0,\pi/2)$ be a rotation angle and $R=R(\varphi)$ a rotation matrix, and $T$ a dilution matrix with $c>0$; more precisely,
$$R=\begin{pmatrix}
\cos \varphi & -\sin \varphi \\ \sin \varphi & \cos \varphi
\end{pmatrix}
\quad\mbox{and}\quad
T=\begin{pmatrix}
1 & 0 \\ 0 & c 
\end{pmatrix}
.
$$ 
The geometrically anisotropic model is then given by 
\begin{align}
\wt\delta_{\bs{\wt\theta}}(\bs h,u)=\delta_{\bs \theta}(A\bs h,u), \quad (\bs h,u) \in \mathbb{R}^{d},\label{varioStein_aniso}
\end{align}
where  $A=TR$ is the transformation matrix.
The parameter vector of the transformed model is 
$$\bs{\wt\theta}\in \left\{(C_1,C_2,\alpha_1,\alpha_2,c,\varphi): C_1,C_2 \in (0, \infty), \alpha_1,\alpha_2 \in (0,2],c>0, \varphi \in [0,\pi/2)\right\}.$$ 
For more details about geometric anisotropy see \citet{Blanchet1}, Section~4.2, \cite{Steinkohl}, Section~4.2, or \citet{Engelke1}, Section~5.2.
 \\[2mm]
 \textbf{(II) \, Spatial anisotropy along orthogonal spatial directions} \\
\citet{buhl1} generalize the fractional isotropic model \eqref{varioStein} to
\beam\label{vario0}
\delta_{\bs \theta}(\bs h,u) = \sum_{j=1}^{d-1} C_j |h_j|^{\alpha_j}+C_{d} |u|^{\alpha_{d}}, \quad (\bs h,u) \in \mathbb{R}^{d}
\eeam 
with parameter vector 
 $${\bs \theta}\in \left\{(C_j,\alpha_j, j=1,\ldots,d): C_j \in (0, \infty), \alpha_j \in (0,2], j=1, \ldots, d\right\}.$$ 
It is more flexible than the isotropic model (I) as it allows for different rates of decay of extreme dependence along the axes of a $d$-dimensional spatial grid. 
Arbitrary principal orthogonal directions can be introduced by a rotation matrix $R$ as introduced for the isotropic model in (I), here described for the case $d-1=2$: 
\begin{align}\label{vario0rot}
\wt\delta_{\bs {\wt\theta}}(\bs h,u)=
C_1 |h_1\cos \varphi- h_2\sin\varphi|^{\alpha_1}+C_2 |h_1\sin\varphi +h_2\cos \varphi |^{\alpha_2}+C_3 |u|^{\alpha_3},  (\bs h,u) \in \R^3.
\end{align} 
The new parameter vector is 
$$\bs {\wt\theta}\in \left\{(C_1,C_2,C_3,\alpha_1,\alpha_2,\alpha_3, \varphi): C_j \in (0, \infty), \alpha_j \in (0,2], j=1,2,3, \varphi \in [0,\pi/2)\right\}.
$$
In \cite{buhl1} this model is applied to extreme precipitation in Florida and, according to a specifically developed goodness-of-fit method, performs extremely well.
\\[2mm]
\textbf{(III) Time-shifted Brown-Resnick processes} \\
With the goal to allow for some influence of the spatial dependence from previous values of the process we time-shift the Gaussian processes in  the definition of the Brown-Resnick model~\eqref{BR}. 
For $\bs \tau=(\tau_1,\tau_2) \in \mathbb{R}^{d-1}$ define 
$$W^{(\bs \tau)}(\bs s,t):=W(\bs s-t \bs \tau,t).$$
Then $\{W^{(\bs \tau)}(\bs s,t): \bs s \in \mathbb{R}^{d-1}, t \in [0,\infty)\}$ is also a centred Gaussian process starting in 0 with stationary increments: for $(\bs s^{(1)},t^{(1)}), (\bs s^{(2)},t^{(2)}) \in \mathbb{R}^{d-1} \times [0,\infty)$, because of the stationary increments of $\{W(\bs s,t)\}$, where $\stackrel{d}{=}$ stands for equality in distribution,
\begin{align*}
W^{(\bs \tau)}(\bs s^{(1)},t^{(1)})-W^{(\bs \tau)}(\bs s^{(1)},t^{(1)}) & \stackrel{d}{=}W(\bs s^{(1)}- \bs s^{(2)}-(t^{(1)}-t^{(2)})\bs \tau,t^{(1)}-t^{(2)}) \\
&=W^{(\bs \tau)}(\bs s^{(1)}-\bs s^{(2)},t^{(1)}-t^{(2)}),
\end{align*}
The corresponding time-shifted dependence function is given by
$$\delta^{(\bs \tau)}(\bs s,t):=\frac{\mathbb{V}ar[W^{(\bs \tau)}(\bs s,t)-W^{(\bs \tau)}(\bs 0,0)]}{2} =\frac{\mathbb{V}ar[W(\bs s-t\bs \tau,t)-W(\bs 0,0)]}{2}=\delta(\bs s-t \tau,t),$$ 
which yields the covariance function
\begin{align*}
\mathbb{C}ov[W^{(\bs \tau)}(\bs s^{(1)},t^{(1)}),&W^{(\bs \tau)}(\bs s^{(2)},t^{(2)})]=\\
&\delta^{(\bs \tau)}(\bs s^{(1)},t^{(1)})+\delta^{(\bs \tau)}(\bs s^{(2)},t^{(2)})-\delta^{(\bs \tau)}(\bs s^{(1)}-\bs s^{(2)},t^{(1)}-t^{(2)}).
\end{align*}
By Theorem~10 of \citet{Schlather2} the process 
\begin{align}
\eta^{(\bs \tau)}(\bs s,t):=\bigvee\limits_{i=1}^{\infty} \xi_i \ex{W_i^{(\bs \tau)}(\bs s,t)-\delta^{(\bs \tau)}(\bs s,t)}=\eta(\bs s-t\bs \tau,t), \quad (\bs s,t) \in \mathbb{R}^{d-1} \times [0, \infty),
\end{align}
defines a strictly stationary space-time Brown-Resnick process. 

This method does not depend on the specific dependence function:
every Brown-Resnick process $\{\eta(\bs s,t): (\bs s,t) \in \mathbb{R}^{d-1}, t \in [0,\infty)\}$ with dependence function  $\{\delta_{\bs \theta}, {\bs \theta} \in \Theta\}$ results in a time-shifted Brown-Resnick process with dependence function \\$\{\delta^{(\bs \tau)}_{\bs{\theta}}, {\bs \theta} \in \Theta, \bs \tau\in\R^{d-1}\}$. 
To give an example, for the Brown-Resnick process (II) without rotation, the parametrised time-shifted dependence function is given by
\begin{align}
\delta_{\bs \theta}^{(\bs \tau)}(\bs h,u)=\sum\limits_{i=1}^{d-1} C_i|h_i-u \tau_i|^{\alpha_i}+C_{d}|u|^{\alpha_{d}}, \quad (\bs h,u) \in \mathbb{R}^d \label{BKLU_timeshift}
\end{align}
with parameter vector  $$({\bs \theta},\bs \tau) \in \left\{(C_j,\alpha_j, j=1,\ldots,d): C_j \in (0, \infty), \alpha_j \in (0,2], j=1, \ldots, d\right\}\times \bbr^{d-1}.$$
This model is somewhat motivated by the time-shifted moving maxima Brown-Resnick process introduced by \citet{Koch}, it is however much simpler to analyse and to estimate. 
As a referee has pointed out, similar models have been suggested in Section~5.3.2, models (ii)-(iv) on p.~213 in \citet{HuserPhd}.
\medskip

In the following we show that models (I)-(III) satisfy Assumption~\ref{regcondgls} and the conditions of Theorem~\ref{GLSEcons} and Corollary~\ref{sufficient_BR}.

\subsection*{\bf Asymptotic properties of models (I)-(III)}

As before, we assume space-time observations on $\mathcal{D}_n=\cals \times \mathcal{T}=(\cals \times \mathcal{T})(n)$, where $\cals \subset \mathbb{Z}^{d-1}$ are the spatial and $\mathcal{T} \subset \mathbb{Z}$ the time series observations. 
Moreover, we assume that they decompose into $\mathcal{D}_n=\mathcal{F} \times \mathcal{I}_n$, where $\mathcal{F} \subset \mathbb{Z}^q$ is some fixed domain and $\mathcal{I}_n=\{1,\ldots,n\}^w$ is a sequence of regular grids, and $q+w=d$. 
 
For two points $(\bs s^{(1)},t^{(1)})$ and $(\bs s^{(2)},t^{(2)}) \in \mathbb{R}^{d-1} \times [0, \infty)$, we denote by $(\bs h,u)=(\bs s^{(1)},t^{(1)})-(\bs s^{(2)},t^{(2)}) \in \mathbb{R}^d$ their space-time lag vector. 
Furthermore, we choose Borel sets $A=B=(\underline{A},\infty)$ for some $\underline{A}>0$.  
We denote by $\widehat{\rho}_{AB,m_n}(\bs h,u)$ the (possibly bias-corrected) empirical space-time extremogram \eqref{biascorrectedextremo}, sampled at lags in $\calh \subset \bbr^d$, and by $\widehat{\bs \theta}_{n,V}$ the  GLSE \eqref{GLSE}, referring to some positive definite weight matrix $V$. 

To show consistency and asymptotic normality of the corresponding GLSE, we need to verify the assumptions required in Theorem~\ref{GLSEcons}; i.e. the relevant parts of Assumption~\ref{regcondgls}.
Note that Corollary~\ref{sufficient_BR} applies for all models, since they all satisfy $\delta_{\bs \theta}(\bs h,u) \geq C|u|^{\alpha}$ for $C>0$ and $\alpha \in (0,2]$. 
Thus we obtain the CLTs of the empirical extremogram centred by the pre-asymptotic extremogram \eqref{asyspace_BR}, centred by the true one \eqref{asyspace_BR_true} and of the bias corrected empirical extremogram centred by the true one \eqref{asynbias}.
Hence (G1) and (G2) hold for the empirical extremogram. Furthermore, we assume that the parameter space $\Theta \subset \bbr^k$, which contains the true parameter $\bs \theta^\star$ as an interior point, is a compact subset of the spaces introduced above for the corresponding models. 

The following requirements concern the model-independent assumptions.
\begin{enumerate}[$\bullet$]
\item
In order to determine the GLSE we need to choose a positive definite matrix $V(\bs \theta)$ for $\bs \theta \in \Theta$, and we take one, which satisfies condition (G5ii) with $z_2=1$. Due to compactness of the parameter space $\Theta$, condition (G5i) is therefore automatically satisfied.
\item
We require that $|\mathcal{H}| \geq k$, such that the rank condition (G6)
can be satisfied.
\end{enumerate}

Next we discuss the model-dependent assumptions. First note that the smoothness condition (G4) is satisfied for $z_1=0$ for all models $\{\rho_{AB,\bs \theta}(\cdot)\}$ (equivalently $\{\delta_{\bs \theta}(\cdot)\}$). Furthermore, due to compactness of the parameter space, it suffices to show condition (G3') in order to verify identifiability of the models. 
Condition (G3') is satisfied for models (I)-(III) if for two distinct parameter vectors $\bs \theta^{(1)} \neq \bs \theta^{(2)}$ there is at least one $(\bs h,u) \in \mathcal{H}$ such that $\rho_{AB,\bs \theta^{(1)}}(\bs h,u)\neq \rho_{AB,\bs \theta^{(2)}}(\bs h,u)$ {or, equivalently, $\delta_{\bs \theta^{(1)}}(\bs h, u)\neq \delta_{\bs \theta^{(2)}}(\bs h, u)$.}
This holds due to the power function structure of the models. 
For the geometric anisotropic model in (I) we need to exclude $c=1$ to ensure identifiability of the angle $\varphi$; however, if $c=1$ then $\varphi$ has no influence on the dependence function and can be neglected.
Thus, the GLSEs are consistent according to Theorem~\ref{GLSEcons}. 

We now turn to the CLT \eqref{asymptGLS}, where it remains to show (G4) for $z_1=1$.
Difficulties arise due to norms and absolute values of certain parameters in the model equations:
\begin{enumerate}[$\bullet$]
\item
In their basic forms without rotation or dilution, models (I) and (II) are infinitely often continuously partially differentiable in the model parameters. Hence asymptotic normality of the GLSEs follows by Theorem~\ref{GLSEcons}.
\item 
If rotation and/or dilution parameters are included, continuous partial differentiability still holds under the following restrictions: Let $\alpha_1$ (for model (I)) or \\$\alpha_1,\ldots,\alpha_{d-1}$ (for model (II)) be the spatial smoothness parameters. 
Since they are the powers of some norm or absolute value, 
restricting them to values in $[1,2]$ makes the models continuously partially differentiable; {otherwise, they are partially differentiable everywhere but not in 0. }
As to model (II), in the case $d-1=2$, one of the parameters $\alpha_1$ and $\alpha_2$ being larger than 1 is already sufficient.
To see this, recall that the spatial part of the dependence function is given by 
$$C_1|h_1\cos \varphi - h_2\sin \varphi |^{\alpha_1}+C_2|h_1\sin \varphi + h_2\cos \varphi |^{\alpha_2},\quad (h_1,h_2) \in \mathbb{R}^2.$$  
Assume w.l.o.g that $\alpha_2>1$. 
Then critical values of $\varphi \in [0,\pi/2)$ are the roots of $h_1\cos \varphi - h_2\sin \varphi $. 
Given a value $h_2 \in \mathbb{R}$ we need to choose $h_1 \in \mathbb{R}$ such that $h_1 \neq  h_2\tan \varphi $ for all $\varphi \in [0,\pi/2)$. 
Since $\tan \varphi>0$ for $\varphi \in [0,\pi_2)$, we can choose $h_1$ such that $\text{sgn}(h_1)=-\text{sgn}(h_2)$. 
If all lags $(h_1,h_2,u) \in \mathcal{H}$ are chosen such that $(h_1,h_2)$ have opposite signs (or, trivially, are equal to $(0,0)$) and if  ${\rm rank}(\Rho_{AB}(\bs \theta^{\star}))=k$, then the GLSE is asymptotically normal. 
\item 
Model (III) is continuous partially differentiable, if the spatial smoothness parameters $\alpha_i$ for $i=1,\ldots,d-1$ are all larger than 1. 
If $\alpha_i \leq 1$ for some $i$, then the term $C_i|h_i-u\tau_i|^{\alpha_i}$ is, as a function of $\tau_i$,  not differentiable at $\tau_i=h_i/u \in \mathbb{R}$. 
However, it is possible to restrict the parameter space such that such equalities do not occur.
\end{enumerate}

\section{Simulation study}\label{s6}

\subsection*{\bf Specifications}

Consider the framework of Section~\ref{s32}. 
In particular, let $\{\eta(\bs s,t): \bs s \in \bbr^{2}, t \in [0,\infty)\}$ be a strictly stationary space-time Brown-Resnick process \eqref{BR} observed on $\cald_n=\calf \times \cali_n$. 
Denote by {$\wh{\rho}_{AB,m_n}(\bs h,u)$} the space-time version of the (possibly bias corrected) empirical extremogram given in \eqref{biascorrectedextremo}, sampled at lags in $\calh \subset \bbr^d$, where $\calh$ is specified below and we choose the sets $A=B=(1,\infty)$. 
As already indicated in its Definition~\ref{ee}(1), 
{the computation involves the practical issue of choosing the value $a_{m_n}=m_n=:q$ as a large quantile, where the first equality is due to  the standard unit Fr{\'e}chet distribution of the marginals of the Brown-Resnick model, so that $q$ should be chosen as a large quantile of the standard unit Fr{\'e}chet distribution. 
In a data example it should be chosen from a set $Q$ of large empirical quantiles of $\{\eta(\bs s,t): (\bs s,t) \in \cald_n\}$ for which the empirical extremograms $\wh{\rho}_{AB,q}(\bs h,u)$, are robust. 
{For a practical guideline see \citet{Davis2}, Section~3.4 and the upper left panel of their Figure~1, and also \citet{DMZ} after their Theorem~2.1.}
In the following simulation scenarios we choose the lowest quantile of a given level of the sets $Q$. Note that due to the variability of the large empirical quantiles, this might involve (as below) the choice of different quantiles in different data examples.

In order to test the small sample performance of the GLSE $\wh{\bs\theta}_{n,V}$ defined in \eqref{GLSE}, we consider some of the models (I)-(III) for the dependence function $\delta_{\bs \theta}$. 
For the simulations we use the \texttt{R}-package \texttt{RandomFields} (\cite{Schlather5}) and the exact method {via extremal functions} proposed in \citet{Dombry2}, {Section~2}. {In this simulation study we use standardised univariate margins. If this in not the case (as for instance in the data example treated in Section~5 of \citet{buhl1}), they need to be estimated and standardised first, which naturally might lead to inferior estimation results.}\\[2mm]
\textbf{(i) Spatially isotropic fractional space-time model}\\
We generate 100 realisations from the 
model \eqref{varioStein} on a grid of size 15x15x300. 
This corresponds to the situation of a fixed spatial and an increasing temporal observation area; i.e., it is given by $\cald_n=\calf \times \cali_n$ with $\calf=\{1,\ldots,15\}^2$ and $\cali_n=\{1,\ldots,300\}$. 
We simulate the model with the  true parameter vector   
$$\bs \theta_1^\star=(0.8,0.4,1.5,1),$$
which we assume to lie in a compact subset of 
$$\Theta_1=\left\{(C_1,C_2,\alpha_1,\alpha_2): C_1,C_2 \in (0, \infty), \alpha_1,\alpha_2 \in (0,2]\right\}.$$
As the large empirical quantile $q$ we take the $96\%$-quantile of $\{\eta(\bs s,t): (\bs s,t) \in \cald_n\}$. \\[2mm]
\textbf{(ii) Geometrically anisotropic fractional space-time model}\\
We generate 100 realisations from  model \eqref{varioStein_aniso} on a grid of size 15x15x300. 
This corresponds to the same situation as in (i).
We simulate the model with the true parameter vector
$$\bs \theta_2^\star=(0.8,0.4,1.5,0.5,3,\pi/4),$$
which we assume to lie in a compact subset of 
\begin{align*}
\Theta_2=\big\{(C_1,C_2,&\alpha_1,\alpha_2,c,\varphi): \\
&C_1,C_2 \in (0, \infty), \alpha_1 \in [1,2],\alpha_2 \in (0,2],c>0, \varphi \in [0,\pi/2)\big\},
\end{align*}
where we choose $\alpha_1 \geq 1$ to ensure differentiability of the model, cf. the discussion in Section~\ref{s32}. 
As the large empirical quantile $q$ we take the $97\%$-quantile of $\{\eta(\bs s,t): (\bs s,t) \in \cald_n\}$. \\[2mm]
\textbf{(iii) Spatially anisotropic time-shifted model}\\
 We generate 100 realisations from model \eqref{BKLU_timeshift} on a grid of size 40x40x40, and consider this as a situation where the observation area increases in all dimensions; i.e., it is given by $\cald_n=\cali_n$ with $\cali_n=\{1,\ldots,40\}^3$. 
We simulate the model with the  true parameter vector  
$$\bs \theta_3^\star=(0.4,0.8,0.5,1.5,1.5,1,1,1),$$
which we assume to lie in a compact subset of 
$$\Theta_3=\left\{(C_1,C_2,C_3,\alpha_1,\alpha_2,\alpha_3,\tau_1,\tau_2): C_j \in (0, \infty), \alpha_1,\alpha_2 \in [1,2], \alpha_3 \in (0,2],\tau_j \in \bbr\right\},$$
where we choose $\alpha_1,\alpha_2 \geq 1$ to ensure differentiability of the model, cf. the discussion in Section~\ref{s32}.
As the large empirical quantile $q$ we take the $95\%$-quantile of $\{\eta(\bs s,t): (\bs s,t) \in \cald_n\}$.
\halmos\\[2mm]

\noindent
In all three settings we base the estimation on the set $\mathcal{H}$ of lags  given by
\begin{align*}
\mathcal{H}=\{&(0,0,1),(0,0,2),(0,0,3),(0,0,4),(1,0,0),(2,0,0),(3,0,0),\\&(4,0,0),(2,1,0),(4,2,0),(1,2,0),(2,4,0),(1,1,1),(2,2,2),(1,3,2)\}.
\end{align*}
With this choice we ensure that the lag vectors vary in all three dimensions so that we obtain reliable estimates. Generally one should choose $\mathcal{H}$ such that the whole range of clear extremal dependence is covered. However, beyond that, no lags should be included for the estimation, since independence effects can introduce a bias in the  least squares estimates, similarly as in pairwise likelihood estimation; cf. \citet{buhl1}, Section~5.3. 
One way to determine the range of extremal dependence are permutation tests, which are described in \citet{Steinkohl3}, Section~6. 
From those tests we know that our choice of lags satisfies this requirement for all three models.

For the weight matrix $V$ in \eqref{GLSE} we propose two choices, which yield equally good results in our statistical analysis. 
The first choice is $V_1=\diag\{\exp(-\|(\bs h,u)\|^2): \bs (\bs h,u) \in \mathcal{H}\},$ which reflects the exponential decay of the tail dependence coefficients $\rho_{(1,\infty)(1,\infty)}(\bs h,u)$ of Brown-Resnick processes given by tail probabilities of the standard normal distribution. 
The second choice is to include the (possibly bias corrected) empirical extremogram estimates as in $V_2=\diag\{\wh{\rho}_{(1,\infty)(1,\infty),q}(\bs h,u): (\bs h,u) \in \mathcal{H}\}$ (provided this is a valid choice; i.e., $V_2$ has only positive diagonal entries).
Since the so defined weight matrix is random, what follows is conditional on its realisation.
It is in practice not possible to incorporate the asymptotic covariance matrix $\Pi$ of the empirical extremogram estimates $(\wh{\rho}_{(1,\infty)(1,\infty),q}(\bs h,u): (\bs h,u) \in \mathcal{H})$ (cf. Remark~\ref{rem_weight_matrix}) to obtain a weight matrix that is optimal in theory. 
As can be seen from its specification in Theorem~\ref{stasyn}, it contains infinite sums and is, hence, numerically hardly tractable.

\subsection*{\bf Results}

For each of the scenarios (i)-(iii) we report the mean, the mean absolute error (MAE), the root mean squared error (RMSE), {and a relative root mean squared error (REL) of the resulting GLSEs for the 100 simulations. 
Exemplary for the parameter $C_1$, the REL is defined as
$$\sqrt{\frac{1}{100}\sum_{j=1}^{100} \Big[\frac{\wh{C}_{1,j}-C_1^\star}{C_1^\star}\Big]^2},$$ 
where $C_1^\star$ denotes the true parameter value and $\wh{C}_{1,j}$ the $j$th parameter estimate.}

 As weight matrix we choose {$V_2=\diag\{\wh{\rho}_{(1,\infty)(1,\infty),q}(\bs h,u): (\bs h,u) \in \mathcal{H}\}$ defined above.} The average computing time per simulation depends on the complexity of the model (i.e., the number of parameters to be estimated) and more crucially on the chosen set $\mathcal{H}$ and on the grid size. We report an average time of 14.51 seconds for scenario (i), 14.95 seconds for scenario (ii) and 14.63 seconds for scenario (iii).
The {estimation} results are summarised in Tables~\ref{summary1}-\ref{summary3}. Furthermore, in Figures~\ref{sim_I}-\ref{sim_III}  we plot the parameter estimates 
and add $95\%$-confidence bounds found by {subsampling}; cf. \citet{Politis4}, Chapter~5. 
We use subsampling methods, since the asymptotic covariance matrix $\Pi_V$ specified in Theorem~\ref{GLSEcons} contains the matrix $\Pi$ as specified in Theorem~\ref{stasyn}, which is, as explained above, hardly tractable. 
The fact that subsampling yields asymptotically valid confidence intervals for the true parameter vectors $\bs\theta_i^\star$ for $i=1,2,3$ can be proved analogously to the proof of {Theorem~3.5} in \citet{Steinkohl3} based on Corollary~5.3.4 of \cite{Politis4}. 
It requires mainly the existence of continuous limit distributions of $\sqrt{n^w/m_n^d}\|(\wh{\bs\theta}_{n,V}-\bs\theta_i^\star)\|$, which are guaranteed by Theorem~\ref{GLSE}, and some conditions on the $\alpha$-mixing coefficients, which can be shown similarly as those required in Theorem~\ref{stasyn}.

Summarising our results, we find that the GLSE estimates the model parameters accurately.
Bias and variance are largest for the parameter estimates of model (ii). 
There are two main reasons for this. 
Compared to model (i), for model (ii) we estimate two more parameters based on the same observation scheme.
However, one is a direction, which is non-trivial to estimate and decreases the overall quality of the estimates. 
For the estimation of model (iii) the observation scheme is different; in particular, there is a relatively large number of both spatial and temporal observations available.
In contrast, in the setting of models (i) and (ii) only the number of temporal observations is large.

{From Tables~\ref{summary1} and~\ref{summary2} we conclude that bias and REL of the spatial parameter estimates $\wh{C}_1$ and $\wh{\alpha}_1$ are comparable with those of the temporal parameter estimates $\wh{C}_2$ and $\wh{\alpha}_2$. Bias of the spatial estimates is slightly larger than bias of the temporal estimates, which might be due to the fact that only the number of temporal observations is large.}

{From Table~\ref{summary3} we read off that the RELs of the estimates $\wh{C}_1$ and $\wh{\alpha}_1$, which correspond to the first spatial dimension, are slightly smaller than those of  $\wh{C}_2$ and $\wh{\alpha}_2$. 
A reason for this might be the choice of the lag vectors which we included in the set $\mathcal{H}$ and which show more variation with respect to the first dimension than with respect to the second.}

In her PhD thesis, \citet{steinkohlphd} compares computing times of the commonly applied pairwise likelihood estimation with the semiparametric method described in \citet{Steinkohl3}, which can be regarded as a special case of the method described in this paper. She reports in Table~6.4  a reduction of computing time by about a factor 15. Furthermore, in Section~5 of  \cite{Steinkohl3} we show that the semiparametric methods are more robust against small deviations from the model assumptions such as measurement errors.
\begin{table}[H]
\begin{center}\begin{small}
\captionsetup{type=table}
\begin{tabular}{c|c|c|c|c|c}
& TRUE & MEAN & MAE & RMSE & { REL}\\
\hline
$\wh{C}_1$ & 0.8 & 0.7856 & 0.1353 & 0.1763 & { 0.2204}\\
$\wh{C}_2$ & 0.4 & 0.3987 & 0.0785 & 0.0995 & { 0.2486}\\
$\wh{\alpha}_1$ & 1.5 & 1.4830 & 0.0897 & 0.1131 & { 0.0754}\\
$\wh{\alpha}_2$ & 1 & 0.9916 & 0.0625 & 0.0820 & { 0.0820}
\end{tabular}\end{small}
\captionof{table}{True parameter values (first column) and mean, MAE, RMSE,   {and REL} of the estimates of the parameters of model (i).}
\label{summary1}
\end{center}
\begin{center}\begin{small}
\captionsetup{type=table}
\begin{tabular}{c|c|c|c|c|c}
& TRUE & MEAN & MAE & RMSE & { REL}\\
\hline
$\wh{C}_1$ & 0.8 & 0.7270 & 0.2750 & 0.3350 & { 0.4192}\\
$\wh{C}_2$ & 0.4 & 0.3708 & 0.1097 & 0.1377 & { 0.3443}\\
$\wh{\alpha}_1$ & 1.5 & 1.4349 & 0.2274 & 0.2692 & { 0.1794}\\
$\wh{\alpha}_2$ & 0.5 & 0.5143 & 0.0491 & 0.0684 & { 0.1369}\\
$\wh{c}$ & 3 & 2.9441 & 0.1365 & 0.2645 & { 0.0882}\\
$\wh{\varphi}$ & $\pi/4$ & 0.7906 & 0.1214 & 0.1567 & { 0.1995 }
\end{tabular}\end{small}
\captionof{table}{True parameter values (first column) and mean, MAE, RMSE, {and REL} of the estimates of the parameters of model (ii).}
\label{summary2}
\end{center}
\begin{center}\begin{small}
\captionsetup{type=table}
\begin{tabular}{c|c|c|c|c|c}
& TRUE & MEAN & MAE & RMSE  & { REL}\\
\hline
$\wh{C}_1$ & 0.4 & 0.4072 & 0.0690 & 0.0898 & { 0.2244}\\
$\wh{C}_2$ & 0.8 & 0.8482 & 0.1667 & 0.2187 & { 0.2734}\\
$\wh{C}_3$ & 0.5 & 0.5003 & 0.1085 & 0.1366 & { 0.2733}\\
$\wh{\alpha}_1$ & 1.5 & 1.5144 & 0.0594 & 0.0781 & {0.0521}\\
$\wh{\alpha}_2$ & 1.5 & 1.5043 & 0.1054 & 0.1282 &  {0.0855}\\
$\wh{\alpha}_3$ & 1 & 0.9694 & 0.1082 & 0.1415 & { 0.1415}\\
$\wh{\tau}_1$ & 1 & 1.0459 & 0.0945 & 0.1250 & { 0.1250}\\
$\wh{\tau}_2$ & 1 & 0.9916 & 0.0320 & 0.0420 & { 0.0420}
\end{tabular}\end{small}
\captionof{table}{True parameter values (first column) and mean, MAE, RMSE, {and REL} of the estimates of the parameters of model (iii).}
\label{summary3}
\end{center}
\end{table}

\begin{figure}[H]
\centering
\subfloat[]{\includegraphics[scale=0.175]{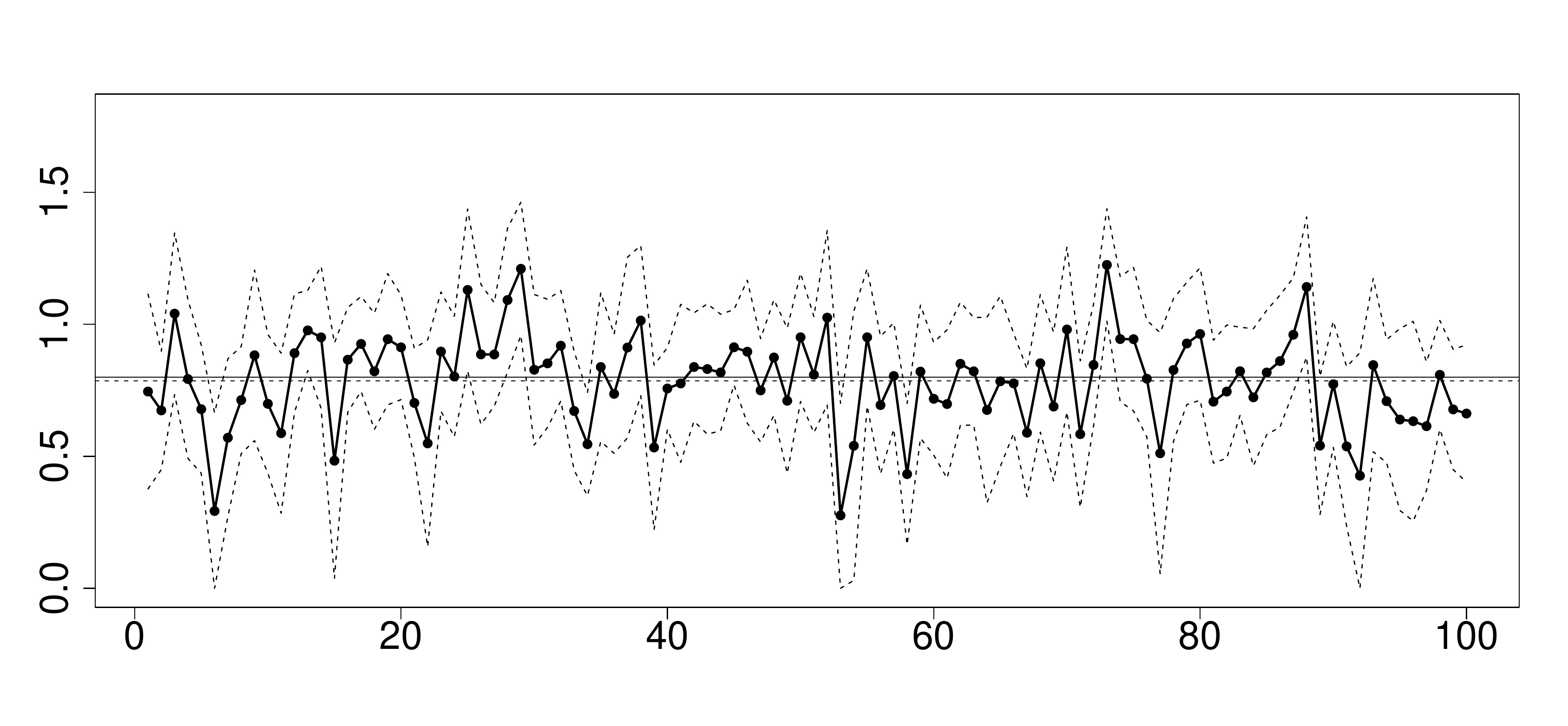}}
\subfloat[]{\includegraphics[scale=0.175]{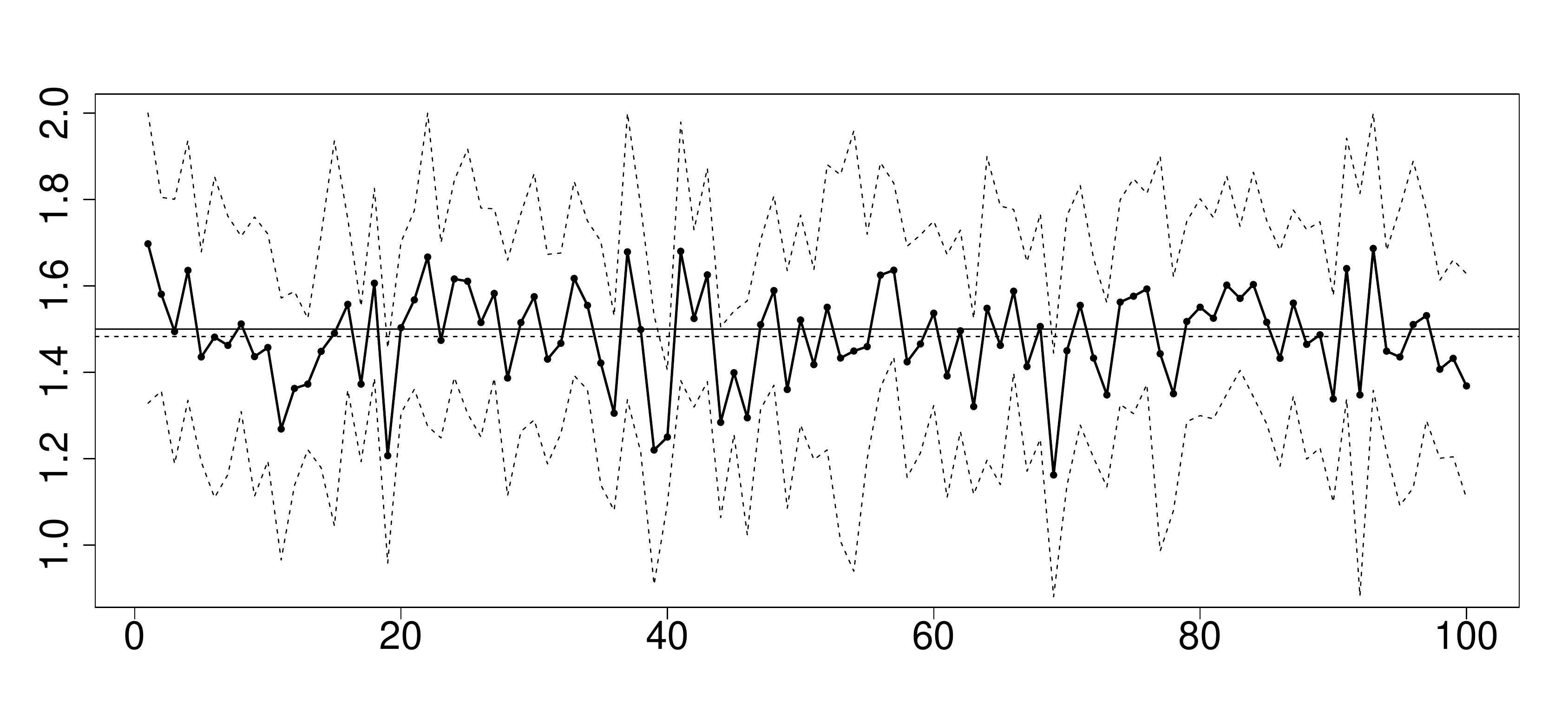}}\\[-7mm]
\subfloat[]{\includegraphics[scale=0.175]{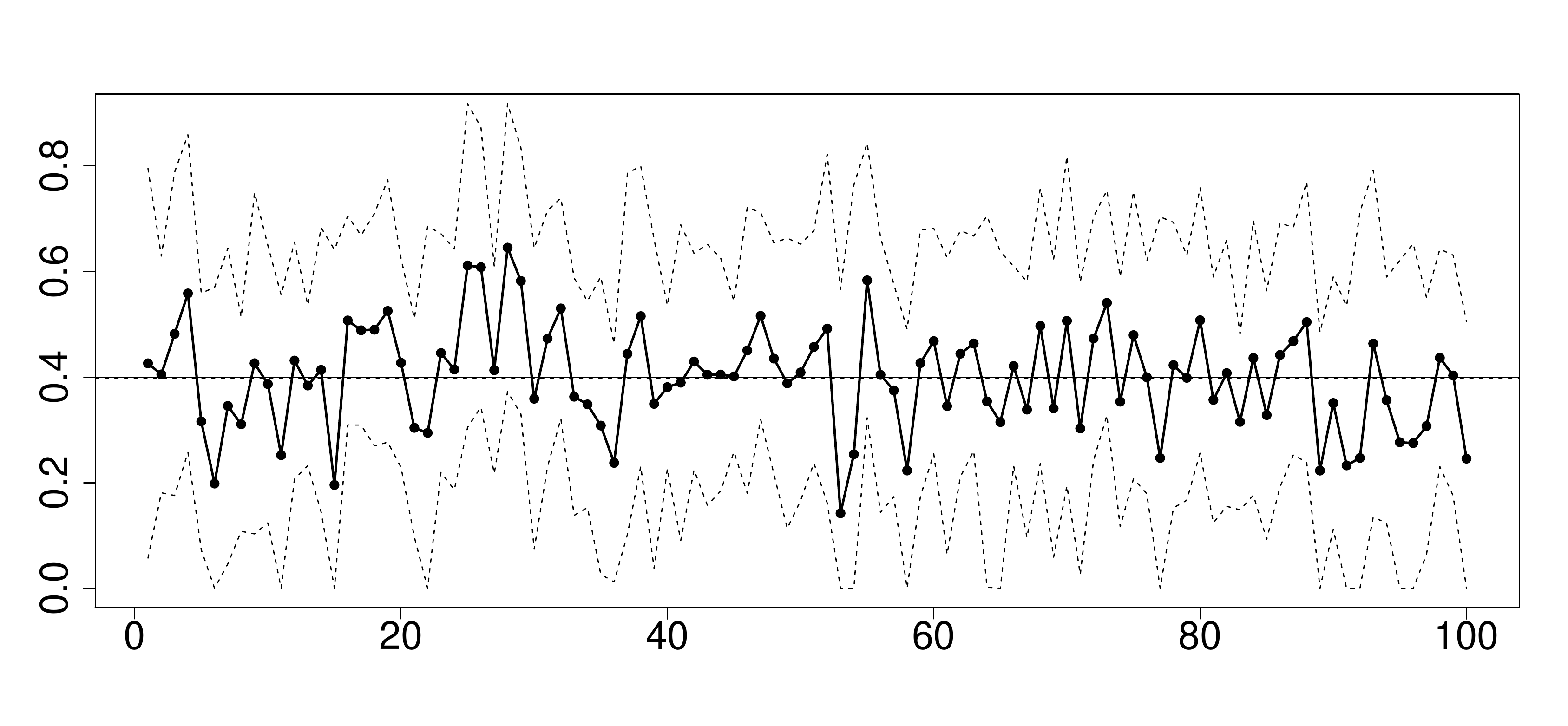}}
\subfloat[]{\includegraphics[scale=0.175]{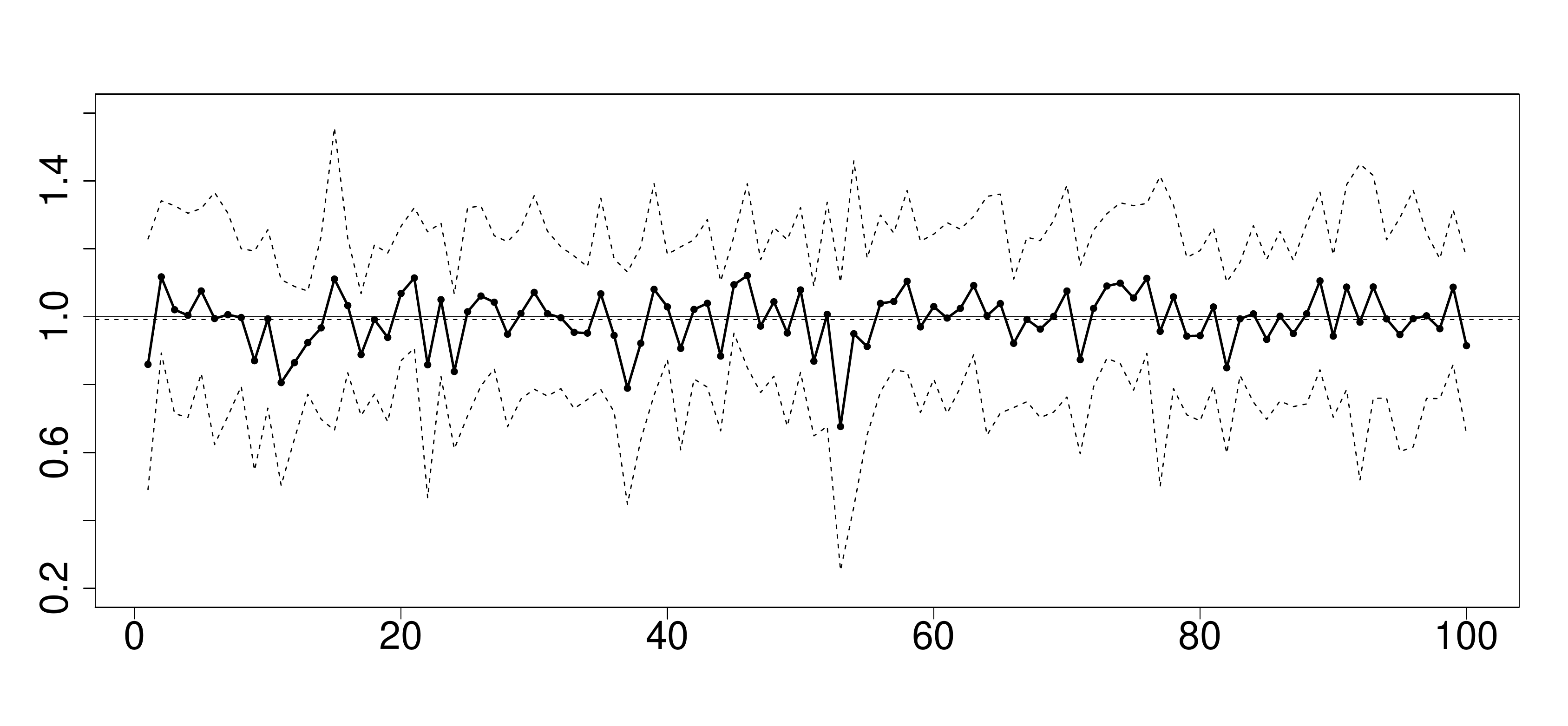}}\\[-7mm]
\caption{GLSEs of the parameters of model (i) for 100 simulated Brown-Resnick space-time processes together with pointwise $95\%$-subsampling confidence intervals (dotted). 
First row: $C_1$, $\alpha_1$, second row: $C_2$, $\alpha_2$. 
The middle solid line is the true parameter value and the middle dotted line represents the mean over all estimates.} \label{sim_I}
\end{figure}

\begin{figure}
\centering
\subfloat[]{\includegraphics[scale=0.175]{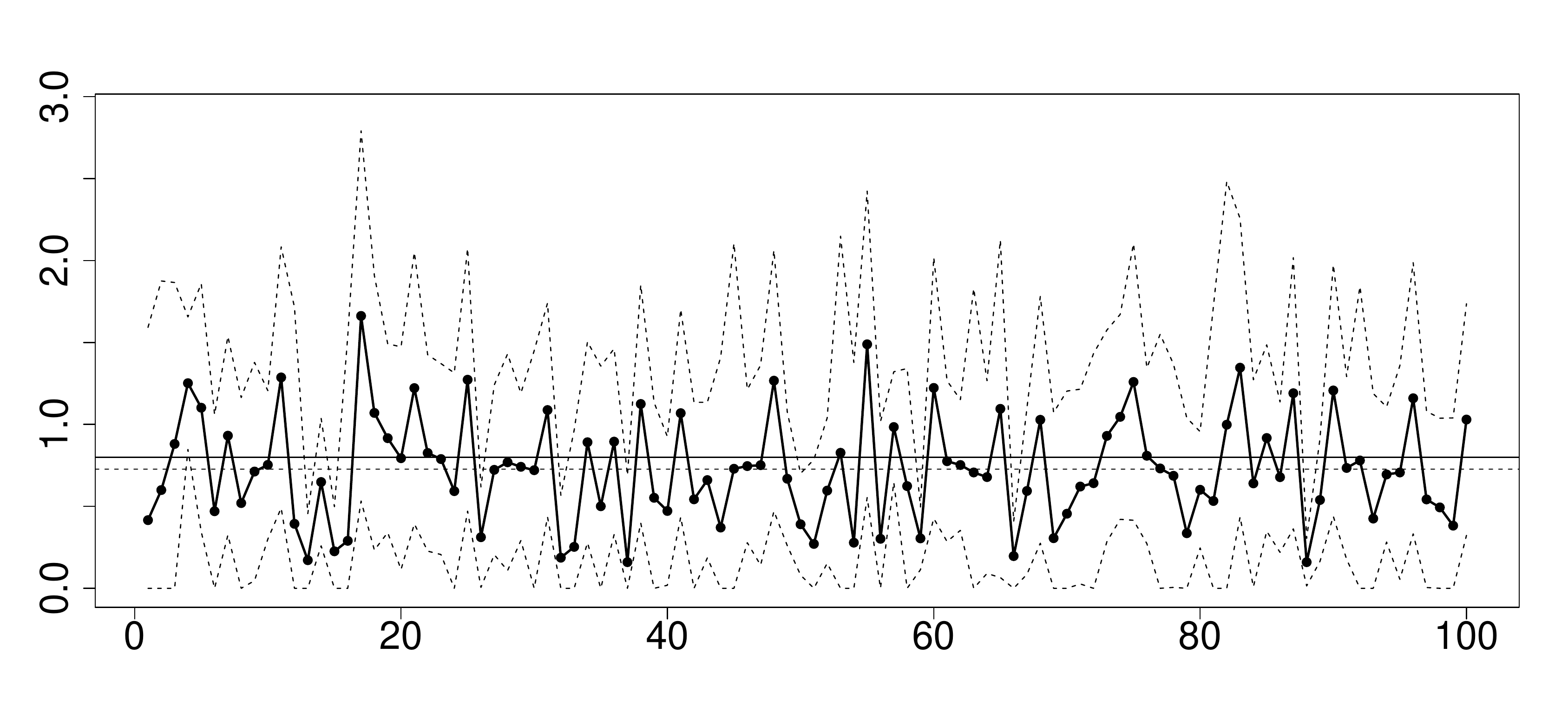}}
\subfloat[]{\includegraphics[scale=0.175]{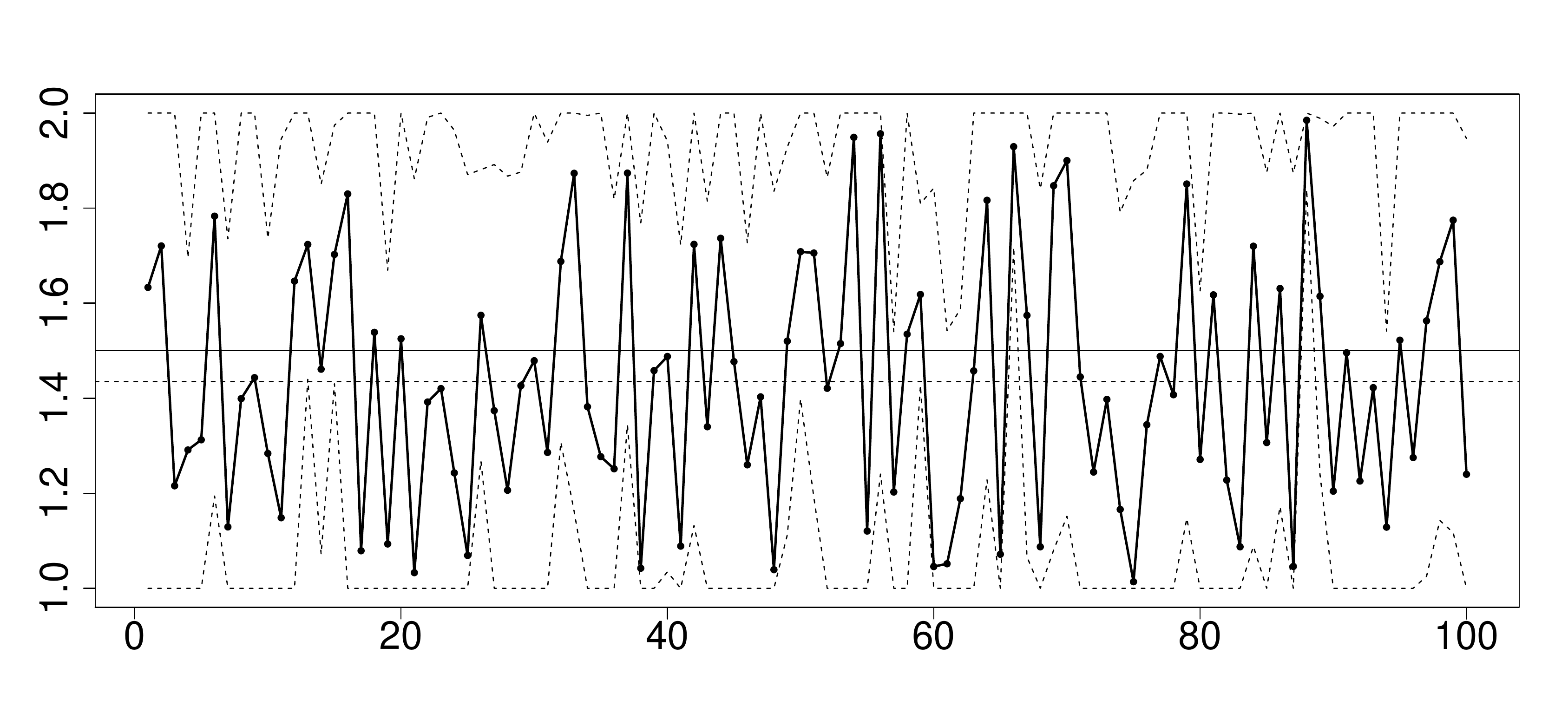}}\\[-7mm]
\subfloat[]{\includegraphics[scale=0.175]{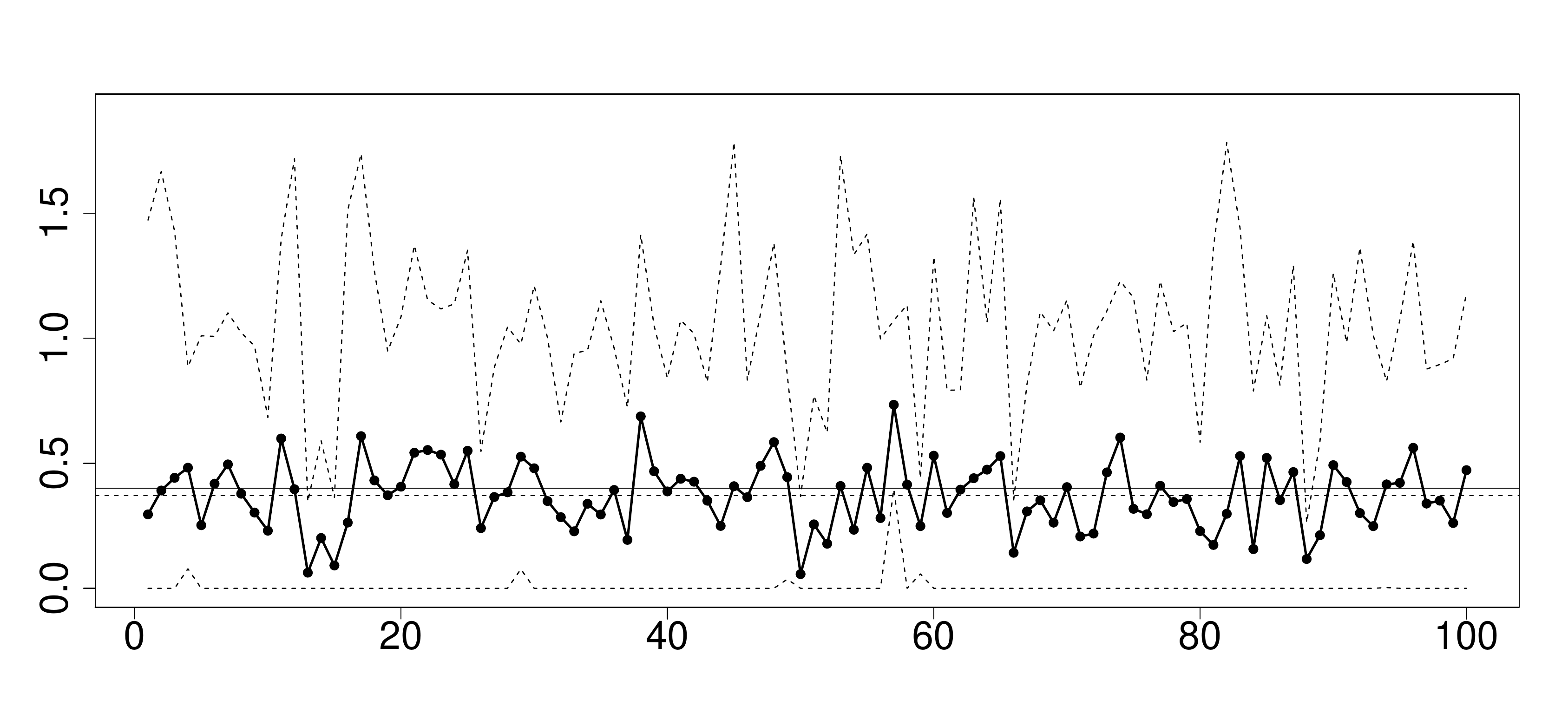}}
\subfloat[]{\includegraphics[scale=0.175]{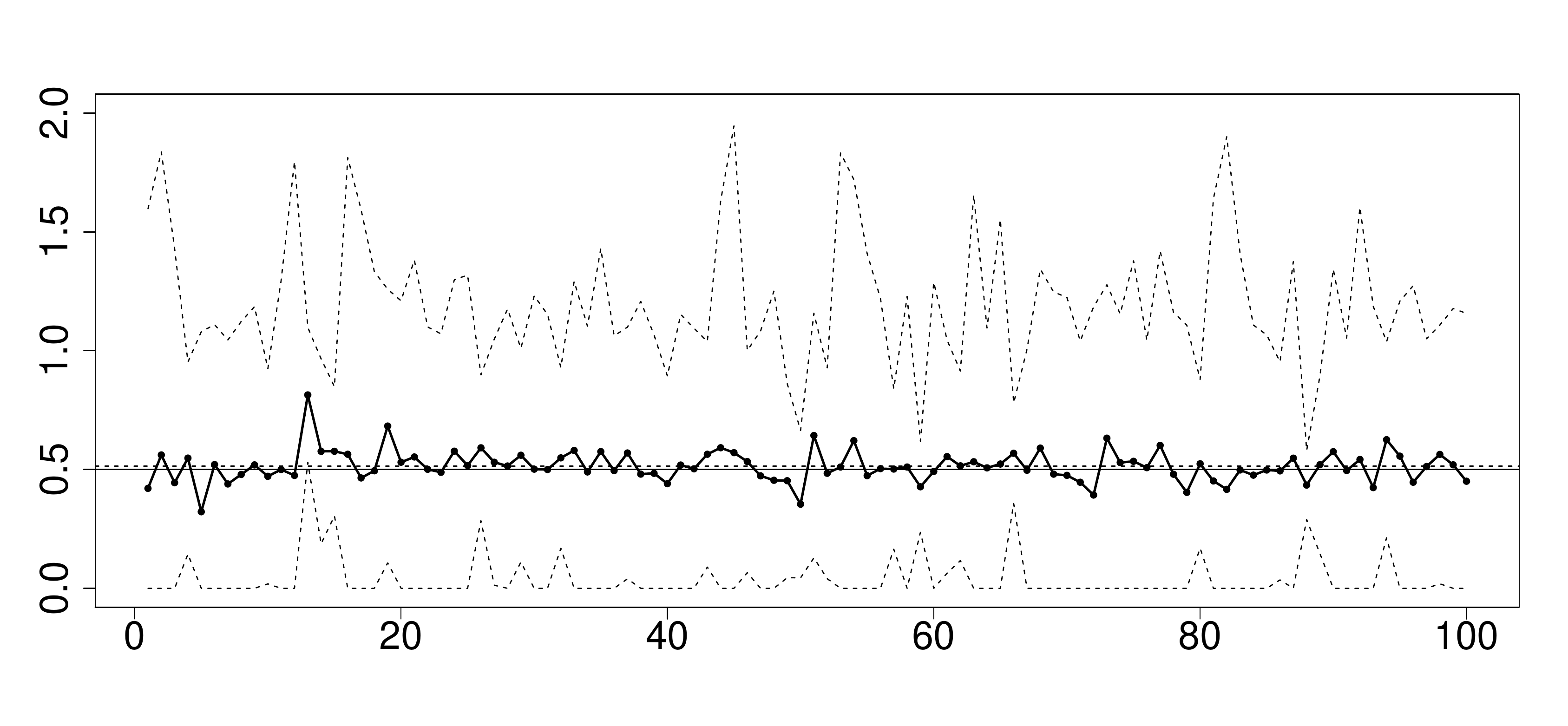}}\\[-7mm]
\subfloat[]{\includegraphics[scale=0.175]{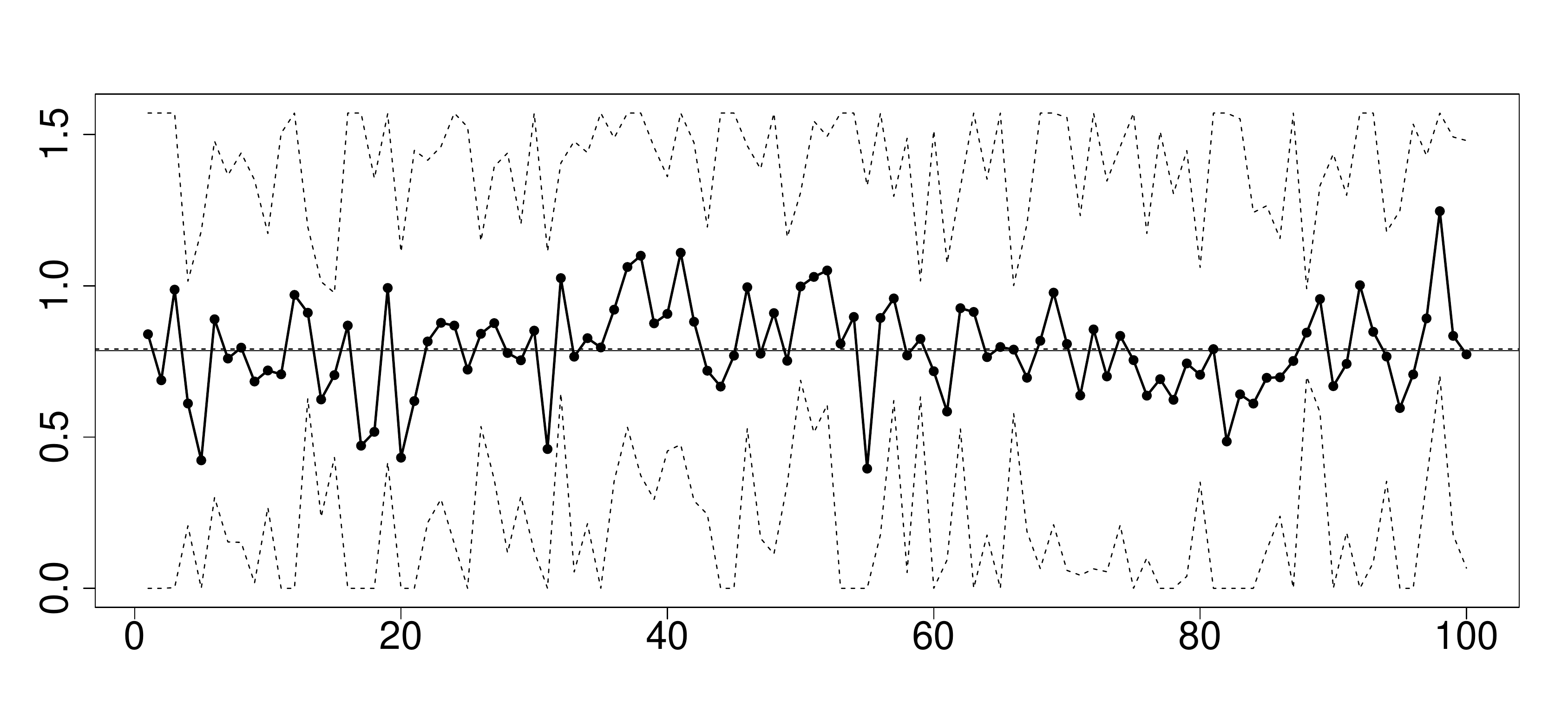}}
\subfloat[]{\includegraphics[scale=0.175]{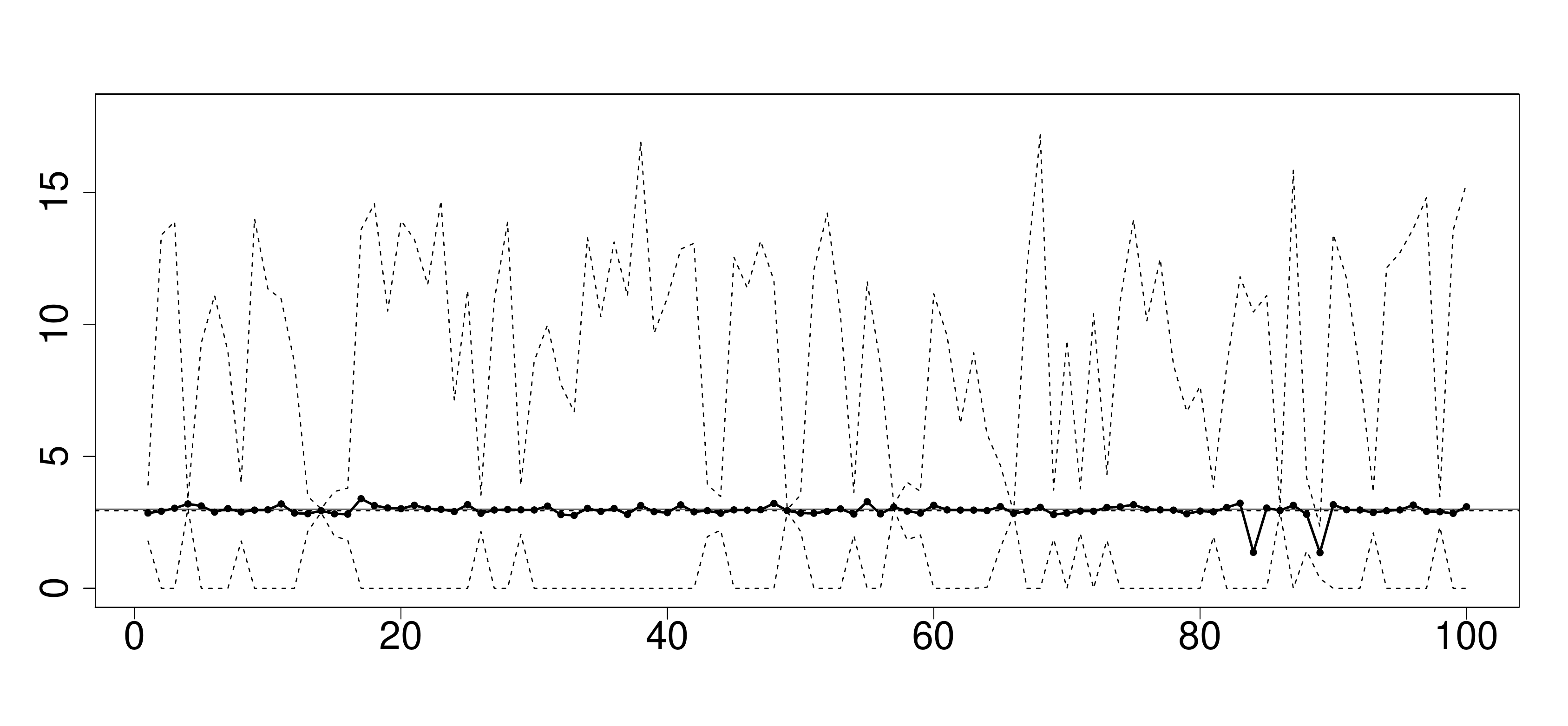}}
\caption{GLSEs of the parameters of model (ii) for 100 simulated Brown-Resnick space-time processes together with pointwise $95\%$-subsampling confidence intervals (dotted). 
First row: $C_1$, $\alpha_1$, middle row: $C_2$, $\alpha_2$, last row: $\varphi$ and $c$. 
The middle solid line is the true value and the middle dotted line represents the mean over all estimates.} \label{sim_II}
\end{figure}

\begin{figure}
\centering
\subfloat[]{\includegraphics[scale=0.175]{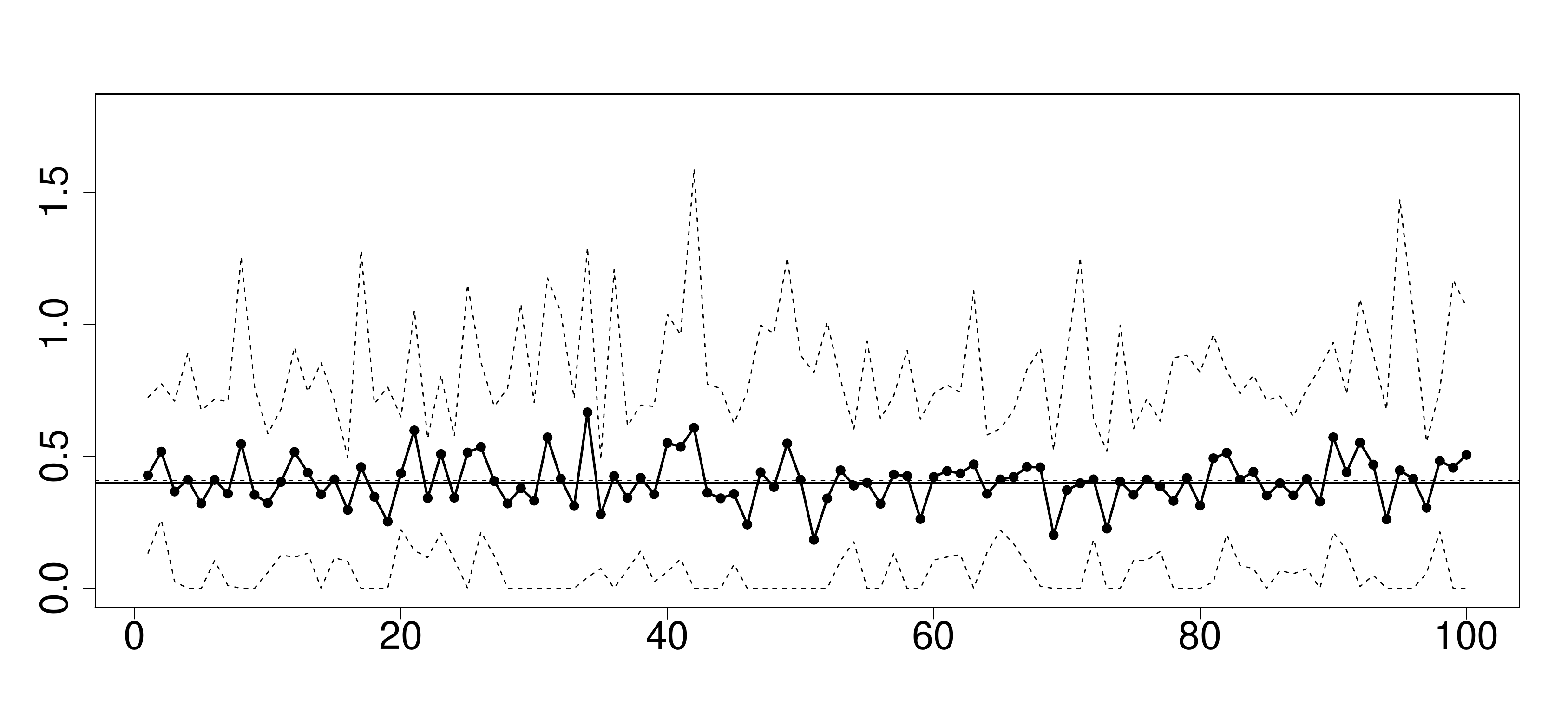}}
\subfloat[]{\includegraphics[scale=0.175]{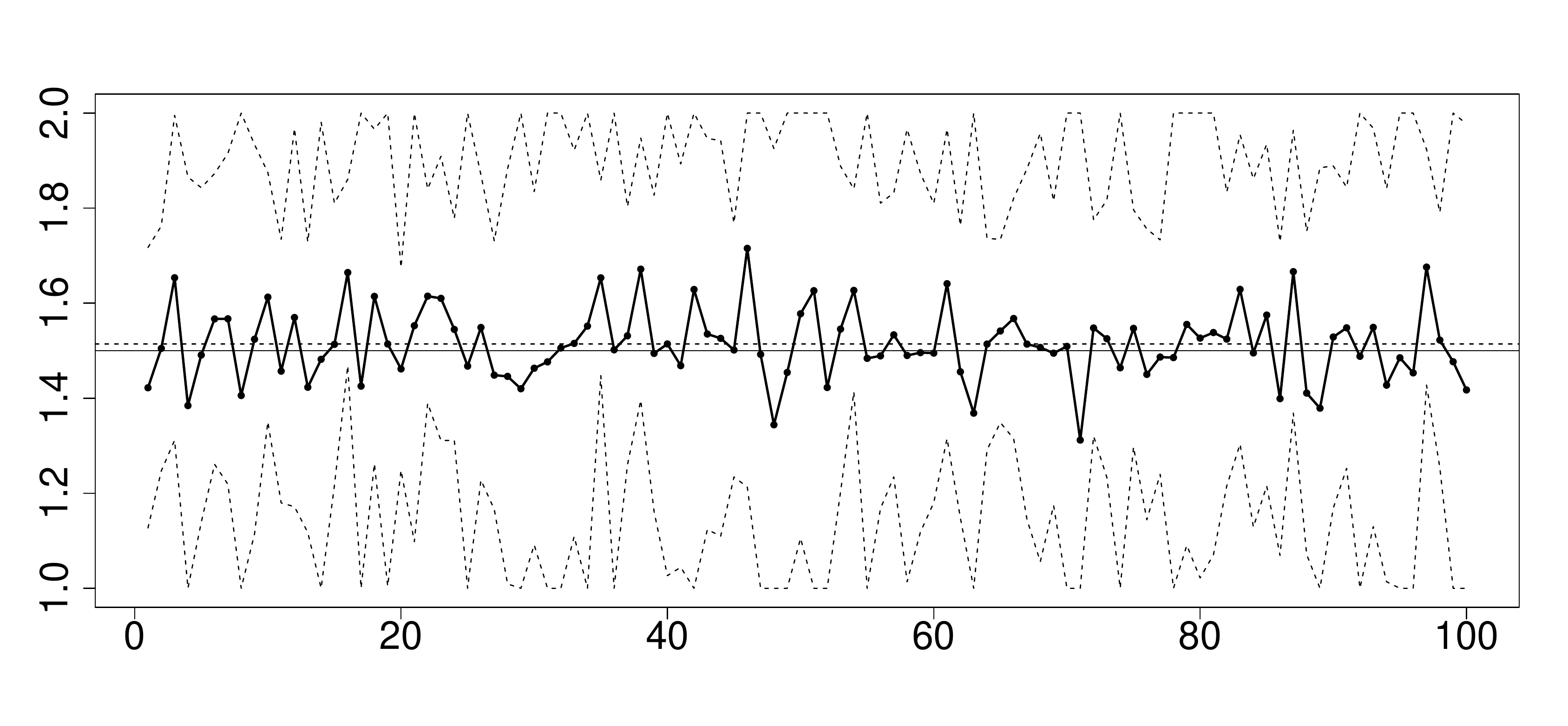}}\\[-7mm]
\subfloat[]{\includegraphics[scale=0.175]{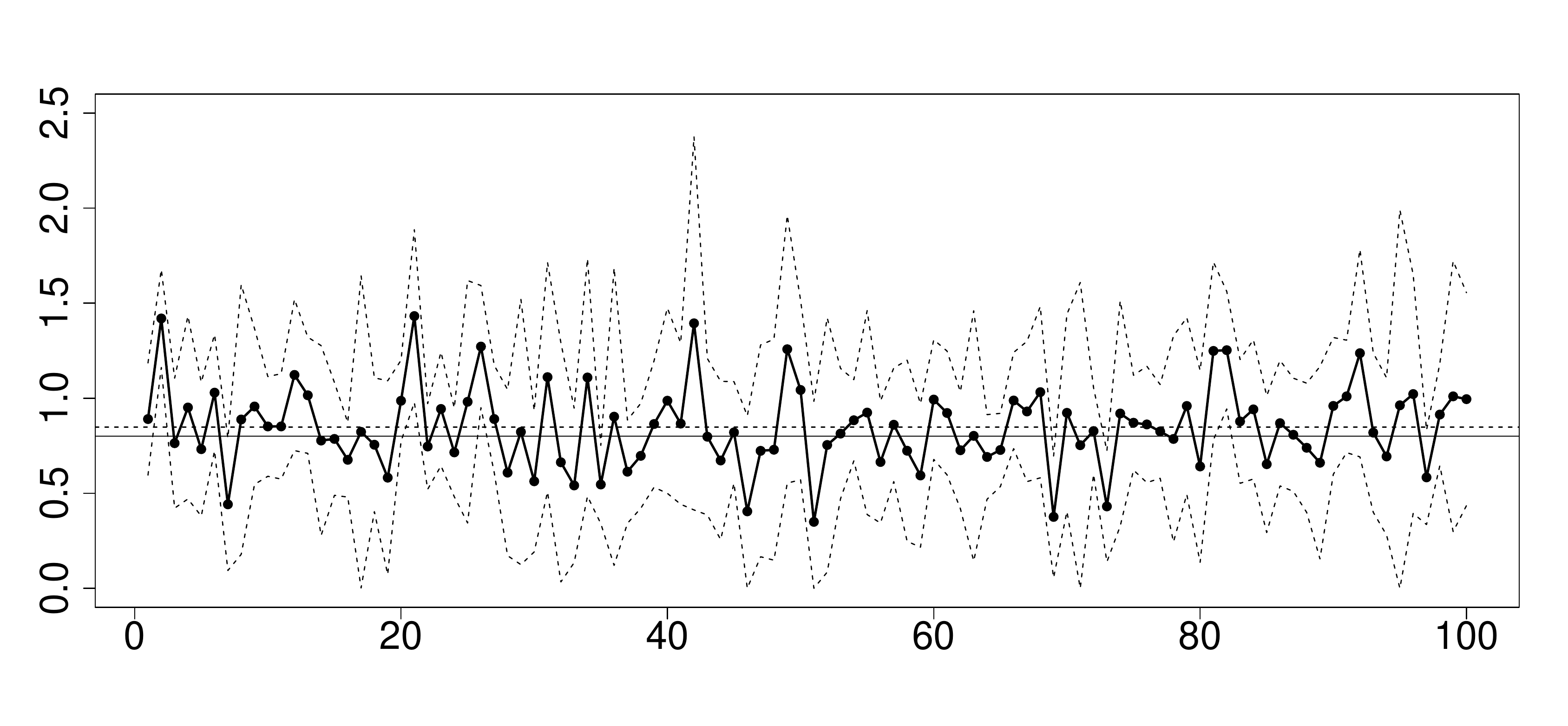}}
\subfloat[]{\includegraphics[scale=0.175]{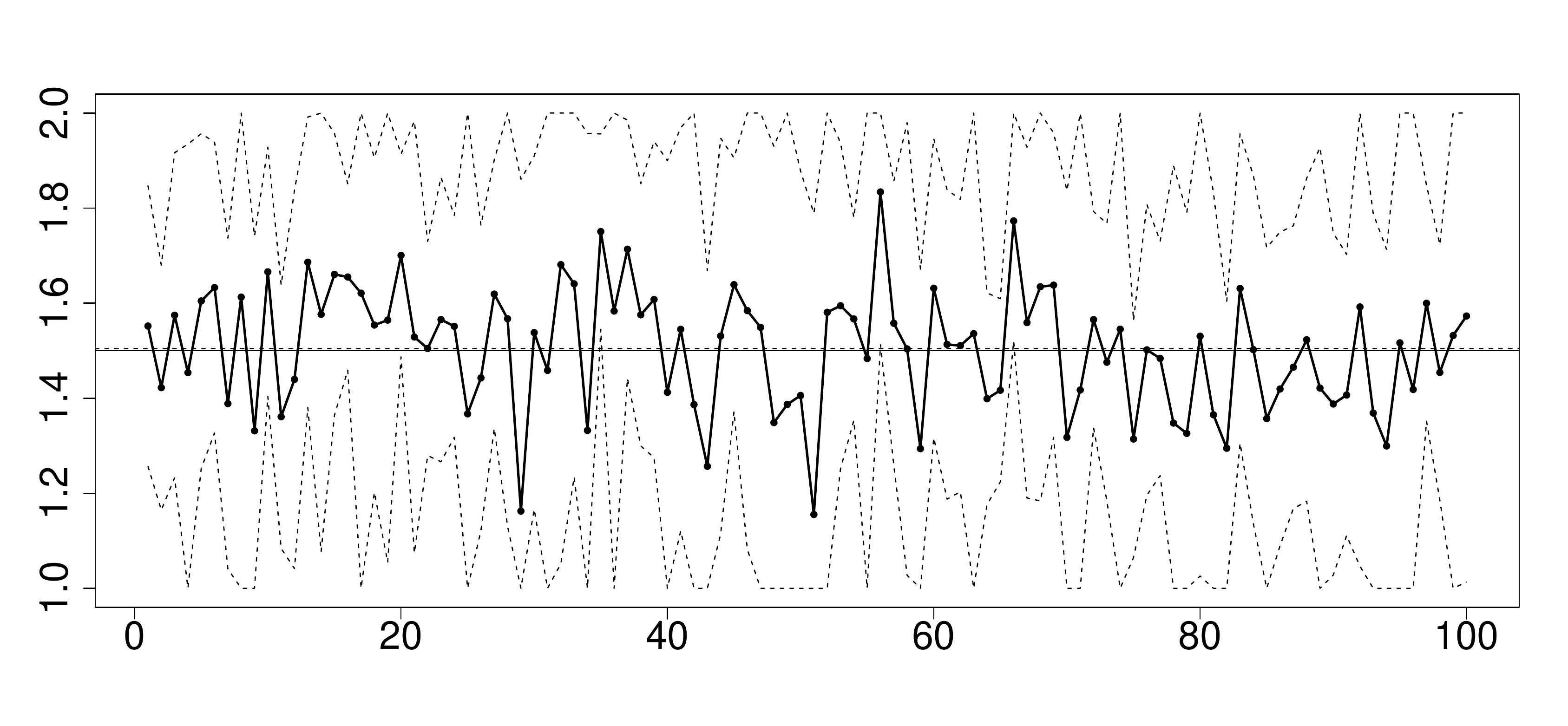}}\\[-7mm]
\subfloat[]{\includegraphics[scale=0.175]{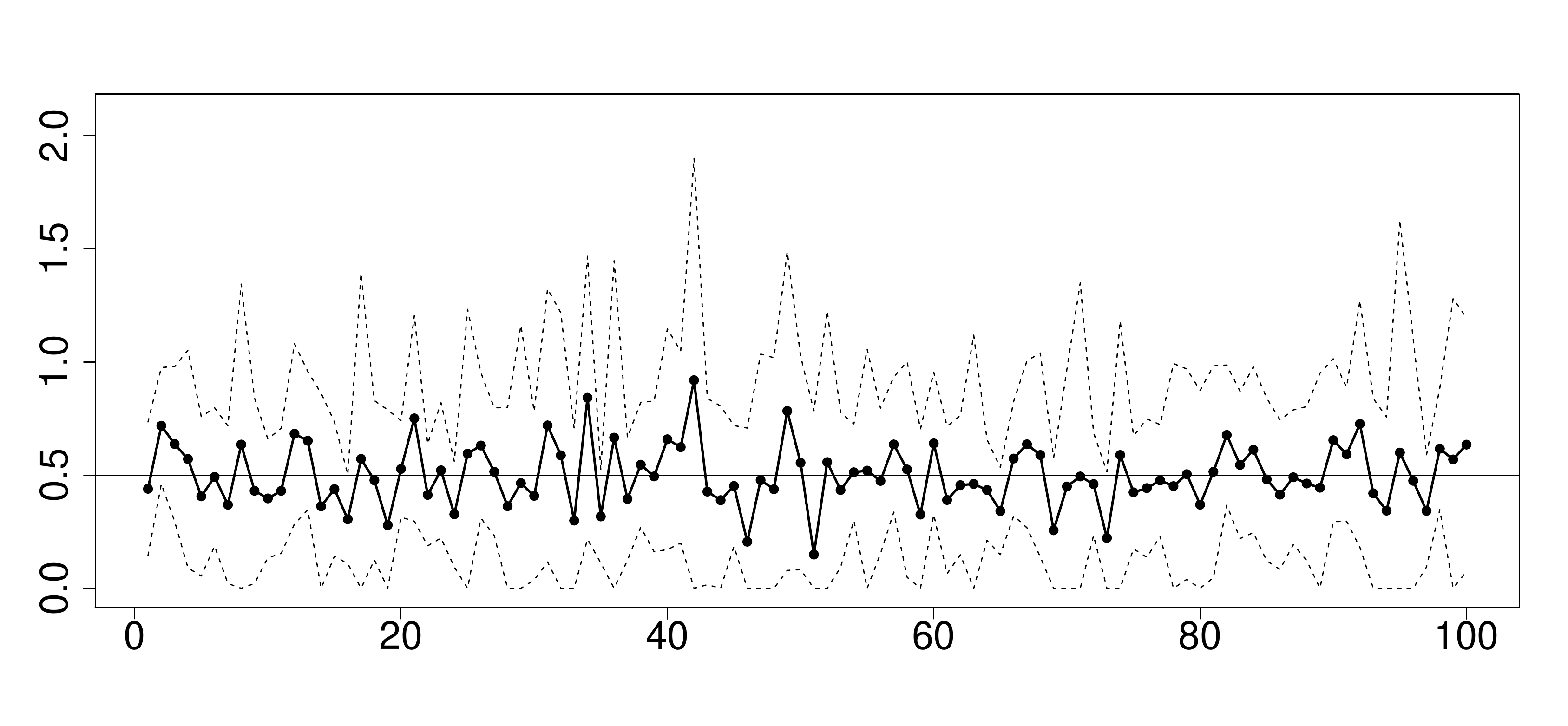}}
\subfloat[]{\includegraphics[scale=0.175]{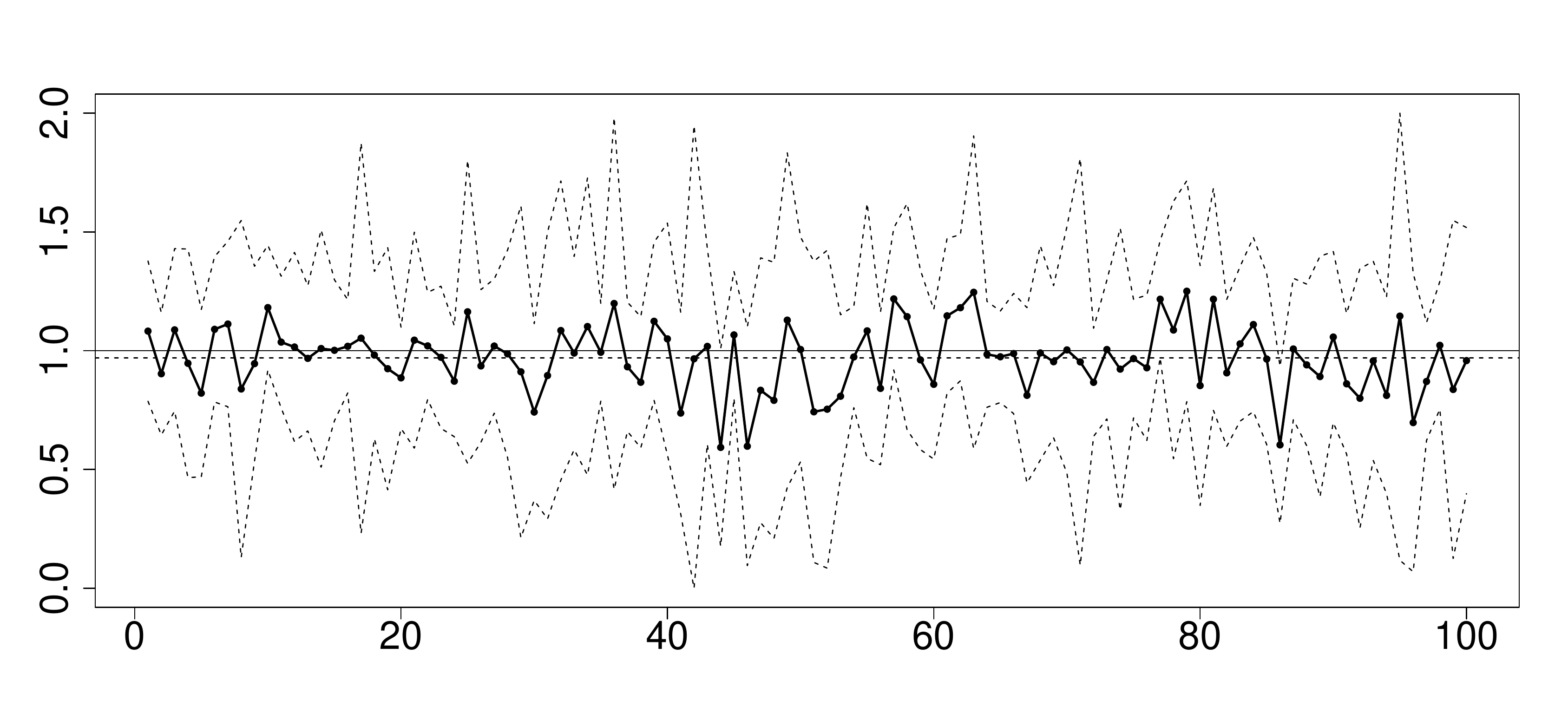}}\\[-7mm]
\subfloat[]{\includegraphics[scale=0.175]{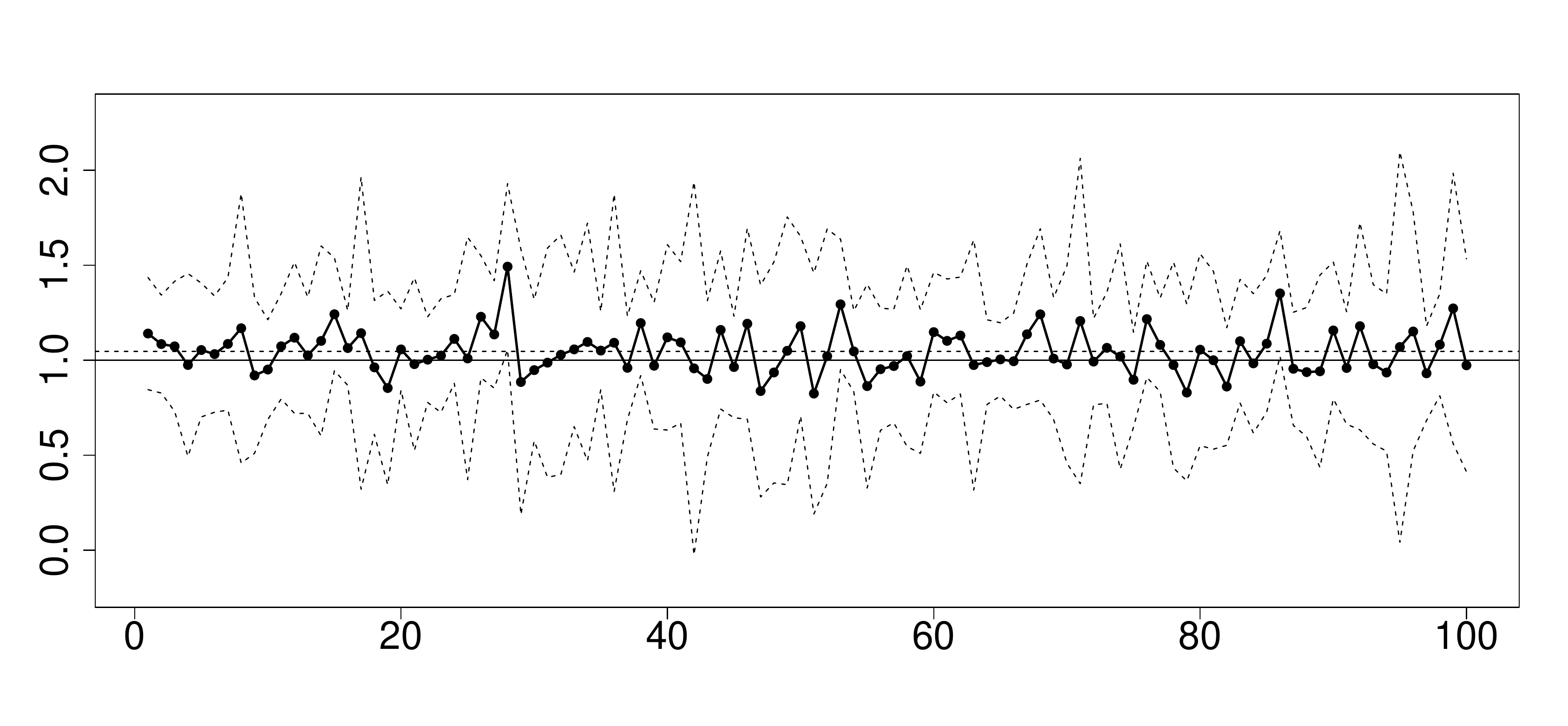}}
\subfloat[]{\includegraphics[scale=0.175]{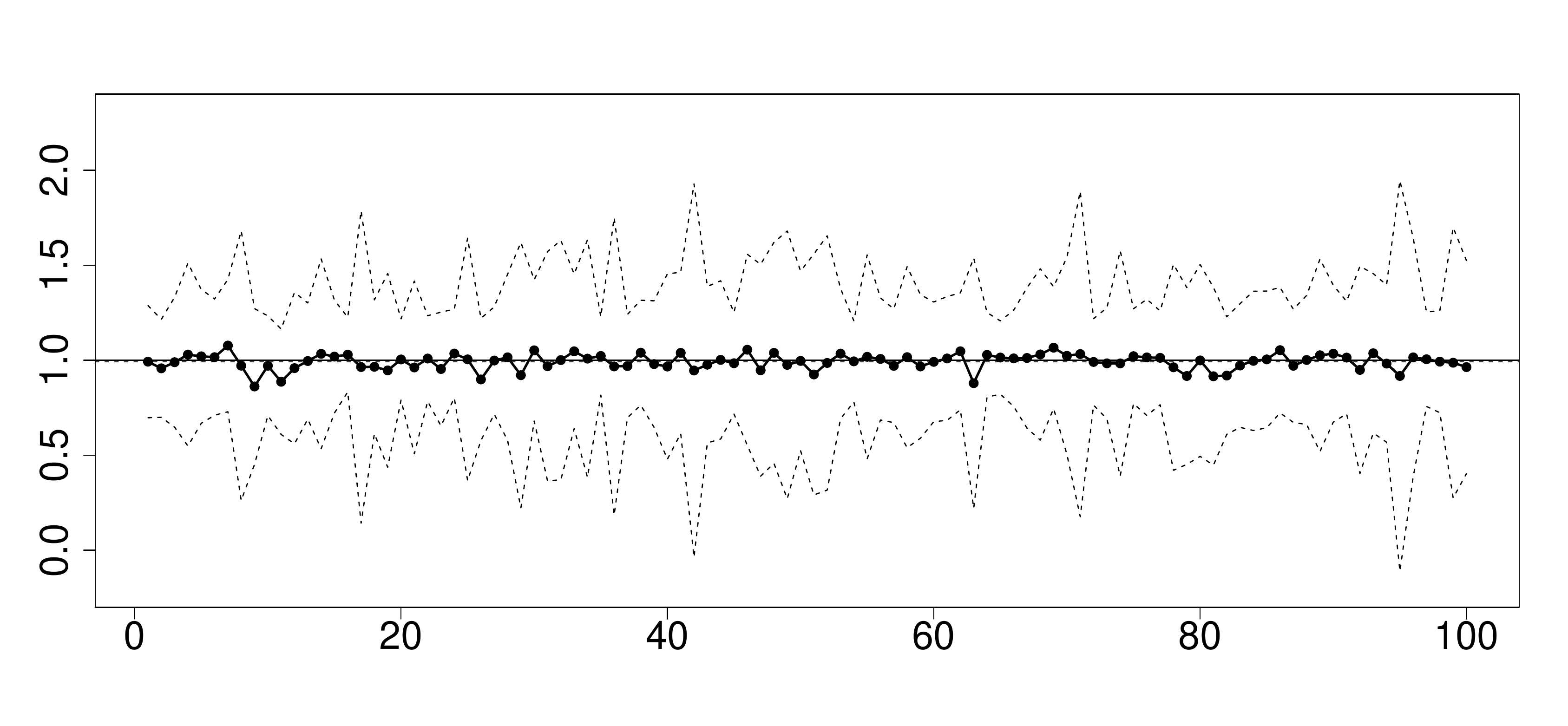}}\\[-7mm]
\caption{GLSEs of the parameters of model (iii) for 100 simulated Brown-Resnick space-time processes together with pointwise $95\%$-subsampling confidence intervals (dotted). First row: $C_1$, $\alpha_1$, second row: $C_2$, $\alpha_2$, third row: $C_3$, $\alpha_3$, fourth row: $\tau_1$, $\tau_2$.
The middle solid line is the true value and the middle dotted line represents the mean over all estimates.} \label{sim_III}
\end{figure}

\subsection*{\bf Further {insight}}
\subsection*{(a) Influence of the choice of lags}
In order to understand how the choice of lags in $\mathcal{H}$ influences computing times and the quality of the estimates, we repeat simulation scenario (i) for different sets $\mathcal{H}_{\ell}$ where $\ell=1,\ldots,5$. These are given by 
\begin{small}
\begin{align*}
\mathcal{H}_1&=\{(0,0,1),(1,0,0),(0,0,2)\}, \\
\mathcal{H}_2&=\mathcal{H}_1 \cup \{(2,0,0),(2,1,0),(1,2,0),(1,1,1),(1,3,2)\}, \\
\mathcal{H}_3&=\mathcal{H}_2 \cup \{(0,0,3),(0,0,4),(3,0,0),(4,0,0),(4,2,0),(2,4,0),(2,2,2),(2,6,4)\}, \\
\mathcal{H}_4&=\mathcal{H}_3 \cup \{(0,0,5),(0,0,6),(5,0,0),(6,0,0),(8,4,0),(4,8,0),(3,3,3),(3,9,6)\}, \\
\mathcal{H}_5&=\mathcal{H}_4  \cup \{(0,0,7),(0,0,8),(7,0,0),(8,0,0),(10,5,0),(5,10,0),(4,4,4),(4,12,8)\}. 
\end{align*} 
\end{small}
{ From Table~\ref{summary_choiceoflags} we read off roughly stable results across all choices.}
As to the computational burden inherent with the choice of lags we observe from Table~\ref{table:comptimes} that computing times increase roughly linearly with $|\mathcal{H}|$; more precisely, computing times approximately double when $|\mathcal{H}|$ doubles. 
Hence, it is advisable to choose $\mathcal{H}$ such that its cardinality is minimal across a selection of valid choices. 
\begin{center}\begin{small}
\captionsetup{type=table}
\begin{tabular}{c|c|c|c|c|c|c|c|c|c|c|c}
& \tiny{TRUE} & $M_1$ & $M_2$ & $M_3$ & $M_4$ & $M_5$ & $R_1$ & $R_2$ & $R_3$ & $R_4$ & $R_5$   \\
\hline
$\wh{C}_1$ & 0.8 & 0.776 & 0.789 & 0.798 & 0.804 & 0.810 & 0.140 & 0.179 & 0.182 & 0.184 & 0.185\\
$\wh{C}_2$ & 0.4 & 0.399 & 0.399 & 0.399 & 0.400 & 0.402 & 0.099 & 0.101 & 0.103 & 0.104 & 0.106\\
$\wh{\alpha}_1$ & 1.5 & 1.490  & 1.462 & 1.436 & 1.418 & 1.403 &  0.074 & 0.119 & 0.114 & 0.130 & 0.145\\
$\wh{\alpha}_2$ & 1 & 0.990 & 0.991 & 0.986 & 0.984 & 0.979 & 0.084 & 0.084 & 0.072 & 0.075 & 0.080
\end{tabular}\end{small}
\captionof{table}{True parameter values (first column), means $M_1-M_5$ and RMSEs $R_1-R_5$ of the estimates of the parameters of model (i) based on the different sets of lags $\mathcal{H}_1-\mathcal{H}_5$.}
\label{summary_choiceoflags}
\end{center}

\enlargethispage{3\baselineskip}
\begin{center}
\captionsetup{type=table}
\begin{tabular}{c|c|c|c|c|c}
$\ell$ & $1$ & $2$ & $3$ & $4$ & $5$ \\
\hline
Computing time in seconds & 3.2 & 8.0 & 15.3 & 21.9 & 28.1 
\end{tabular}
\captionof{table}{Average computing times for one realisation of the Brown-Resnick model (I) based on the sets of lags $\mathcal{H}_{\ell}$ for $\ell=1,\ldots,5$.}
\label{table:comptimes}
\end{center}

\subsection*{(b) Effect of the sample size}

We extend the simulation scenario (i) by repeating the procedure with an increased sample size. Since the number of spatial points is considered as fixed, this involves an increase of the number of time points. In a first run, the observation area is now given by $\cald_n=\calf \times \cali_n$ with $\calf=\{1,\ldots,15\}^2$ and $\cali_n=\{1,\ldots,500\}$; i.e., the process is observed at 500 time points (instead of 300 as before). In a second run, the time points are extended to $\cali_n=\{1,\ldots,1000\}$.
Compared to the original scenario, everything else remains unchanged; in particular, as the large quantile $q$ we choose as before the $96\%$-quantile of $\{\eta(\bs s,t): (\bs s,t) \in \cald_n\}$. 

With regard to the results summarised in Tables~\ref{summary1_500} and \ref{summary1_1000}, we notice that there is no significant change in mean; the confidence bounds (cf. Figure~\ref{sim_I}) are too wide to support such a hypothesis. However, the RMSE and the MAE (and thus the empirical standard deviation) of the estimates decrease considerably. 
This is not an unexpected behaviour: since we do not change $q$, we increase the number of observed points used for the estimation of the empirical extremogram and thus decrease its variance without introducing additional bias. 
In theory, we expect from Theorem~\ref{GLSEcons} and Remark~\ref{summary_cases} that an increase of the number of time points by a factor $k$ leads to a decrease of the standard deviation of the estimates by a factor $f_k(\beta_1)=(1/k)^{(w-\beta_1d)/2}=(1/k)^{(1-3\beta_1)/2}$ for $\beta_1 \in  (w/(5d),w/(2d)) = (1/15,1/6)$, possibly after a bias correction. The extensions from 300 to 500 and that from 300 to 1000 time points correspond to $k=5/3$ and $k=10/3$, respectively. The theoretical factors $f_k(\cdot)$ for $k=5/3$ and $k=10/3$ therefore lie in the intervals $(0.81,0.88)$ and $(0.62,0.74)$, respectively. This behaviour should be confirmed by the empirical standard deviation and related measures. Indeed, dividing the RMSE of the individual estimates of the four parameters based on 500 and 1000 time points by the RMSE based on 300 time points, we obtain factors $0.80$, 0.82, 0.89, 0.75 (mean value $0.82$) and $0.70$, 0.65, 0.70, 0.55 (mean value $0.65$), which all lie in the corresponding theoretical intervals or are close to them. Reasons for slight deviations from theory are of course { sampling variability and} the fact that in practice, the sequence $m_n$ is, as explained above, chosen as a large empirical quantile of the observations. 
Our findings are visualised in Figure~\ref{sim_I_ext}.
\begin{center}\begin{small}
\captionsetup{type=table}
\begin{tabular}{c|c|c|c|c|c}
& TRUE & MEAN & MAE & RMSE & {REL}\\
\hline
$\wh{C}_1$ & 0.8 & 0.7819 & 0.1057 & 0.1410 & { 0.1763}\\
$\wh{C}_2$ & 0.4 & 0.3938 & 0.0628 & 0.0819 & { 0.2048}\\
$\wh{\alpha}_1$ & 1.5 & 1.4549 & 0.0793 & 0.1011 & { 0.0674}\\
$\wh{\alpha}_2$ & 1 & 1.0015 & 0.0464 & 0.0613 & { 0.0613}
\end{tabular}\end{small}
\captionof{table}{True parameter values (first column) and mean, MAE, RMSE, { and REL} of the estimates of the parameters of model (i) based on an increased number of 500 time points.}
\label{summary1_500}
\end{center}

\begin{center}\begin{small}
\captionsetup{type=table}
\begin{tabular}{c|c|c|c|c|c}
& TRUE & MEAN & MAE & RMSE & {REL}\\
\hline
$\wh{C}_1$ & 0.8 & 0.7584 &  0.0995 & 0.1241 & {0.1552}\\
$\wh{C}_2$ & 0.4 & 0.3848 & 0.0522 & 0.0647 & {0.1618}\\
$\wh{\alpha}_1$ & 1.5 & 1.4504 & 0.0644 & 0.0788 & {0.0525}\\
$\wh{\alpha}_2$ & 1 & 0.9858 & 0.0348 & 0.0453 & {0.0453}
\end{tabular}\end{small}
\captionof{table}{True parameter values (first column) and mean, MAE, RMSE,  {and REL} of the estimates of the parameters of model (i) based on an increased number of 1000 time points.}
\label{summary1_1000}
\end{center}

\begin{figure}[H]
\centering
\includegraphics[scale=0.27]{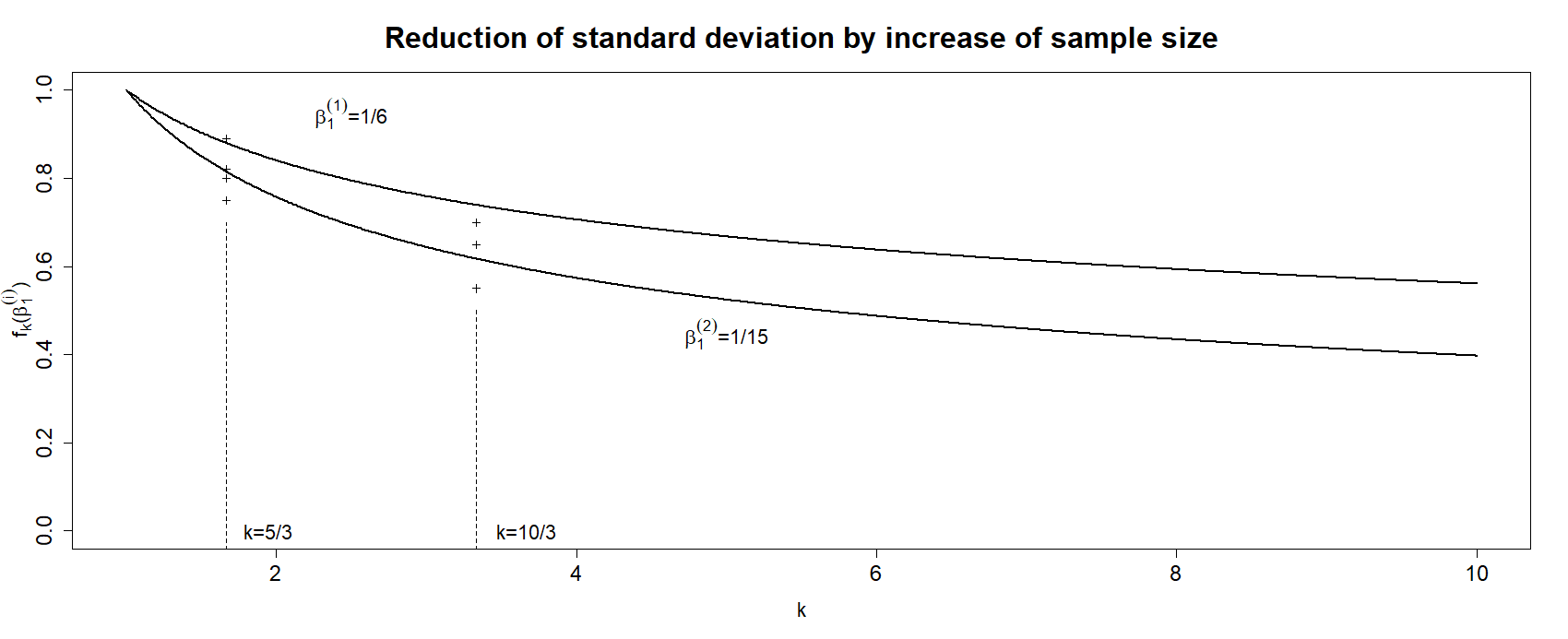}
\caption{{Theoretical minimum and maximum  factors $f_k(\beta_1^{(1)})$ and $f_k(\beta_1^{(2)})$ of decrease of the standard deviation for $\beta_1^{(1)}=1/6$ and $\beta_1^{(2)}=1/15$ (solid curves). The $+$ symbols correspond to the empirical RMSE reduction factors of the four individual paramater estimates, when the number of time points is increased by factors $k=5/3$ and $k=10/3$.}} \label{sim_I_ext}
\end{figure}
}

\subsection*{Acknowledgements}
Sven Buhl acknowledges support by the Deutsche Forschungsgemeinschaft (DFG) through the TUM  International Graduate School of Science and Engineering (IGSSE).


\newpage

\appendix

\section{Appendix}

\subsection{$\alpha$-mixing with respect to the increasing dimensions}

We need the concept of $\alpha$-mixing for the process $\{X(\bs{s}): \bs{s}\in \bbr^d\}$ with respect to $\R^w$. 
In a space-time setting with fixed spatial setting and increasing time series this is called \textit{temporal} $\alpha$-mixing.

\begin{definition}[$\alpha$-mixing and $\alpha$-mixing coefficients] \label{mixing}
Consider a strictly stationary process $\left\{X(\bs{s}): \bs{s}\in \bbr^d\right\}$ and let  $\|\cdot\|$ be some norm on $\mathbb{R}^d$.
For $\Lambda_1, \Lambda_2 \subset \mathbb{Z}^w$ define 
\begin{align*}
d(\Lambda_1,\Lambda_2) := \inf\left\{\|\bs{s}_1-\bs{s}_2\|: \ \bs{s}_1 \in \mathcal{F} \times \Lambda_1, \bs{s}_2 \in \mathcal{F} \times \Lambda_2\right\}.
\end{align*}
Further, for $i=1,2$ denote by $\sigma_{\calf\times \Lambda_i}= \sigma\left\{X(\bs{s}):  \bs{s}\in \mathcal{F} \times \Lambda_i\right\}$ the $\sigma$-algebra generated by $\{X(\bs{s}): \ \bs{s}\in \mathcal{F} \times \Lambda_i\}$.
\begin{enumerate}[(i)]
\item 
We define the {\em $\alpha$-mixing coefficients with respect to $\mathbb{R}^w$} for $k_1,k_2 \in \mathbb{N}$ and $z \geq 0$ as
\begin{equation}\label{alphaBolt}
\hspace*{-1.1cm}
\alpha_{k_1,k_2}(z) := \sup\left\{\left|\mathbb{P}(A_1\cap A_2) - \mathbb{P}(A_1)\mathbb{P}(A_2)\right|: \ A_i \in \sigma_{\calf \times\Lambda_i}, |\Lambda_i|\leq k_i,  d(\Lambda_1,\Lambda_2) \geq z\right\}.
\end{equation}
\item 
We call $\{X(\bs{s}): \bs{s}\in \bbr^d\}$  {\em $\alpha$-mixing with respect to $\mathbb{R}^w$}, if $\alpha_{k_1,k_2}(z) \to 0$ as $z\to\infty$ for all $k_1,k_2\in \mathbb{N}$.
\end{enumerate}
\end{definition}

We have to control the dependence between vector processes $\{\bs Y(\bs s)=X_{\mathcal{B}(\bs s,\ga)}: \bs s \in \Lambda_1'\}$ and $\{\bs Y(\bs s)=X_{\mathcal{B}(\bs s,\ga)}: \bs s \in \Lambda_2'\}$ for subsets $\Lambda_i' \subset \mathbb{Z}^w$ with cardinalities $|\Lambda_1'| \leq k_1$ and $|\Lambda_2'| \leq k_2.$. This entails dealing with unions of balls $\Lambda_i=\cup_{\bs s \in \mathcal{F}\times\Lambda_i'} \mathcal{B}(\bs s,\ga)$. 
Since $\ga>0$ is some predetermined finite constant independent of $n$, we keep notation simple by redefining the $\alpha$-mixing coefficients corresponding to the vector processes for $k_1,k_2 \in \mathbb{N}$ and $z \geq 0$ as
\begin{align}\label{alpha_balls}
\alpha_{k_1,k_2}(z) := & \sup\{|\mathbb{P}(A_1\cap A_2) - \mathbb{P}(A_1)\mathbb{P}(A_2)|: \nonumber\\
& \quad\quad A_i \in \sigma_{\Lambda_i},\,\,\Lambda_i=\cup_{\bs s \in \mathcal{F}\times\Lambda_i'} \mathcal{B}(\bs s,\ga),\,\, |\Lambda_i'|\leq k_i, d(\Lambda_1',\Lambda_2') \geq z\}.
\end{align}

\subsection{Proof of Theorem~\ref{stasyn}}\label{app_1}

The proof of Theorem~\ref{stasyn} is divided into two parts. 
In the first part we prove a LLN and a CLT in Lemmas~\ref{st-asy_la} and \ref{asym_normal} for the estimators $\widehat{\mu}_{\mathcal{B}(\bs 0,\ga),m_n}$ in \eqref{estmu}. In the second part of the proof we derive the CLT for the empirical extremogram $\wh\rho_{AB,m_n}$ in \eqref{EmpEst}, and compute the asymptotic covariance matrix $\Pi$. 
The proof generalizes corresponding proofs in \citet{buhl3} (where the observation area increases in all dimensions) in a non-trivial way.
We recall the separation of every point and every lag in its components corresponding to the fixed domain, indicated by the sub index $\calf$, and the remaining components, indicated by $\cali$, from Assumption~\ref{ass0}. 
In particular, we decompose $\bs h^{(i)}=(\bs h_\calf^{(i)},\bs h_{\cali}^{(i)})\in\calh$.

The separation of the observation space with its fixed domain has to be introduced into the proofs given in \cite{buhl3}, which is even in the regular grid situation highly non-trivial. 
We will give detailed references to those proofs, whenever possible, to support the understanding.
On the other hand, if arguments  just follow a previous proof line by line we avoid the details.\\

\noindent 
\textbf{Part I: } LLN and CLT for $\widehat{\mu}_{\mathcal{B}(\bs 0,\ga),m_n}$  \\[2mm]
As in \cite{buhl3}, Section~5, we make use of a large/small block argument. 
For simplicity we assume that $n^w/m_n^d$ is an integer and subdivide $\mathcal{D}_n$ into $n^w/m_n^d$ non-overlapping $d$-dimensional large blocks $\mathcal{F} \times \mathcal{B}_i$ for $i=1,\ldots,n^w/m_n^d$, where the $\mathcal{B}_i$ are $w$-dimensional cubes with side lengths $m_n^{d/w}$.
 From those large blocks we then cut off smaller blocks, which consist of the first $r_n$ elements in each of the $w$ increasing dimensions. The large blocks are then separated (by these small blocks) with at least the distance $r_n$ in all $w$ increasing dimensions and shown to be asymptotically independent. 

We divide the lags in $L_n$ into different sets according to the large and small blocks. 
Recall the notation of \eqref{lag_fix_2} and around.
Observe that a lag $(\bs \ell_{\mathcal{F}},\bs \ell_{\mathcal{I}})$ with $\bs \ell_{\mathcal{I}}=(\ell_{\mathcal{I}}^{(1)},\ldots,\ell_{\mathcal{I}}^{(w)})$ appears in $L_{\mathcal{F}}^{(i,i)} \times L_n$ exactly $\textnormal{N}_{\mathcal{F}}^{(i,i)}(\bs \ell_{\mathcal{F}})\prod_{j=1}^w(n-|\ell_{\mathcal{I}}^{(j)}|)$ times, where $\textnormal{N}_{\mathcal{F}}^{(i,i)}(\bs \ell_{\mathcal{F}})$ is defined in~\eqref{lag_appearance_fix}.
This term will replace $\prod_{j=1}^d (n-|h_j|)$ in the proofs of \cite{buhl3}.

\begin{lemma} \label{st-asy_la}
Let $\{X(\bs s): \bs s \in \mathbb{R}^d\}$ be a strictly stationary regularly varying process observed on $\cald_n=\mathcal{F} \times \mathcal{I}_n$ as in \eqref{observed}. 
For $i \in \{1,\ldots,p\}$, let $\bs h^{(i)}=(\bs h_{\mathcal{F}}^{(i)},\bs h_{\mathcal{I}}^{(i)}) \in \mathcal{H}\subseteq \mathcal{B}(\bs 0, \ga)$ for some $\ga>0$ be a fixed lag vector and use as before the convention that $(\bs h_{\mathcal{F}}^{(p+1)},\bs h_{\mathcal{I}}^{(p+1)})=\bs 0$.
Suppose that the following  mixing conditions are satisfied.
\begin{enumerate}[(1)]
\item 
$\{X(\bs{s}): \bs{s}\in\bbr^d\}$ is $\alpha$-mixing with respect to $\R^w$ with mixing coefficients $\alpha_{k_1,k_2}(\cdot)$ defined in \eqref{alphaBolt}.
\item
There exist sequences $m:=m_n, r:=r_\nto$ with $m_n^d/n^w \to 0$ and $r_n^w/m_n^d \to 0$ as $\nto$ such that (M3) and (M4i) hold.
\end{enumerate}
Then for every fixed $i=1,\ldots,p+1$, as $\nto$,
\begin{align}
\mathbb{E}\big[\widehat{\mu}_{\mathcal{B}(\bs 0,\ga),m_n}(D_i)\big] & \to \mu_{\mathcal{B}(\bs 0,\ga)}(D_i),\label{consist1}\\
\var\big[\widehat{\mu}_{\mathcal{B}(\bs 0,\gamma),m_n}(D_i)\big] 
&\sim 
\frac{m_n^d}{n^w}\sigma_{\mathcal{B}(\bs 0,\gamma)}^2(D_i),\label{varpm2}
\end{align}
with $\sigma_{\mathcal{B}(\bs 0,\gamma)}^2(D_i)$ specified in \eqref{varpm3}.
If $\mu_{\mathcal{B}(\bs 0,\ga)}(D_i)=0$, then \eqref{varpm2} is interpreted as \linebreak $\var\big[\widehat{\mu}_{\mathcal{B}(\bs 0,\gamma),m_n}(D_i)\big] = o(m_n^d/n^w)$.
In particular, 
\begin{align}\label{lln}
\widehat{\mu}_{\mathcal{B}(\bs 0,\ga),m_n}(D_i) \stp \mu_{\mathcal{B}(\bs 0,\ga)}(D_i), \quad \nto .
\end{align}
\end{lemma}

\bproof[Proof of Lemma~\ref{st-asy_la}.]
We suppress the superscript $(i)$ of $\bs h^{(i)}$ (respectively $\bs h_{\mathcal{F}}^{(i)}$) for notational ease. 
Strict stationarity and relation \eqref{B6} imply that 
\begin{align*}
\mathbb{E}\Big[\widehat{\mu}_{\mathcal{B}(\bs 0,\ga),m_n}(D_i)\Big]
=\frac{m_n^d}{n^w} \sum_{\bs i \in \mathcal{I}_n} \frac{|\mathcal{F}(\bs h_{\mathcal{F}})|}{|\mathcal{F}(\bs h_{\mathcal{F}})|}  \mathbb{P}\Big(\frac{\bs Y(\bs 0)}{a_m} \in D_i\Big) 
= m_n^d\mathbb{P}\Big(\frac{\bs Y(\bs 0)}{a_m} \in D_i\Big) \to \mu_{\mathcal{B}(\bs 0,\ga)}(D_i).
\end{align*}
As to the asymptotic variance, we start from \eqref{varpm0}, where it has been calculated that
\begin{align}
\var\big[\widehat{\mu}_{\mathcal{B}(\bs 0,\ga),m_n}(D_i)\big] 
&=\frac{m_n^{2d}}{n^{2w}|\mathcal{F}(\bs h_{\mathcal{F}})|^2} \Big(|\mathcal{F}(\bs h_{\mathcal{F}})| n^w\var\big[\mathbbmss{1}_{\{\frac{\bs Y(\bs 0)}{a_m} \in D_i\}}\big] \notag\\
&\quad\quad\quad 
+\sum_{\bs f, \bs f' \in \mathcal{F}(\bs h_{\mathcal{F}})}\sum_{\bs i, \bs i' \in \mathcal{I}_n \atop (\bs f,\bs i) \neq (\bs f',\bs i')} 
\cov\big[\mathbbmss{1}_{\{\frac{\bs Y(\bs f, \bs i)}{a_m} \in D_i\}},\mathbbmss{1}_{\{\frac{\bs Y(\bs f', \bs i')}{a_m} \in D_i\}}\big]\Big) \notag\\
&=: A_1+A_2. \label{varpm}
\end{align}
By \eqref{B6} and since $\mathbb{P}(\bs Y(\bs 0)/a_m \in D_i) \to 0$,
\begin{align*}
A_1&= \frac{m_n^{2d}}{n^w |\mathcal{F}(\bs h_{\mathcal{F}})|}\mathbb{P}\Big(\frac{\bs Y(\bs 0)}{a_m} \in D_i\Big) \Big(1-\mathbb{P}\Big(\frac{\bs Y(\bs 0)}{a_m} \in D_i\Big)\Big) \sim \frac{m_n^d}{n^w |\mathcal{F}(\bs h_{\mathcal{F}})|} \mu_{\mathcal{B}(\bs 0,\ga)}(D_i) {\to 0},\quad \nto.
\end{align*} 
Counting the lags as explained above this proof, for fixed $k \in\N$ we have by stationarity the analogy of (5.6) in \cite{buhl3}
\beam
\frac{n^w}{m_n^d} A_2
&=&\frac{m_n^d}{|\mathcal{F}(\bs h_{\mathcal{F}})|^2} \Big(\sum_{\bs \ell_{\mathcal{I}} \in L_n \atop 0 \leq \|\bs \ell_{\mathcal{I}}\| \leq k} + \sum_{\bs \ell_{\mathcal{I}} \in L_n \atop k <  \|\bs \ell_{\mathcal{I}}\| \leq r_n} + \sum_{\bs \ell_{\mathcal{I}} \in L_n \atop \|\bs \ell_{\mathcal{I}}\| > r_n} \Big) \nonumber\\ & & \sum_{\bs \ell_{\mathcal{F}} \in L_{\mathcal{F}}^{(i,i)}  \atop (\bs \ell_{\mathcal{F}}, \bs \ell_{\mathcal{I}}) \neq \bs 0} \textnormal{N}_{\mathcal{F}}^{(i,i)}(\bs \ell_{\mathcal{F}})\prod_{j=1}^w\big(1-\frac{|\ell_{\mathcal{I}}^{(j)}|}{n}\big) \cov[\mathbbmss{1}_{\{\frac{\bs Y(\bs 0)}{a_m} \in D_i\}},\mathbbmss{1}_{\{\frac{\bs Y(\bs \ell_{\mathcal{F}},\bs \ell_{\mathcal{I}})}{a_m} \in D_i\}}] \nonumber \\
&=:& A_{21}+A_{22}+A_{23}.\label{cov}
\eeam
Concerning $A_{21}$ we have, 
\begin{align*}
A_{21} =&\frac{m_n^d}{|\mathcal{F}(\bs h_{\mathcal{F}})|^2} \sum_{\bs \ell_{\mathcal{I}} \in L_n \atop 0 \leq \|\bs \ell_{\mathcal{I}}\| \leq k} \sum_{\bs \ell_{\mathcal{F}} \in L_{\mathcal{F}}^{(i,i)}  \atop (\bs \ell_{\mathcal{F}}, \bs \ell_{\mathcal{I}}) \neq \bs 0} \textnormal{N}_{\mathcal{F}}^{(i,i)}(\bs \ell_{\mathcal{F}})\prod_{j=1}^w\big(1-\frac{|\ell_{\mathcal{I}}^{(j)}|}{n}\big)\\ 
& \Big[\mathbb{P}\Big(\frac{\bs Y(\bs 0)}{a_m} \in D_i,\frac{\bs Y(\bs \ell_{\mathcal{F}},\bs \ell_{\mathcal{I}})}{a_m} \in D_i\Big)-\mathbb{P}\Big(\frac{\bs Y(\bs 0)}{a_m} \in D_i\Big)^2\Big].
\end{align*}
With \eqref{B6} and \eqref{B7} we obtain by dominated convergence, 
\begin{align}\label{dom}
\lim_{\kto } \limsup_{n \to \infty} A_{21} =\frac{1}{|\mathcal{F}(\bs h_{\mathcal{F}})|^2}\sum_{\bs \ell_{\mathcal{I}} \in \mathbb{Z}^w} \sum_{\bs \ell_{\mathcal{F}} \in L_{\mathcal{F}}^{(i,i)} \atop (\bs \ell_{\mathcal{F}},\bs \ell_{\mathcal{I}}) \neq \bs 0} \textnormal{N}_{\mathcal{F}}^{(i,i)}(\bs \ell_{\mathcal{F}}) \tau_{\mathcal{B}(\bs{0},\ga)\times \mathcal{B}((\bs \ell_{\mathcal{F}},\bs \ell_{\mathcal{I}}),\ga)}(D_i \times D_i).
\end{align}
As to $A_{22}$, observe that for all $n \geq 0$ we have $\prod\limits_{j=1}^w (1-\frac{|\ell_{\mathcal{I}}^{(j)}|}{n})\leq 1$ for $\bs \ell_{\mathcal{I}}\in L_n$. 
Furthermore, since $D_i$ is bounded away from $\bs 0$, there exists $\epsilon>0$ such that $D_i \subset \{\bs x \in \ov{\mathbb{R}}^{|\mathcal{B}(\bs 0,\ga)|}: \| \bs x \| > \epsilon\}.$
Hence, we obtain 
\begin{align*}
|A_{22}| 
\leq & \frac{1}{|\mathcal{F}(\bs h_{\mathcal{F}})|^2} \sum_{\bs \ell_{\mathcal{F}} \in L_{\mathcal{F}}^{(i,i)}} \textnormal{N}_{\mathcal{F}}^{(i,i)}(\bs \ell_{\mathcal{F}}) \sum_{\bs \ell_{\mathcal{I}}\in \mathbb{Z}^w \atop k < \Vert{\bs \ell_{\mathcal{I}}}\Vert \leq r_n} \Big\{
m_n^d  \mathbb{P}\Big(\|\bs Y(\bs 0)\|>\epsilon a_m, \|\bs Y(\bs \ell_{\mathcal{F}},\bs \ell_{\mathcal{I}})\|>\epsilon a_m\Big) \\
& + \frac1{m_n^d} \Big(m_n^d \mathbb{P}\Big(\frac{\bs Y(\bs 0)}{a_m} \in D_i\Big)\Big)^2\Big\}.
\end{align*}
which differs from the corresponding expression in \cite{buhl3} only by finite factors.
Thus by an obvious modification of the arguments in that paper it follows that, using $r_n^w/m_n^d \to 0$ and condition~(M3),
$$\lim_{\kto }\limsup_{\nto } A_{22}=0.$$ 
Using the definition~\eqref{alpha_balls} of $\alpha$-mixing for $A_1=\{\bs Y(\bs 0)/a_m \in D_i\}$ and $A_2=\{\bs Y(\bs \ell_{\mathcal{F}}, \bs \ell_{\mathcal{I}})/a_m \in D_i\}$, we obtain by (M4i),
\begin{align}\label{a23}
|A_{23}| 
& \leq \frac{1}{|\mathcal{F}(\bs h_{\mathcal{F}})|^2} \sum_{\bs \ell_{\mathcal{F}} \in L_{\mathcal{F}}^{(i,i)}} \textnormal{N}_{\mathcal{F}}^{(i,i)}(\bs \ell_{\mathcal{F}}) m_n^d \sum_{\bs \ell_{\mathcal{I}} \in \mathbb{Z}^w: \|\bs \ell_{\mathcal{I}}\| > r_n} \alpha_{1,1}(\|\bs \ell_{\mathcal{I}}\|) \to 0,\quad n \to \infty.
\end{align}
Summarising these computations, we conclude from \eqref{cov} and \eqref{dom} that for $\nto$,
$$A_2 \sim \frac{m_n^d}{n^w} \sum_{\bs \ell_{\mathcal{I}} \in \mathbb{Z}^w} \frac{1}{|\mathcal{F}(\bs h_{\mathcal{F}})|^2} \sum_{\bs \ell_{\mathcal{F}} \in L_{\mathcal{F}}^{(i,i)} \atop (\bs \ell_{\mathcal{F}}, \bs \ell_{\mathcal{I}}) \neq \bs 0} \textnormal{N}_{\mathcal{F}}^{(i,i)}(\bs \ell_{\mathcal{F}}) \tau_{\mathcal{B}(\bs 0,\gamma) \times \mathcal{B}((\bs \ell_{\mathcal{F}},\bs \ell_{\mathcal{I}}),\gamma)}(D_i \times D_i),$$
 and, therefore, \eqref{varpm} implies \eqref{varpm2}.
Since ${m_n^d/n^w\to 0}$ as $\nto$, equations~\eqref{consist1} and \eqref{varpm2} imply \eqref{lln}. 
\eproof

\ble\label{asym_normal}
Let $\{X(\bs s): \bs s \in \mathbb{R}^d\}$ be a strictly stationary regularly varying  process observed on $\mathcal{D}_n=\mathcal{F} \times \mathcal{I}_n.$ 
For $i \in \{1,\ldots,p\}$, let $\bs h^{(i)}=(\bs h_{\mathcal{F}}^{(i)},\bs h_{\mathcal{I}}^{(i)}) \in \mathcal{H}\subseteq \mathcal{B}(\bs 0, \ga)$ for some $\ga>0$ be a fixed lag vector and take as before the convention that $(\bs h_{\mathcal{F}}^{(p+1)},\bs h_{\mathcal{I}}^{(p+1)})=\bs 0$.
Let the assumptions of Theorem~\ref{stasyn} hold.
Then for every fixed $i=1,\ldots,p+1$,
\begin{align}\label{cltpm}
\wh S_{\mathcal{B}(\bs 0,\gamma),m_n} & :=\sqrt{\frac{m_n^d}{n^w}} \sum_{\bs i \in \mathcal{I}_n} 
\Big[\frac{1}{|\mathcal{F}(\bs h_{\mathcal{F}})|} \Big( \sum_{\bs f \in \mathcal{F}(\bs h_{\mathcal{F}})} \mathbbmss{1}_{\{\frac{\bs Y(\bs f, \bs i)}{a_m} \in D_i\}} \Big)-\mathbb{P}\Big(\frac{\bs Y(\bs f, \bs i)}{a_m} \in D_i\Big)\Big]\nonumber\\
&=\sqrt{\frac{n^w}{m_n^d}} \big[\widehat{\mu}_{\mathcal{B}(\bs 0,\ga),m_n}(D_i)-\mu_{\mathcal{B}(\bs 0,\ga),m_n}(D_i)\big] \std \mathcal{N}(0,\sigma_{\mathcal{B}(\bs 0,\gamma)}^2(D_i)), \quad \nto ,
\end{align}
with $\widehat{\mu}_{\mathcal{B}(\bs 0,\ga),m_n}(D_i)$ as in \eqref{estmu}, $\mu_{\mathcal{B}(\bs 0,\ga),m_n}(D_i)):=m_n^d \mathbb{P}(\bs Y(\bs 0)/a_m \in D_i)$ and $\sigma_{\mathcal{B}(\bs 0,\gamma)}^2(D_i)$ given in~\eqref{varpm3}.
\ele

\bproof
Again we suppress the superscript $(i)$ of $\bs h^{(i)}$ and $\bs h_{\mathcal{F}}^{(i)}$. 
As for the proof of consistency above, we generalise the proof of the CLT in \cite{buhl3} (based on \citet{Bolthausen}) to the new setting. 
We consider the process 
$$\Big\{\frac{\sqrt{m_n^d}}{|\mathcal{F}(\bs h_{\mathcal{F}})|} \Big( \sum_{\bs f \in \mathcal{F}(\bs h_{\mathcal{F}})} \mathbbmss{1}_{\{\frac{\bs Y(\bs f, \bs i)}{a_m} \in D_i\}} \Big): \bs i \in \mathbb{Z}^w\Big\},$$
observed on the $w$-dimensional regular grid $\mathcal{I}_n$.
In analogy to (5.11) in \cite{buhl3} define 
\beam\label{I}
I(\bs i):=\frac{1}{|\mathcal{F}(\bs h_{\mathcal{F}})|} \Big( \sum_{\bs f \in \mathcal{F}(\bs h_{\mathcal{F}})} \mathbbmss{1}_{\{\frac{\bs Y(\bs f, \bs i)}{a_m} \in D_i\}}\Big)-\mathbb{P}\Big(\frac{\bs Y(\bs 0)}{a_m} \in D_i\Big),\quad \bs i \in \mathcal{I}_n,
\eeam
and note that by stationarity,
\begin{align}\label{smpm}
\wh S_{\mathcal{B}(\bs 0,\gamma),m_n} 
&=\sqrt{\frac{m_n^d}{n^w}} \sum_{\bs i\in \mathcal{I}_n } I(\bs i).
\end{align}
The boundary condition required in Eq.~(1) in \citet{Bolthausen} is satisfied for the regular grid $\cali_n$. By the same arguments as in \cite{buhl3},
\begin{align}\label{covfinite}
0 < \sigma^2_{\mathcal{B}(\bs 0,\ga)}(D_i) \sim \var[\wh S_{\mathcal{B}(\bs 0,\gamma),m_n}] \le \frac{m_n^d}{n^w} \sum_{\bs i,\bs i' \in \mathbb{Z}^w} |\mathbb{E}[I(\bs i)I(\bs i')]| < \infty, 
\end{align}
such that $\sum_{\bs i,\bs i' \in \mathbb{Z}^w} \cov[I(\bs i), I(\bs i')]>0$.
Replacing $\cals_n$ in \cite{buhl3} by $\mathcal{I}_n$ and $n^d$ by $n^w$, we define
\beam\label{ssm}
 v_n:=\frac{m_n^d}{n^w} \ \sum_{\bs i, \bs i' \in \mathcal{I}_n \atop \|\bs i-\bs i'\| \leq {r_n}}  \mathbb{E}\Big[ I(\bs i)  I(\bs i')\Big].
\eeam
and obtain by the same arguments that
\begin{align*}
\frac{v_n}{\var{[\wh S_{\mathcal{B}(\bs 0,\gamma),m_n}]}}
&=1-\frac{m_n^d}{n^w} \frac1{\si^2_{\mathcal{B}(\bs 0,\gamma)}(D_i)}
\sum\limits_{\bs i,\bs i'\in \mathcal{I}_n \atop \|\bs i- \bs i'\| > r_n}  \E[I(\bs i) I(\bs i')](1+o(1)).
\end{align*}
Now note that 
\begin{align*}
\frac{m_n^d}{n^w} 
\sum_{\bs i, \bs i'\in \mathcal{I}_n \atop \|\bs i- \bs i'\| > r_n} \E[I(\bs i) I(\bs i')] 
\leq &\frac{1}{|\mathcal{F}(\bs h_{\mathcal{F}})|^2}   \sum_{\bs \ell_{\mathcal{F}} \in L_{\mathcal{F}}^{(i,i)}} \textnormal{N}_{\mathcal{F}}^{(i,i)}(\bs \ell_{\mathcal{F}})m_n^d\sum_{\bs \ell_{\mathcal{I}} \in \mathbb{Z}^q: \|\bs \ell_{\mathcal{I}}\| > r_n}  \alpha_{1,1}(\| \bs \ell_{\mathcal{I}} \|) \rightarrow 0, \quad \nto,
\end{align*}
as in \eqref{a23}, with mixing coefficients defined in~\eqref{alpha_balls}.
Therefore, 
\begin{align} \label{preasyvar}
v_n \sim \var[\wh S_{\mathcal{B}(\bs 0,\gamma),m_n}] \rightarrow \si^2_{\mathcal{B}(\bs 0,\gamma)}(D_i), \quad \nto.
\end{align}
The standardized quantities are again as in \cite{buhl3}, with $\cals_n$ replaced by $\cali_n$ and $n^d$ by $n^w$, by
$$\overline{S}_{n}:=v_n^{-1/2}\wh S_{\mathcal{B}(\bs 0,\gamma),m_n}=  v_n^{-1/2}\sqrt{\frac{m_n^d}{n^w}} \sum_{\bs i\in \mathcal{I}_n } I(\bs i)\quad\mbox{and}\quad\overline{S}_{\bs i,n}:=v_n^{-1/2}\sqrt{\frac{m_n^d}{n^w}}  \sum_{\bs i' \in \mathcal{I}_n \atop \|\bs i- \bs i'\| \leq r_n} I(\bs i').$$
The proof continues in \cite{buhl3}, with $n^d$ replaced by $n^w$, by estimating the quantities $B_1$, $B_2$ and $B_3$.
The estimation of $B_1$ follows the same lines of the proof, resulting in
\beao
E[|B_1|^2]=\lambda^2 v_n^{-2} \Big(\frac{m_n^d}{n^w}\Big)^{2} \sum_{  \|\bs i-\bs i'\| \leq { r_n}} \sum_{  \|\bs j- \bs j'\| \leq {r_n}} \cov\big[I(\bs i)I(\bs i'),I(\bs j)I(\bs j') \big].
\eeao
We use definition~\eqref{alpha_balls} of the $\alpha$-mixing coefficients for
$$\Lambda_1'=\{\bs i,\bs i'\} \quad\mbox{and}\quad \Lambda_2'=\{\bs j,\bs j'\},$$
then  $|\Lambda_1'|,|\Lambda_2'| \leq 2$ and
for $d(\Lambda_1',\Lambda_2')$ we consider the following two cases:
\begin{enumerate}[(1)]
\item 
$\|\bs i-\bs j\| \geq 3 r_n.$ Then $2r_n \leq (2/3) \|\bs i-\bs j\|$ and  $d(\Lambda_1',\Lambda_2')\ge \|\bs i- \bs j\|-2 r_n$.
Since indicator variables are bounded and $\alpha_{2,2}$ is a decreasing function, 
\begin{align*}
|\cov\big[I(\bs i)I(\bs i'),I(\bs j)I(\bs j') \big]| &\leq 4 \alpha_{2,2}\Big(\|\bs i-\bs j\|-2r_n\Big)
\leq 4 \alpha_{2,2}\Big(\frac1{3} \|\bs i-\bs j\|\Big).
\end{align*}
\item 
$\|\bs i-\bs j\| {< 3 r_n.}$ 
Set 
$z:=\min\{\|\bs i- \bs j\|, \|\bs i- \bs j'\|,\|\bs i'- \bs j\|, \|\bs i'- \bs j'\|\}$,
then $d(\Lambda_1',\Lambda_2')\ge z$ and, hence, 
$$\cov\big[I(\bs i)I(\bs i'),I(\bs j)I(\bs j') \big] \leq 4 \alpha_{k_1,k_2}(z), \quad 2 \leq k_1+k_2 \leq 4.$$ 
\end{enumerate}
Therefore,
\begin{align*}
E[|B_1|^2] 
\leq & \frac{4\lambda^2}{v_n^{2}} \Big(\frac{m_n^d}{n^w}\Big)^2  \Big[\sum_{  \|\bs i- \bs j\| {\geq 3 r_n}} 
 \sum_{{ \|\bs i- \bs i'\| 
 \leq { r_n}} \atop \|\bs j- \bs j'\| 
 \leq { r_n} } \alpha_{2,2}\Big(\frac1{3} \|\bs i- \bs j\|\Big)
+\sum_{\|\bs i- \bs j\| {< 3 r_n}} \sum_{{ \|\bs i- \bs i'\| \leq { r_n}} \atop 
\|\bs j- \bs j'\| \leq {r_n} } \alpha_{k_1,k_2}(z)\Big]\\
\leq & \frac{4\lambda^2}{v_n^2} \Big(\frac{m_n^d}{n^w}\Big)^2  n^w r_n^{2w} \Big[\sum_{\bs \ell_{\mathcal{I}}\in \mathbb{Z}^{w}: \|\bs \ell_{\mathcal{I}}\| {\geq 3 r_n}} 
 \alpha_{2,2}\Big(\frac1{3} \|\bs \ell_{\mathcal{I}}\|\Big)+\sum_{\bs \ell_{\mathcal{I}}\in \mathbb{Z}^{w}: \|\bs \ell_{\mathcal{I}}\| {< 3 r_n}}  \alpha_{k_1,k_2} (\|\bs \ell_{\mathcal{I}}\|)\Big].
\end{align*}
The analogous argument as in \cite{buhl3} yields
$$E[|B_1|^2]  = {\mathcal{O}\Big(\frac{m_n^{2d} r_n^{2w}}{n^w}\Big)}\to 0.$$ 
Next, $\mathbb{E}[|B_2|] \to 0$ as $\nto$ by the same arguments as in \cite{buhl3} replacing $\cals_n$ by $\cali_n$ and $n^d$ by $n^w$.
Then we find for $B_3$ with the same replacements
\begin{align*}
\mathbb{E}[B_3]
&= v_n^{-\frac{1}{2}} m_n^{d/2} n^{w/2} 
 \mathbb{E}\Big[I(\bs 0) \exp\Big\{i \lambda v_n^{-\frac{1}{2}} \sqrt{\frac{m_n^d}{n^w}} \sum\limits_{\|\bs i \|> r_n} I(\bs i) \Big\} \Big].
\end{align*}
We use definition~\eqref{alpha_balls} of the $\alpha$-mixing coefficients for
$$\Lambda_1'=\{\bs 0\}\quad\mbox{and}\quad \Lambda_2'=\{\bs i \in \mathcal{I}_n:  \|\bs i\|>r_n\},$$
such that $|\Lambda_1'|=1$, $|\Lambda_2'| \leq n^w$ and $d(\Lambda_1',\Lambda_2')>r_n$. 
Abbreviate 
$$\eta(r_n) :=\exp\Big\{i \lambda v_n^{-\frac{1}{2}} \sqrt{\frac{m_n^d}{n^w}} \sum\limits_{ \|\bs i \|> r_n} I(\bs i) \Big\},$$ 
then $I(\bs 0)$ and $\eta(r_n)$ are measurable with respect to ${\sigma}_{\Lambda_1}$ and ${\sigma}_{\Lambda_2}$, respectively, where $\Lambda_i=\cup_{\bs s \in \mathcal{F}\times\Lambda_i'} \mathcal{B}(\bs s,\ga)$.
Now we apply  Theorem~17.2.1 of Ibragimov and Linnik to obtain
$$|\mathbb{E}[B_3]|
\leq  {4 v_n^{-1/2} m_n^{d/2} n^{w/2}\alpha_{1, n^w }(r_n) \to 0},$$
where  convergence to 0 is guaranteed by condition (M4iii).
\end{proof}

\noindent \textbf{Part II: } CLT for $\wh\rho_{AB,m_n}$ and  limit covariance matrix\\[2mm]
Recall the definition of $\mathcal{H}=\{\bs h^{(1)},\ldots,\bs h^{(p)}\}$.
For $i \in \{1,\ldots,p\}$, write $\bs h^{(i)}=(\bs h_{\mathcal{F}}^{(i)},\bs h_{\mathcal{I}}^{(i)})$ with respect to the fixed and increasing domains $\mathcal{F}$ and $\mathcal{I}_n$. Write further $\bs h_{\mathcal{F}}^{(i)}=(h_{\mathcal{F}}^{(i,1)},\ldots,h_{\mathcal{F}}^{(i,q)})$ and $\bs h_{\mathcal{I}}^{(i)}=(h_{\mathcal{I}}^{(i,1)},\ldots,h_{\mathcal{I}}^{(i,w)})$.
Now we define the ratio
$$R_n(D_i,D_{p+1}):=\frac{\mathbb{P}(\bs Y(\bs 0)/a_m \in D_i)}{\mathbb{P}(\bs Y(\bs 0)/a_m \in D_{p+1})}=\frac{\mu_{\mathcal{B}(\bs 0,\gamma),m_n}(D_i)}{\mu_{\mathcal{B}(\bs 0,\gamma),m_n}(D_{p+1})}$$
and the corresponding empirical estimator
\begin{align*}
\widehat{R}_n(D_i,D_{p+1})&:=\frac{|\mathcal{F}| \sum_{\bs i \in \mathcal{I}_n} \sum_{\bs f \in {\mathcal{F}}(\bs h_{\mathcal{F}}^{(i)})}\mathbbmss{1}_{\{\bs Y(\bs f,\bs i)/a_m \in D_i\}}}{|{\mathcal{F}}(\bs h_{\mathcal{F}}^{(i)})|\sum_{\bs i \in \mathcal{I}_n} \sum_{\bs f \in \mathcal{F}}\mathbbmss{1}_{\{\bs Y(\bs f,\bs i)/a_m \in D_{p+1}\}}}\\
&=\frac{\frac{m_n^d}{n^w} \sum_{\bs i\in \mathcal{I}_n} \frac{1}{|{\mathcal{F}}(\bs h_{\mathcal{F}}^{(i)})|} \sum_{\bs f \in {\mathcal{F}}(\bs h_{\mathcal{F}}^{(i)})} \mathbbmss{1}_{\{\bs Y(\bs f, \bs i)/a_m \in D_i\}}}{\frac{m_n^d}{n^w} \sum_{\bs i\in \mathcal{I}_n} \frac{1}{|{\mathcal{F}}(\bs 0)|} \sum_{\bs f \in {\mathcal{F}}(\bs 0)} \mathbbmss{1}_{\{\bs Y(\bs f, \bs i)/a_m \in D_{p+1}\}}}=\frac{\wh{\mu}_{\mathcal{B}(\bs 0,\gamma),m_n}(D_i)}{\wh{\mu}_{\mathcal{B}(\bs 0,\gamma),m_n}(D_{p+1})},
\end{align*}
using that ${\mathcal{F}}(\bs 0)=\mathcal{F}.$
Observe that
$$|{\mathcal{D}_n}(\bs h^{(i)})|=|{\mathcal{F}}(\bs h_{\mathcal{F}}^{(i)})|\prod\limits_{j=1}^w (n-|h_{\mathcal{I}}^{(i,j)}|) \sim |{\mathcal{F}}(\bs h_{\mathcal{F}}^{(i)})| n^w,\quad \nto.$$
Then the empirical  extremogram as defined in \eqref{EmpEst} for $\mu$-continuous Borel sets $A,B$ in $\overline{\mathbb{R}}\backslash\{0\}$ satisfies as $\nto$,
\begin{align*}
\widehat{\rho}_{AB,m_n}(\bs h^{(i)})
&= \frac{\frac{1}{|{\mathcal{D}_n}(\bs h^{(i)})|}\sum\limits_{\bs s \in {\mathcal{D}_n}(\bs h^{(i)})} \mathbbmss{1}_{\{X(\bs s)/a_m \in A, X(\bs s + \bs h^{(i)})/a_m \in B\}}}{\frac{1}{|\mathcal{D}_n|}\sum\limits_{\bs s \in \mathcal{D}_n} \one{\{ X(\bs s)/a_m \in A\}}} \\
& \sim \frac{\frac{1}{|{\mathcal{F}}(\bs h_{\mathcal{F}}^{(i)})|n^w}\sum_{\bs i \in {\mathcal{I}_n}(\bs h_{\mathcal{I}}^{(i)})}  \sum_{\bs f \in {\mathcal{F}}(\bs h_{\mathcal{F}}^{(i)})} \one{\{X(\bs f,\bs i)/a_m \in A, X(\bs f+\bs h_{\mathcal{F}}^{(i)},\bs i+\bs h_{\mathcal{I}}^{(i)})/a_m \in B\}}}{\frac{1}{|\mathcal{F}| n^w}\sum_{\bs i \in \mathcal{I}_n} \sum_{\bs f \in \mathcal{F}} \one{\{X(\bs f,\bs i)/a_m \in D_{p+1}\}}}\\
&\sim \frac{|\mathcal{F}|\sum_{\bs i \in \mathcal{I}_n} \sum_{\bs f \in {\mathcal{F}}(\bs h_{\mathcal{F}}^{(i)})} \one{\{\bs Y(\bs f,\bs i)/a_m \in D_i\}}}{|{\mathcal{F}}(\bs h_{\mathcal{F}}^{(i)})|\sum_{\bs i \in \mathcal{I}_n} \sum_{\bs f \in \mathcal{F}} \one{\{\bs Y(\bs f,\bs i)/a_m \in D_{p+1}\}}}=\widehat{R}_n(D_i,D_{p+1}),
\end{align*}
by definition \eqref{Di} of the sets $D_i$ for $i=1,\ldots,p$.
The remaining proof follows exactly as that of Theorem~4.2 in \cite{buhl3}, where in the last part the decomposition into a fixed and increasing grid has to be taken into account.
\halmos

\subsection{Proof of Theorem~\ref{Thm_asynbias}} \label{app_3}


{Throughout this proof, we suppress the sub index $m_n$ of $\wh{{\rho}}_{AB,m_n}$ and $\wh{{\rho}}_{AB,m_n}$ for notational ease.} The case, where $n^w/m_n^{3d} \to 0$ as $\nto$, is covered by Theorem~\ref{CLT_True}, so we assume that $n^w/m_n^{3d} \not\to 0$. 
Hence, by definition~\eqref{biascorrectedextremo} we have to consider
$$\wt{{\rho}}_{AB}(\bs h)=\wh{\rho}_{AB}(\bs h)-\frac{\underline{A}^{-1}}{2m_n^d}\Big[(\widehat{\rho}_{AB}(\bs h)-2\underline{A}/ \underline{B})(\widehat{\rho}_{AB}(\bs h)-1)\Big].$$
Observe that for $\bs h \in \mathcal{H}=\{\bs h^{(1)},\ldots,\bs h^{(p)}\}$, as $n \rightarrow \infty,$
\begin{align*}
\wt{{\rho}}_{AB}(\bs h)&-\rho_{AB}(\bs h)\\
=&\widehat{\rho}_{AB}(\bs h)-\rho_{AB,m_n}(\bs h)+\rho_{AB,m_n}(\bs h)-\frac{\underline{A}^{-1}}{2m_n^d}\Big[(\widehat{\rho}_{AB}(\bs h)-2\underline{A}/ \underline{B})(\widehat{\rho}_{AB}(\bs h)-1)\Big]-\rho_{AB}(\bs h) \\
=& (1+o(1))\Big\{\widehat{\rho}_{AB}(\bs h)-\rho_{AB,m_n}(\bs h) +\rho_{AB}(\bs h)+\frac{\underline{A}^{-1}}{2m_n^d}\Big[({\rho}_{AB}(\bs h)-2\underline{A}/ \underline{B})({\rho}_{AB}(\bs h)-1)\Big] \\
&-\frac{\underline{A}^{-1}}{2m_n^d}\Big[(\widehat{\rho}_{AB}(\bs h)-2\underline{A}/ \underline{B})(\widehat{\rho}_{AB}(\bs h)-1)\Big]-\rho_{AB}(\bs h)\Big\} 
\end{align*}
Since the conditions of Theorem~\ref{stasyn} are satisfied we have that 
$$\sqrt{\frac{n^w}{m_n^d}} \Big[\widehat{\rho}_{AB}(\bs h^{(i)}) -\rho_{AB,m_n}(\bs h^{(i)})\Big]_{i=1,\ldots,p} \std \mathcal{N}(\bs 0, \Pi)$$
and thus, by the continuous mapping theorem, it remains to show that for $\bs h \in \calh$,
\begin{align*}
\sqrt{\frac{n^w}{4m_n^{3d}}}\underline{A}^{-1}\Big[(\widehat{\rho}_{AB}(\bs h)-2\underline{A}/ \underline{B})(\widehat{\rho}_{AB}(\bs h)-1)-({\rho}_{AB}(\bs h)-2\underline{A}/ \underline{B})({\rho}_{AB}(\bs h)-1)\Big] \stp 0. 
\end{align*}
We rewrite the latter as 
\begin{align*}
\sqrt{\frac{n^w}{4m_n^{3d}}}\underline{A}^{-1}\Big[&(\widehat{\rho}_{AB}(\bs h)-2\underline{A}/ \underline{B})(\widehat{\rho}_{AB}(\bs h)-1)-(\rho_{AB,m_n}(\bs h)-2\underline{A}/ \underline{B})(\rho_{AB,m_n}(\bs h)-1)\\
&+(\rho_{AB,m_n}(\bs h)-2\underline{A}/ \underline{B})(\rho_{AB,m_n}(\bs h)-1)-({\rho}_{AB}(\bs h)-2\underline{A}/ \underline{B})({\rho}_{AB}(\bs h)-1)\Big]\\
=:A_1+A_2.
\end{align*}
As to $A_1$, we calculate
\begin{align*}
&\sqrt{\frac{n^w}{4m_n^{d}}}\frac{1}{2\rho_{AB}(\bs h)-(2\underline{A}/ \underline{B}+1)}\Big[(\widehat{\rho}_{AB}(\bs h)-2\underline{A}/ \underline{B})(\widehat{\rho}_{AB}(\bs h)-1)-(\rho_{AB,m_n}(\bs h)-2\underline{A}/ \underline{B})(\rho_{AB,m_n}(\bs h)-1)\Big] \\
=&\sqrt{\frac{n^w}{4m_n^{d}}}\frac{1}{2\rho_{AB}(\bs h)-(2\underline{A}/ \underline{B}+1)}\Big[\widehat{\rho}_{AB}(\bs h)^2-(2\underline{A}/ \underline{B}+1)\widehat{\rho}_{AB}(\bs h)-\Big(\rho^2_{AB,m_n}(\bs h)-(2\underline{A}/ \underline{B}+1)\rho_{AB,m_n}(\bs h)\Big)\Big] \\
=&\sqrt{\frac{n^w}{4m_n^{d}}}\frac{1}{2\rho_{AB}(\bs h)-(2\underline{A}/ \underline{B}+1)}\Big[(\widehat{\rho}_{AB}(\bs h)-\rho_{AB,m_n}(\bs h))(\widehat{\rho}_{AB}(\bs h)+\rho_{AB,m_n}(\bs h))\\
&-(2\underline{A}/ \underline{B}+1)(\widehat{\rho}_{AB}(\bs h)-\rho_{AB,m_n}(\bs h))\Big] \\
=&\sqrt{\frac{n^w}{4m_n^{d}}}(\widehat{\rho}_{AB}(\bs h)-\rho_{AB,m_n}(\bs h))\frac{\widehat{\rho}_{AB}(\bs h)+\rho_{AB,m_n}(\bs h)-(2\underline{A}/ \underline{B}+1)}{2\rho_{AB}(\bs h)-(2\underline{A}/ \underline{B}+1)}.
\end{align*}
By Theorem~\ref{stasyn}, the first term converges weakly to a normal distribution. Since $\wh\rho_{AB}(\bs h) \stp \rho_{AB}(\bs h)$ and $\rho_{AB,m_n}(\bs h) \to \rho_{AB}(\bs h)$ as $\nto$, the second term converges to $1$ in probability. Slutzky's theorem hence yields that $A_1 \stp 0$.
As to $A_2$, observe that
\begin{align*}
-\sqrt{\frac{4m_n^{3d}}{n^w}}\underline{A}A_2&=\rho_{AB}^2(\bs h)-\rho_{AB,m_n}^2(\bs h))+(2\underline{A}/ \underline{B}+1)(\rho_{AB,m_n}(\bs h)-\rho_{AB}(\bs h))\\
&=(1+o(1)) \Big\{\rho_{AB}^2(\bs h)-\Big[\rho_{AB}(\bs h)+\frac{\underline{A}^{-1}}{2m_n^d}\Big[({\rho}_{AB}(\bs h)-2\underline{A}/ \underline{B})({\rho}_{AB}(\bs h)-1)\Big]\Big]^2\\
&\hspace*{0.3cm}+(2\underline{A}/ \underline{B}+1)\Big[\rho_{AB}(\bs h) +\frac{\underline{A}^{-1}}{2m_n^d}\Big[({\rho}_{AB}(\bs h)-2\underline{A}/ \underline{B})({\rho}_{AB}(\bs h)-1)\Big] -\rho_{AB}(\bs h)\Big]\Big\}\\
&=(1+o(1)) \Big\{\rho_{AB}^2(\bs h)-\rho_{AB}^2(\bs h)-\frac{\underline{A}^{-1}\rho_{AB}(\bs h)}{m_n^d}\Big[({\rho}_{AB}(\bs h)-2\underline{A}/ \underline{B})({\rho}_{AB}(\bs h)-1)\Big]\\
&\hspace*{0.3cm}-\frac{\underline{A}^{-2}}{4m_n^{2d}}\Big[({\rho}_{AB}(\bs h)-2\underline{A}/ \underline{B})({\rho}_{AB}(\bs h)-1)\Big]^2\\
&\hspace*{0.3cm}+(2\underline{A}/ \underline{B}+1)\Big[\rho_{AB}(\bs h)+\frac{\underline{A}^{-1}}{2m_n^d}\Big[({\rho}_{AB}(\bs h)-2\underline{A}/ \underline{B})({\rho}_{AB}(\bs h)-1)\Big]-\rho_{AB}(\bs h)\Big]\Big\} \\
&=(1+o(1)) \Big\{\frac{\underline{A}^{-1}}{m_n^d}\Big[(\underline{A}/ \underline{B}+\frac{1}{2}-\rho_{AB}(\bs h))\Big[({\rho}_{AB}(\bs h)-2\underline{A}/ \underline{B})({\rho}_{AB}(\bs h)-1)\Big]\\
&\hspace*{0.3cm}-\frac{\underline{A}^{-1}}{4m_n^d}\Big[({\rho}_{AB}(\bs h)-2\underline{A}/ \underline{B})({\rho}_{AB}(\bs h)-1)\Big]^2\Big]\Big\}.
\end{align*}
Therefore $A_2$ converges to $0$ if and only if $\sqrt{n^w/m_n^{3d}} m_n^{-d}=\sqrt{n^w/m_n^{5d}}$ converges to $0$.
\halmos

\subsection{Proof of Theorem~\ref{GLSEcons}}\label{app_2}

We start with the proof of consistency and use a subsequence argument. Let $n'=n'(n)$ be some arbitrary subsequence of $n$. We show that there exists a further subsequence $n''=n''(n')$ such that 
$\widehat{\bs \theta}_{n'',V} \stas \bs \theta^{\star}$ as $n \rightarrow \infty$, which in turn implies \eqref{GLSEcons2}. \\
By (G1) we have for $i=1,\ldots,p$ that $\widehat{\rho}_{AB,m_n}(\bs h^{(i)}) \stp \rho_{AB,\bs \theta^{\star}}(\bs h^{(i)})$  as $n  \rightarrow \infty.$ 
Hence, there exists a subsequence $n''$ of $n'$ such that 
\begin{align} \label{asconssubseq}
\big[\widehat{\rho}_{AB,m_{n''}}(\bs h^{(i)})\big]_{i=1,\ldots,p} \stas \big[\rho_{AB,\bs \theta^{\star}}(\bs h^{(i)})\big]_{i=1,\ldots,p},
\end{align} 
as $n  \rightarrow \infty.$
For $\bs \theta \in \Theta$, we define the column vector and the quadratic forms
\beao
g(\bs \theta) &:=& \big[\rho_{AB,\bs \theta^{\star}}(\bs h^{(i)})-\rho_{AB, \bs \theta}(\bs h^{(i)}): i=1,\ldots,p\big]\trans_{i=1,\ldots,p},\\
Q(\bs \theta) &:=& g(\bs \theta)^TV(\bs \theta)g(\bs \theta)\quad\mbox{and}\quad
\widehat{Q}_n(\bs \theta) \, := \, \widehat{\bs g}_n(\bs \theta)\trans V(\bs \theta)\widehat{\bs g}_n(\bs \theta),
\eeao
where we recall from \eqref{g_hat} that
$\wh{\bs g}_n(\bs \theta)=\big[\widehat{\rho}_{AB,m_n}(\bs h^{(i)})-\rho_{AB, \bs \theta}(\bs h^{(i)})\big]\trans_{i=1,\ldots,p}.$
Assumptions (G1) and (G3) imply that $Q(\bs \theta)>0$ for $\bs \theta^{\star} \neq \bs \theta \in \Theta$ and that $Q(\bs \theta^{\star})=0$, so $\bs \theta^{\star}$ is the unique minimizer of $Q$. 
Smoothness and continuity of the functions $\rho_{AB,\bs \theta}(\bs h^{(i)})$ and $V(\bs \theta)$ (Assumptions (G4) and (G5) with $z_1=z_2=0$) and \eqref{asconssubseq} yield 
\begin{align}
\widehat{\Delta}_{n''}:=\sup\limits_{\bs \theta \in \Theta}\{|\widehat{Q}_{n''}(\bs \theta)-Q(\bs \theta)|\} \stas 0, \quad n \rightarrow \infty. \label{deltasubseq}
\end{align} 
Now assume that there exists some $\omega \in \Omega$ such that \eqref{deltasubseq} holds, but $\widehat{\bs \theta}_{n'',V}(\omega)\not\rightarrow \bs \theta^{\star}$. 
Then there exist $\epsilon>0$ and a subsequence $n'''=n'''(n'')$ such that for all $n \geq 1$,
$$\Vert{\widehat{\bs \theta}_{n''',V}(\omega)-\bs \theta^{\star} }\Vert > \epsilon.$$
Thus,
\begin{align*}
&\widehat{Q}_{n'''}(\widehat{\bs \theta}_{n''',V}(\omega))-\widehat{Q}_{n'''}(\bs \theta^{\star})\\&=-(Q(\widehat{\bs \theta}_{n''',V}(\omega))-\widehat{Q}_{n'''}(\widehat{\bs \theta}_{n''',V}(\omega)))+Q(\widehat{\bs \theta}_{n''',V}(\omega))-(\widehat{Q}_{n'''}(\bs \theta^{\star})-Q(\bs \theta^{\star}))-Q(\bs \theta^{\star}) \\
&\geq Q(\widehat{\bs \theta}_{n''',V}(\omega))-Q(\bs \theta^{\star})-2\widehat{\Delta}_{n'''}=Q(\widehat{\bs \theta}_{n''',V}(\omega))-2\widehat{\Delta}_{n'''} \\
&\geq \inf\{Q(\bs \theta): \Vert{\bs \theta-\bs \theta^{\star}}\Vert > \epsilon\}-2\widehat{\Delta}_{n'''}>0
\end{align*}
for all $n \geq n_0$ for some $n_0 \geq 1.$ But this contradicts the definition of $\widehat{\bs \theta}_{n''',V}$ as the minimizer of $\widehat{Q}_{n'''}(\bs \theta),$ $\bs \theta \in \Theta.$ Hence $\widehat{\bs \theta}_{n'',V} \stas \bs \theta^{\star}$ as $n \rightarrow \infty$ and this shows~\eqref{GLSEcons2}.


To prove the CLT~\eqref{asymptGLS}, we introduce the following notation:
\begin{enumerate}[$\bullet$]
\item We set $\rho_{AB,\bs \theta}^{(\ell)}(\bs h^{(i)}):=\frac{\partial}{\partial \theta_{\ell}}\rho_{AB,\bs \theta}(\bs h^{(i)})$ for $1 \leq i \leq p, 1 \leq \ell \leq k$ and 
\item 
$\bs \rho_{AB}^{(\ell)}(\bs \theta):=(\rho_{AB,\bs \theta}^{(\ell)}(\bs h^{(i)}): i=1,\ldots,p)\trans$ for $1 \leq \ell \leq k.$
The Jacobian matrix $\Rho_{AB}(\bs \theta)$~\eqref{Rho_matrix} can then be written as 
\begin{align*}
\Rho_{AB}(\bs \theta)=(-\bs \rho_{AB}^{(1)}(\bs \theta),\ldots, -\bs \rho_{AB}^{(k)}(\bs \theta)). 
\end{align*} 
\item We denote by $\bs e_{\ell} \in \mathbb{R}^k$ the $\ell$th unit vector.
\item For $1 \leq i,j \leq p$, let $v_{ij}(\bs \theta):=(V(\bs \theta))_{ij}$ be the entry in the $i$th row and $j$th column of $V(\bs \theta)$.
\item Set $v_{ij}^{(\ell)}(\bs \theta):= \frac{\partial}{\partial \theta_{\ell}} v_{ij}(\bs \theta)$ and $V^{(\ell)}(\bs \theta):=(v_{ij}^{(\ell)}(\bs \theta))_{1 \leq i,j \leq p}, \quad 1 \leq \ell \leq k.$
\end{enumerate}
As $\widehat{\bs \theta}_{n,V}$ minimizes $\widehat{\bs g}_n(\bs \theta)\trans V(\bs \theta) \widehat{\bs g}_n(\bs \theta)$ w.r.t. $\bs \theta$, we obtain for $1 \leq \ell \leq k$, 
\begin{align}
0&=\frac{\partial}{\partial \theta_{\ell}}(\widehat{\bs g}_n(\bs \theta)\trans V(\bs \theta)\widehat{\bs g}_n(\bs \theta))\Big|_{\bs \theta=\widehat{\bs \theta}_{n,V}} \nonumber\\
&=\wh{\bs g}_n(\wh{\bs \theta}_{n,V})\trans V^{(\ell)}(\wh{\bs \theta}_{n,V}) \wh{\bs g}_n(\wh{\bs \theta}_{n,V})-\bs \rho_{AB}^{(\ell)}(\wh{\bs \theta}_{n,V})\trans [V(\wh{\bs \theta}_{n,V})+V(\wh{\bs \theta}_{n,V})\trans] \wh{\bs g}_n(\wh{\bs \theta}_{n,V}).\label{asynGLSbefTay}
\end{align}
Now define the $p\times k$-matrix $\widehat{\Rho}_{AB,n}:=\int_0^1 \Rho_{AB}(u \bs \theta^{\star} + (1-u)\widehat{\bs \theta}_{n,V})\,\mathrm{d}u$, where the integral is taken componentwise.
 Assumptions~(G4) and (G5) with $z_1=z_2=1$ allow for a multivariate Taylor expansion of order 0 with integral remainder term of $\widehat{\bs g}_n(\widehat{\bs \theta}_{n,V})$ around the true parameter vector $\bs \theta^{\star}$, which yields 
\begin{align*}
\widehat{\bs g}_n(\widehat{\bs \theta}_{n,V})
&=\widehat{\bs g}_n(\bs \theta^{\star})+\widehat{\Rho}_{AB,n}\cdot(\widehat{\bs \theta}_{n,V}-\bs \theta^{\star}). 
\end{align*}
Plugging this into \eqref{asynGLSbefTay} and rearranging terms, we find
\begin{align}
&\Big(-\bs \rho_{AB}^{(\ell)}(\wh{\bs \theta}_{n,V})\trans [V(\widehat{\bs \theta}_{n,V})+V(\widehat{\bs \theta}_{n,V})\trans]\widehat{\Rho}_{AB,n} 
+(\widehat{\bs \theta}_{n,V}- \bs \theta^{\star})\trans \widehat{\Rho}_{AB,n}\trans V^{(\ell)}(\widehat{\bs \theta}_{n,V})\widehat{\Rho}_{AB,n}\Big)(\widehat{\bs \theta}_{n,V}-\bs \theta^{\star})\nonumber\\
=&\bs \rho_{AB}^{(\ell)}(\wh{\bs \theta}_{n,V})\trans [V(\wh{\bs \theta}_{n,V})+V(\wh{\bs \theta}_{n,V})\trans] \wh{\bs g}_n(\bs \theta^\star)-\wh{\bs g}_n(\bs \theta^\star)\trans V^{(\ell)}(\wh{\bs \theta}_{n,V}) \wh{\bs g}_n(\bs\theta^\star)\nonumber\\
&-\wh{\bs g}_n(\bs\theta^\star)\trans [V^{(\ell)}(\wh{\bs \theta}_{n,V})+V^{(\ell)}(\wh{\bs \theta}_{n,V})\trans]\widehat{\Rho}_{AB,n}(\widehat{\bs \theta}_{n,V}-\bs \theta^{\star}) \label{SOE_Thm43}
\end{align}
for $1 \leq \ell \leq k.$
Defining $\widehat{R}_{n,V}$ as the $k\times k$-matrix whose $\ell$th row is given by
$$(\widehat{\bs \theta}_{n,V}- \bs \theta^{\star})\trans \widehat{\Rho}_{AB,n}\trans V^{(\ell)}(\widehat{\bs \theta}_{n,V})\widehat{\Rho}_{AB,n}, \quad 1 \leq \ell \leq k,$$
the system of equations \eqref{SOE_Thm43} can be written in compact matrix form as
\begin{align}
&(\Rho_{AB}(\widehat{\bs \theta}_{n,V})\trans[V(\widehat{\bs \theta}_{n,V})+V(\widehat{\bs \theta}_{n,V})\trans]\widehat{\Rho}_{AB,n}+\widehat{R}_{n,V})(\widehat{\bs \theta}_{n,V}-\bs \theta^{\star}) \nonumber\\
=&-\Rho_{AB}(\widehat{\bs \theta}_{n,V})\trans[V(\widehat{\bs \theta}_{n,V})+V(\widehat{\bs \theta}_{n,V})\trans]\widehat{\bs g}_n(\bs \theta^{\star})-\sum\limits_{\ell=1}^k\widehat{\bs g}_n(\bs \theta^{\star})\trans V^{(\ell)}(\widehat{\bs \theta}_{n,V})\widehat{\bs g}_n(\bs \theta^{\star})\bs e_{\ell}\nonumber\\
&-\sum\limits_{\ell=1}^k\widehat{\bs g}_n(\bs \theta^{\star})\trans [V^{(\ell)}(\widehat{\bs \theta}_{n,V})+V^{(\ell)}(\widehat{\bs \theta}_{n,V})\trans]\widehat{\Rho}_{AB,n}(\widehat{\bs \theta}_{n,V}-\bs \theta^{\star})\bs e_{\ell}. \label{SOE2_Thm43}
\end{align}
Hence, multiplying \eqref{SOE2_Thm43} by $\sqrt{n^w/m_n^d}$ and rearranging terms, we have,
\begin{align*}
&\sqrt{\frac{n^w}{m_n^d}}(\widehat{\bs \theta}_{n,V}-\bs \theta^{\star})\\
=&-\{\Rho_{AB}(\widehat{\bs \theta}_{n,V})\trans[V(\widehat{\bs \theta}_{n,V})+V(\widehat{\bs \theta}_{n,V})\trans]\widehat{\Rho}_{AB,n}+\widehat{R}_{n,V}\}^{-1}  \\
& \quad \times \Rho_{AB}(\widehat{\bs \theta}_{n,V})\trans[V(\widehat{\bs \theta}_{n,V})+V(\widehat{\bs \theta}_{n,V})\trans]\sqrt{\frac{n^w}{m_n^d}}\widehat{\bs g}_n(\bs \theta^{\star})\\
&-\{\Rho_{AB}(\widehat{\bs \theta}_{n,V})\trans[V(\widehat{\bs \theta}_{n,V})+V(\widehat{\bs \theta}_{n,V})\trans]\widehat{\Rho}_{AB,n}+\widehat{R}_{n,V}\}^{-1}\sum\limits_{\ell=1}^k\sqrt{\frac{n^w}{m_n^d}}\widehat{\bs g}_n(\bs \theta^{\star})\trans V^{(\ell)}(\widehat{\bs \theta}_{n,V})\widehat{\bs g}_n(\bs \theta^{\star})\bs e_{\ell} \\
&-\{\Rho_{AB}(\widehat{\bs \theta}_{n,V})\trans[V(\widehat{\bs \theta}_{n,V})+V(\widehat{\bs \theta}_{n,V})\trans]\widehat{\Rho}_{AB,n}+\widehat{R}_{n,V}\}^{-1}\\
& \quad \times \sum\limits_{\ell=1}^k\sqrt{\frac{n^w}{m_n^d}}\widehat{\bs g}_n(\bs \theta^{\star})\trans [V^{(\ell)}(\widehat{\bs \theta}_{n,V})+V^{(\ell)}(\widehat{\bs \theta}_{n,V})\trans]\widehat{\Rho}_{AB,n}(\widehat{\bs \theta}_{n,V}-\bs \theta^{\star})\bs e_{\ell}\\
&=:-A-B-C.
\end{align*}
Observe that the smoothness conditions (G4) and (G5) and the rank condition (G6) ensure invertibility of the terms in curly brackets and boundedness of its inverse. 
For the remainder of the proof, we can hence use Slutsky's theorem; to this end note that, as $n \rightarrow \infty$:
\begin{itemize}
\item By conditions (G4) and (G5ii) with $z_1=z_2=1$, the matrices $V(\bs \theta)$ and $\Rho_{AB}(\bs \theta)$ are continuous in $\bs \theta$, hence $V(\widehat{\bs \theta}_{n,V}) \stp V(\bs \theta^{\star})$ and $\Rho_{AB}(\widehat{\bs \theta}_{n,V}) \stp \Rho_{AB}(\bs \theta^{\star})$ by continuous mapping.
\item Using \eqref{GLSEcons2}, we find that $(\widehat{\bs \theta}_{n,V}-\bs \theta^{\star}) \stp \bs 0$, $\widehat{R}_{n,V} \stp (\bs 0, \ldots, \bs 0)$ and $\widehat{\Rho}_{AB,n} \stp \Rho_{AB}(\bs \theta^{\star})$.
\item The previous bullet point directly implies that $C \stp \bs 0.$ 
\item As to $A$, condition (G2) directly yields $\sqrt{\frac{n^w}{m_n^d}}\widehat{\bs g}_n(\bs \theta^{\star})
\std \mathcal{N}(\bs 0, \Pi)$.
\item Furthermore, $\widehat{\bs g}_n(\bs \theta^{\star}) \stp \bs 0$ by (G1) and therefore $B \stp \bs 0$.
\end{itemize}
Finally, summarising those results, with $B(\bs \theta^{\star})=\big(\Rho_{AB}(\bs \theta^{\star})\trans[V(\bs \theta^{\star})+V(\bs \theta^{\star})\trans]\Rho_{AB}(\bs \theta^{\star})\big)^{-1}$, we obtain \eqref{asymptGLS}.
\halmos
\end{document}